\documentclass[11pt,a4paper,fleqn]{article}

\usepackage{amsthm}
\usepackage{amssymb}
\usepackage{amstext}
\usepackage{amsmath}
\usepackage{url}
\usepackage[colorlinks]{hyperref}
\hypersetup{
linkcolor=blue,
citecolor=blue,
}
\usepackage{pdfsync}
\usepackage[latin1]{inputenc}
\usepackage[margin = 40pt,font=small,labelfont=bf]{caption}
\usepackage[T1]{fontenc}
\usepackage{color}
\usepackage{pdfpages}
\usepackage{epstopdf}
\usepackage{graphicx}
\usepackage{epsfig}
\usepackage{subfigure}
\usepackage{epstopdf}
\usepackage{todonotes}
\usepackage{resizegather}

\definecolor{chocolate}{rgb}{.82,.41,.12}

\graphicspath{{images/},{prebuiltimages/}}

\numberwithin{equation}{section}

\theoremstyle{plain}
\newtheorem{prop}{Proposition}[section]

\newtheorem{lem}{Lemma}[section]
\newtheorem{hyp}{Assumption}[section]

\theoremstyle{definition}
\newtheorem{defin}{Definition}[section]

\theoremstyle{remark}

\def\var{\mathop{\rm Var}\nolimits}%
\newcommand{\E}{{\mathbb E}}
\newcommand{\R}{{\mathbb R}}
\newcommand{\C}{{\mathbb C}}

\newcommand{\Z}{{\mathbb Z}}

\def\argmin{\mathop{\rm arg \; min}\limits}%
\newcommand{\uargmin}[1]{\underset{#1}{\argmin}\;}
\newcommand{\norm}[1]{\| #1 \|}

\def\NMSE{\mathop{\rm NMSE}}%
\def\MSE{{\rm MSE}}%
\def\KL{{\rm KL}}%
\def\SURE{\mathop{\rm SURE}}%
\def\div{\mathop{\rm div}}%
\def\GSURE{\mathop{\rm GSURE}}%
\def\PURE{\mathop{\rm PURE}}%
\def\SUKLS{{\rm SUKLS}}%
\def\MKLS{{\rm MKLS}}%
\def\MKLA{{\rm MKLA}}%
\def\PUKLA{{\rm PUKLA}}%
\def\KLS{{\rm KLS}}%
\def\KLA{{\rm KLA}}%
\def\SE{{\rm SE}}%
\def\DOF{{\rm DOF}}%
\def\cov{{\rm Cov}}%

\newcommand{\thefont}[2]{\fontsize{#1}{#2}\fontshape{n}\selectfont}
\newcommand{\1}{\rlap{\thefont{10pt}{12pt}1}\kern.16em\rlap{\thefont{11pt}{13.2pt}1}\kern.4em}

\newcommand{\YY}{{\mathcal Y}}
\newcommand{\XX}{{\mathcal X}}

\newcommand{\bX}{\mbox{$\boldsymbol{X}$}}
\newcommand{\bY}{\mbox{$\boldsymbol{Y}$}}

\newcommand{\bV}{\boldsymbol{V}}
\newcommand{\bU}{\boldsymbol{U}}
\newcommand{\bW}{\boldsymbol{W}}
\newcommand{\bv}{\boldsymbol{v}}
\newcommand{\bu}{\boldsymbol{u}}
\newcommand{\be}{\boldsymbol{e}}
\newcommand{\btheta}{\boldsymbol{\theta}}

\newcommand{\bdelta}{\boldsymbol{\delta}}

\title{Generalized SURE for optimal shrinkage of singular values in low-rank matrix denoising}

\author{ J\'{e}r\'{e}mie Bigot, Charles Deledalle \& Delphine F\'eral  \\
\\  Institut de Math\'ematiques de Bordeaux et CNRS  (UMR 5251)   \\ Universit\'e de Bordeaux }

\date{\today}

\begin{document}
\sloppy

\maketitle

\thispagestyle{empty}

\begin{abstract}
We consider the problem of estimating a low-rank signal matrix from noisy measurements under the assumption that the distribution of the data matrix belongs to an exponential family. In this setting, we derive generalized Stein's unbiased risk estimation (SURE) formulas that hold for any  spectral estimators which shrink or threshold the singular values of the data matrix. This leads to new data-driven spectral estimators, whose optimality is discussed using tools from random matrix theory and through numerical experiments. Under the spiked population model and in the asymptotic setting where the dimensions of the data matrix are let going to infinity, some theoretical properties of our approach are compared to recent results on asymptotically optimal shrinking rules for Gaussian noise. It also leads to new procedures for singular values shrinkage in finite-dimensional matrix denoising for Gamma-distributed and Poisson-distributed measurements.
\end{abstract}




\section{Introduction}

\subsection{Low rank matrix denoising in an exponential family}

In various applications, it is of interest to estimate a signal matrix from noisy data. Typical examples include the case of data that are produced in a matrix form, while others are concerned with observations from multiple samples that can be organized in a matrix form. In such setting, a typical inference problem involves the estimation of an unknown (non-random) signal matrix $\bX \in \R^{n \times m}$ from a noisy data matrix $\bY$ satisfying the model:
\begin{equation}
\bY = \bX + \bW, \label{eq:deformodel}
\end{equation}
where $\bW$ is an $n \times m$ noise matrix with real entries $\bW_{ij}$ assumed to be independent random variables with $\E[\bW_{ij}] = 0$ and $\var(\bW_{ij}) = \tau_{ij}^{2}$ for  $1 \leq i \leq n$ and $1 \leq j \leq m$.  In this paper, we focus on the situation where the signal matrix $\bX$   is assumed to have a low rank structure, and we consider the general setting where the distribution of $\bY$ belongs to a continuous exponential family parametrized by the entries of the matrix $\bX = \E[\bY]$. For discrete observations (count data), we also consider the specific case of Poisson noise.

The  low rank assumption on $\bX$  is often met in practice when there exists a significant correlation between the columns of $\bX$. This can be the case when the columns of $\bX$ represent 2D images at different wavelength of hyperspectral data, since images at nearby wavelengths are strongly correlated \cite{MR3105401}. Further applications, where low-rank modeling of $\bX$ is relevant, can be found in genomics \cite{WallDB01,svd-gene}, NMR spectroscopy \cite{NguyenPDL11}, collaborative filtering \cite{MR2565240} or medical imaging \cite{Bydder2005noise,LamBHSL12}, among many others.

Low-rank matrix estimation is classically done in the setting where the additive noise is Gaussian with homoscedastic variance. 
The more general case of observations sampled from an exponential family is less developed, but there exists an increasing research interest in the study of low rank matrix recovery beyond the Gaussian case. Examples of low-rank matrix recovering from Poisson distributed observations can be found in applications with count data such as network traffic analysis \cite{Bazerque2013} or call center data \cite{Shen2005}. A theory  for  low-rank  matrix  recovery and completion in the case of Poisson observations has also been  recently proposed in \cite{Cao2016}. 
 Matrix completion under a low rank assumption with additive errors having a sub-exponential distribution and belonging to an exponential family has also been  considered in \cite{Lafond15}. The recent work  \cite{MAL2016} proposes a novel framework to approximate, by a low rank matrix, a tabular data set made of numerical, Boolean, categorical or ordinal observations.

\subsection{The class of spectral estimators}

A standard approach to estimate a low rank matrix relies on the singular value decomposition (SVD) of the data matrix
\begin{equation}
\bY =  \sum_{k = 1}^{\min(n,m)} \tilde{\sigma}_{k} \tilde{\bu}_{k} \tilde{\bv}_{k}^{t},
\end{equation}
where  $\tilde{\sigma}_{1} \geq \tilde{\sigma}_{2} \geq \ldots \geq \tilde{\sigma}_{\min(n,m)} \geq 0$  denote its  singular values, and $\tilde{\bu}_{k}, \tilde{\bv}_{k}$ denote the associated singular vectors. In this paper, we propose to consider the class of spectral estimators $\hat{\bX}^{f} = f ( \bY )$, where $f : \R^{n \times m} \to \R^{n \times m}$  is a (possibly data-dependent) mapping that acts on the singular values of the data matrix $\bY$ while leaving its singular vectors unchanged. More precisely, these estimators take the form
\begin{equation} \label{eq:hatXf}
\hat{\bX}^{f} = f ( \bY ) =  \sum_{k = 1}^{\min(n,m)} f_{k}(\bY) \tilde{\bu}_{k} \tilde{\bv}_{k}^{t},
\end{equation}
where, for each $1 \leq k \leq \min(n,m)$, $f_{k}(\bY)$ are real positive values that may depend only on $\tilde{\sigma}_{k}$ (hence we write $f_{k}(\tilde{\sigma}_{k})$)
or on the whole matrix $\bY$.

\subsection{Investigated spectral estimators}

Typical examples of spectral estimators include the
classical principal component analysis (PCA) applied to matrix denoising
defined, for some $1 \leq r \leq \min(n,m)$, as
\begin{equation}
  \hat{\bX}^{r} =  \sum_{k = 1}^{r}  \hat{\sigma}_{k}  \tilde{\bu}_{k} \tilde{\bv}_{k}^{t}
  \quad \text{with} \quad
  \hat{\sigma}_{k} = f_{k}(\tilde{\sigma}_{k}) = \tilde\sigma_k
\end{equation}
for all $1 \leq k \leq r$ and where it is implicitely understood that $f_{k}(\tilde{\sigma}_{k}) = 0$ for $k \geq r+1$.
Another typical spectral estimator in matrix denoising with Gaussian measurements is the soft-thresholding \cite{MR3105401} which corresponds to the choice
\begin{equation}
  \hat{\bX}_{\mathrm{soft}} =  \sum_{k = 1}^{\min(m, n)}  \hat{\sigma}_k \tilde{\bu}_{k} \tilde{\bv}_{k}^{t}
  \quad\text{with}\quad
  \hat{\sigma}_k = f_{k}(\bY) = \left(1  - \frac{\lambda(\bY)}{\tilde{\sigma}_{k}} \right)_{+} \tilde{\sigma}_{k},
\end{equation}
for all $1 \leq k \leq \min(n,m)$ and
where $\lambda(\bY) > 0$ is a possibly data-dependent threshold parameter, and $(x)_{+} = \max(x,0)$ for any $x \in \R$.
Finaly, we will consider a more general class of shrinkage estimators, encompassing the
PCA and the soft-thresholding, that perform
\begin{equation}
  \hat{\bX}_{w} =  \sum_{k = 1}^{\min(m, n)} \hat{\sigma}_k \tilde{\bu}_{k} \tilde{\bv}_{k}^{t}
  \quad\text{with}\quad
  \hat{\sigma}_k = f_{k}(\bY) =
  w_k(\bY) \tilde{\sigma}_{k}
  ,
\end{equation}
where $w_k(\bY) \in [0, 1]$ is a possibly data-dependent shrinking weight.

\subsection{Main contributions}

Under the assumption that the distribution of $\bY$ belongs to an exponential family, the goal of this paper is to derive data-driven choices for the weights $w_{k}(\bY)$ in \eqref{eq:hatXf}. 
We construct estimators via a two-step procedure. First, an active set 
of non-zero singular values is defined. Then, in a second step, weights $w_{k}(\bY)$ associated with non-zero singular values are optimized,
and shown to reach desired asymptotical properties in the Gaussian spiked population model.
The main contributions of the paper are then the following ones.

\subsubsection{An AIC inspired criterion for rank and singular values locations estimation}

When no {\it a priori} is available on the rank of the signal matrix $\bX$,
optimizing for the weights $w_k$,
for all $1 \leq k \leq \min(m, n)$, can lead to
estimators with large variance ({\it i.e.}, overfitting the noise). 
We propose an automatic rule to prelocalize the subset of non-zero singular values.
An active set $s^{\star} \subseteq \mathcal{I}  = \{1,2,\ldots,\min(n, m)\}$ of singular values is defined as the minimizer of a penalized log-likelihood criterion that is inspired by the Akaike information criterion (AIC)
\begin{equation} \label{eq:AICintro}
s^\ast  \in  \uargmin{ s \subseteq \mathcal{I}  }
  -2 \log q(\bY; \tilde{\bX}^s) + 2 |s| p_{n,m}  \quad
  \text{with}
  \quad p_{n,m} = \frac{1}{2} \left(\sqrt{m} + \sqrt{n}\right)^2,
\end{equation}
where $\tilde{\bX}^s = \sum_{k \in s} \tilde{\sigma}_k \tilde{\bu}_k \tilde{\bv}^t_k$, $|s|$ is the cardinal of $s$, and $q(\bY; \tilde{\bX}^s)$  is the likelihood of the data in a given exponential family with estimated parameter $\tilde{\bX}^s$.
For the  case of Gaussian measurements with homoscedastic variance $\tau^{2}$,
one has that $q(\bY; \tilde{\bX}^s) = \| \bY - \tilde{\bX}^{s} \|^{2}_{F}/2\tau^2 $,
where $\| \cdot \|_{F}$ denotes the Frobenius norm of a matrix,
and we show that the active set of singular values boils down to
\begin{equation}
s^\star = \{k \; ; \; \tilde{\sigma}_k > c_+^{n,m}\} \label{eq:sastGauss},
\end{equation}
where $c_+^{n,m}  = \tau(\sqrt{m} + \sqrt{n})$. 
For Gamma and Poisson measurements, we resort to a greedy optimization procedure
described in Section \ref{sec:bulk_exp_fam}.

Once the active set has been determined,
the subsequent shrinkage estimator is obtained by optimizing
only for the weights within this subset while setting the other ones to zero.

\subsubsection{Novel data-driven shrinkage rules minimizing SURE-like formulas}

We use the principle of Stein's unbiased risk estimation (SURE) \cite{MR630098} to derive unbiased estimation formulas for the mean squared error (MSE) risk and mean Kullback-Leibler (MKL) risks of spectral estimators. Minimizing such SURE-like formulas over an appropriate class of spectral estimators is shown to lead to  novel data-driven shrinkage rules of  the singular values of the matrix $\bY$. In particular, our approach leads to novel spectral estimators in situations where the variances $\tau_{ij}^{2}$ of the entries $\bW_{ij}$ of the noise matrix are not necessarily equal, and may depend on the signal matrix $\bX$.  

As an illustrative example, let us consider spectral estimators of the form
\begin{equation} \label{eq:hatXf1}
\hat{\bX}_{w}^{1} = f ( \bY ) =  w_{1}(\bY) \tilde{\sigma}_{1} \tilde{\bu}_{1} \tilde{\bv}_{1}^{t},
\end{equation}
which only act on the first singular value $\tilde{\sigma}_{1}$ of the data while setting all the other ones to zero. In this paper, examples of data-driven choices for the weight  $w_{1}(\bY)$  are  the following ones:
\begin{description}
\item[$\bullet$] for Gaussian measurements with $n \leq m$ and known homoscedastic variance $\tau^{2}$ 
  \begin{flalign}
  w_{1}(\bY) = \left( 1 - \frac{\tau^2}{ \tilde{\sigma}_{1}^{2}}  \left(  1+ |m - n| +  2     \sum_{\ell = 2}^{n} \frac{ \tilde{\sigma}_{1}^{2}}{\tilde{\sigma}_{1}^{2} - \tilde{\sigma}_{\ell}^{2} } \right) \right)_{+} \1_{\left\{ \tilde{\sigma}_{1} > c_+^{n,m} \right\}}
  , && \label{eq:optgauss}
\end{flalign}
\item[$\bullet$] for Gamma measurements with $\tau_{ij}^2 = \bX_{ij}^2/L$ and $L > 2$ (see Section \ref{sec:examples} for a precise definition),    
  \begin{gather}
    \hspace{-3em}
  w_{1}(\bY) = \min \left[ 1,  \left(    \frac{L - 1}{ L m n } \sum_{i=1}^{n} \sum_{j=1}^{m}   \frac{\hat \bX^{1}_{ij}}{\bY_{ij} } + \frac{1}{ L m n} \left(1 + |m - n| + 2   \sum_{\ell =2  }^{\min(n,m)} \frac{\tilde{\sigma}_{1}^{2}}{\tilde{\sigma}_{1}^{2} - \tilde{\sigma}_{\ell}^{2} } \right) \right)^{-1} \right]
  \1_{\left\{ 1 \in s^\ast \right\} }
  ,
  \hspace{0.5em}
  \label{eq:optgamma}
\end{gather}
\item[$\bullet$] for Poisson measurements  with $\tau_{ij}^2 = \bX_{ij}$ (see Section \ref{sec:examples} for a precise definition) 
\begin{flalign}
  w_{1}(\bY) =  \min \left[ 1,  \frac{ \sum_{i=1}^{n} \sum_{j=1}^{m}  \bY_{ij} }{  \sum_{i=1}^{n} \sum_{j=1}^{m} \hat \bX^{1}_{ij} } \right] \1_{\left\{ 1 \in s^\ast \right\} }. && \label{eq:optpoisson}
\end{flalign}
\end{description}


Beyond the case of rank one, closed-form solutions for the weights cannot be obtained,
except for
the case of Gaussian measurements with homoscedastic variance $\tau^{2}$. 
In this latter case, the rule for $w_{1}(\bY)$ in \eqref{eq:optgauss} generalizes to other
eigenvalues $w_{k}(\bY)$ as
\begin{equation}
  w_{k}(\bY) = \left( 1 - \frac{\tau^2}{ \tilde{\sigma}_{k}^{2}}  \left(  1+ |m - n| +  2    \sum_{\ell =1 ; \ell \neq k}^{\min(n,m)} \frac{ \tilde{\sigma}_{k}^{2}}{\tilde{\sigma}_{k}^{2} - \tilde{\sigma}_{\ell}^{2} } \right) \right)_{+}
  \1_{\left\{ \tilde{\sigma}_{k} > c_+^{n,m} \right\}}. \label{eq:optGauss}
\end{equation}
For Gamma or Poisson distributed measurements,
we propose fast algorithms to get numerical approximations of the weights $w_{k}(\bY)$
(see Section \ref{sec:rankmoretwo} for more details).

\subsubsection{Asymptotic properties in the Gaussian spiked population model} \label{sec:defspiked}

Another contribution of the paper is to discuss the optimality of the shrinking weights \eqref{eq:optGauss} for Gaussian noise in the asymptotic setting where the dimensions of the matrix $\bY$ are let going to infinity. These theoretical results are obtained for the so-called {\it spiked population model} that has been introduced in the literature on random matrix theory and high-dimensional covariance matrix estimation (see e.g.\ \cite{Baik20061382,Benaych-GeorgesN12,MR2322123,MR3054091}). All the theoretical and asymptotic results of the paper (other than derivation of proposed estimators) assume this model.

\begin{defin} \label{def:spiked}
{\it
The Gaussian spiked population model corresponds to the following setting:
\begin{description}
\item[$\bullet$]   the $\bW_{ij}$  in \eqref{eq:deformodel} are iid Gaussian random variables with zero mean and  variance $\tau^2 = 1/m$,
\item[$\bullet$]    the $\bX_{ij}$'s  in \eqref{eq:deformodel} are the entries  of an unknown $n \times m$ matrix $\bX$ that has a low rank structure, meaning that it admits the  SVD
$
\bX = \sum_{k = 1}^{r^{\ast}} \sigma_{k} \bu_{k} \bv_{k}^{t},
$
where $\bu_{k}$ and $\bv_{k}$ are the left and right singular vectors associated to the singular value $\sigma_{k} > 0$, for each $1 \leq k \leq r^{\ast}$, with $\sigma_1 > \sigma_2 > \ldots > \sigma_{r^{\ast}}$,
\item[$\bullet$] the rank $r^{\ast}$ of the matrix  $\bX$ is assumed to be fixed,
\item[$\bullet$] the dimensions of the data matrix $\bY = \bX + \bW$ are let going to infinity in the asymptotic framework where the sequence $m = m_{n} \geq n$ is such that $\lim_{n \to + \infty} \frac{n}{m} = c$ with $0 < c \leq 1$.
\end{description}
}
\end{defin}

In the Gaussian spiked population model, the asymptotic locations of the empirical singular values $\tilde{\sigma}_{1} \geq \ldots \geq \tilde{\sigma}_{\min(n,m)}$ are well understood in the  random matrix theory (further details are given in Section \ref{sec:asymptotic_sv}). Note that the setting where the rank $r^{\ast}$ is not held fixed but allowed to grow with $\min(n,m)$ is very different, see e.g.\ \cite{ledoit2012} and references therein.

Under the Gaussian spiked population model, our contributions are then as follows:

\begin{description}
\item[$\bullet$] we prove the convergence of the SURE formula when the dimensions of $\bY$ tend to infinity,
\item[$\bullet$] it is shown that minimizing the asymptotic value of SURE leads to the same estimator as the limiting value of the estimator obtained by minimizing the SURE,
\item[$\bullet$] this model allows to show that the novel data-driven spectral estimators derived in this paper are asymptotically  connected to existing  optimal shrinkage rules \cite{MR3054091,GavishDonoho,MR3200641} for low-rank matrix denoising,
\item[$\bullet$] in this setting, we are also able to connect the choice of the penalty function $2 |s| p_{n,m}$ in \eqref{eq:AICintro} with Stein's notion of degrees of freedom (see e.g.\ \cite{Efron04}) for spectral estimators.
\end{description}

\subsubsection{Numerical experiments and publicly available source code}

As the theoretical properties of our estimators are studied in an asymptotic setting, we report the results of various numerical experiments to analyze the performances of the proposed estimators for finite-dimensional matrices. These experiments allow the comparison with existing shrinkage rules for Gaussian-distributed measurements and they are also used to shed some lights on the finite sample properties of the method for Gamma-distributed or Poisson-distributed measurements. We also exhibit the settings where the signal matrix $\bX$ is either easy or more difficult to recover. 
From these experiments, the main findings are the following ones:

\begin{description}
\item[$\bullet$] the use of an appropriate active set $s$ of singular values is an essential step for the quality of shrinkage estimators whose weights are data-driven by SURE-like estimators; taking $s = \left\{1 ,\ldots, \min(n,m) \right\}$ leads to poor results while  the choice of $s = s^{\ast}$ minimizing the AIC criterion \eqref{eq:AICintro} appears to yield the best performances,
\item[$\bullet$] for Gaussian noise, the performances of our approach are similar to those obtained by the asymptotically optimal spectral estimator proposed in \cite{GavishDonoho} when the true rank $r^{\ast}$ of the signal matrix $\bX$ is sufficiently small, but for large to moderate values of the signal-to-noise ratio our approach may perform better than existing methods in the literature,
\item[$\bullet$] for Gamma or Poisson distributed measurements, the spectral estimators proposed in this paper 
  give better results than estimators based on PCA (restricted to the active set $s^{\ast}$) or soft-thresholding of singular values.
\end{description}

Beyond the case of Gaussian noise, the implementation of the estimators is not straightforward, and we thus provide publicly available source code at
\begin{center}
  \url{https://www.math.u-bordeaux.fr/~cdeledal/gsure_low_rank}
\end{center}
to reproduce the figures and the numerical experiments of this paper.

\subsection{Related results in the literature}

Early work on singular value thresholding began with the work in \cite{EYM} on the best approximation of fixed rank to the data matrix $\bY$. Spectral estimators with different amounts of shrinkage for each singular value of the data matrix have then been proposed in \cite{efron1972,efron1976}. In the case of Gaussian measurements with homoscedastic variance, the problem of estimating $\bX$ under a low-rank assumption has recently received a lot of attention in the literature on high-dimensional statistics, see e.g.\ \cite{MR3105401,donoho2014,JosseSardy,MR3054091}. Recent works \cite{GavishDonoho,MR3200641} also consider the more general setting where the distribution of the additive noise matrix $\bW$ is orthogonally invariant, and such that its entries are  iid random variables with zero mean and finite fourth moment. In all these papers, the authors have focused on spectral estimators which shrink or threshold  the singular values of   $\bY$, while its singular vectors are left unchanged. In this setting, the main issue is  to derive  optimal shrinkage rules  that depends on the class of spectral estimators that is considered, on the loss function used to measure the risk of an estimator of $\bX$, and on appropriate assumptions for the distribution of the additive noise matrix $\bW$.

\subsection{Organization of the paper}

Section \ref{sec:exp} is devoted to the analysis of a data matrix whose entries are distributed according to a continuous exponential family. SURE-like formula are first given  for the mean squared error  risk, and then for the Kullback-Leibler risk. As an example of discrete exponential family, we also derive such risk estimators for Poisson distributed measurements. The computation of data-driven shrinkage rules is then discussed for Gaussian, Gamma and Poisson noises. In Section \ref{sec:gauss}, we restrict our attention to the Gaussian spiked population model in order to derive asymptotic properties of our approach. We study the asymptotic behavior of the SURE formula proposed in \cite{MR3105401,donoho2014}  for  spectral estimators using tools from RMT. This result allows to make a connection between data-driven spectral estimators minimizing the SURE for Gaussian noise, and the asymptotically optimal shrinkage rules proposed in  \cite{MR3054091,MR3200641} and \cite{GavishDonoho}. In Section \ref{sec:bulk_exp_fam},  we study  the penalized log-likelihood criterion \eqref{eq:AICintro} used to select an active set of singular values.  Its connection to the degrees of freedom of spectral estimators and rank estimation in matrix denoising is discussed.  Various numerical experiments are finally proposed in Section \ref{sec:num} to illustrate the usefulness of the approach developed in this paper for low-rank denoising and to compare its performances with existing methods. The proofs of the main results of the paper  are gathered in a technical Appendix \ref{sec:proofs}, and numerical implementation details are described in Appendix \ref{sec:algo}.

\section{SURE-like formulas in exponential families} \label{sec:exp}

For an introduction to exponential families, we refer to \cite{Brown}. The idea of unbiased risk estimation in exponential families dates back to \cite{hudson1978}. More recently, generalized SURE formulas have been proposed for the estimation of the MSE risk, 
for denoising under various continuous and discrete distributions in \cite{raphan2007learning},
and for inverse problems whithin the continuous exponential families in \cite{Eldar09}.
In \cite{Deledalle}, SURE-like formula are derived for the estimation of the Kullback-Leibler risk that applies to both
continuous and discrete exponential families. In what follows, we borrow  some ideas and results from these works. We first treat the case of continuous exponential families, and then we focus on Poisson data in the discrete case.

\subsection{Data sampled from a continuous exponential family} \label{sec:examples}

We recall that  $\bY$ is an $n \times m$ matrix with independent and real entries $\bY_{ij}$. For each $1 \leq i \leq n$ and $1 \leq j \leq m$, we assume that the random variable $\bY_{ij}$ is sampled from   a continuous exponential family, in the sense that each $\bY_{ij}$ admits a probability density function (pdf) $q(y ; \bX_{ij})$ with respect to the Lebesgue measure $\mathrm{d}y$ on the real line $\YY = \R$. The pdf  $q(y ; \bX_{ij})$ of  $\bY_{ij}$ can thus be written in the general form:
\begin{equation}
q(y ; \bX_{ij}) =  h(y) \exp \left(\eta(\bX_{ij}) y - A( \eta(\bX_{ij})) \right),\; y \in \YY,  \label{eq:expfam}
\end{equation}
where $\eta$ (the link function) is a one-to-one and smooth function,  $A$ (the log-partition function) is a twice differentiable mapping,  $h$ is a known function, and $\bX_{ij}$  is an unknown parameter of interest belonging to some open subset $\XX$ of $\R$.  Throughout the paper, we will suppose that the following assumption holds:
\begin{hyp} \label{hyp:link}
The link function $\eta$ and the log-partition function $A$ are such that
\begin{align*}
A'(\eta(x)) = x \mbox{ for all } x \in \XX,
\end{align*}
where $A'$ denotes the first derivative of $A$.
\end{hyp}
Since $\E[ \bY_{ij} ] = A'(\eta(\bX_{ij}))$ for exponential families in the general form \eqref{eq:expfam}, Assumption \ref{hyp:link} implies that $\E[\bY_{ij}] = \bX_{ij}$,    and thus the data matrix satisfies the relation $\bY = \bX + \bW$ where $\bW$ is a centered noise matrix, which is in agreement with model \eqref{eq:deformodel}.
Now, if we let  $\Theta = \eta(\XX)$, it will be also convenient to consider the expression of the pdf of $\bY_{ij}$ in the canonical form:
\begin{equation}
p(y ; \btheta_{ij}) =  h(y) \exp \left( \btheta_{ij} y - A(\btheta_{ij}) \right),\; y \in \YY,  \label{eq:expfamcan}
\end{equation}
where $\btheta_{ij} = \eta(\bX_{ij}) \in \Theta$ is usually called the canonical parameter of the exponential family. Finally, we  recall the relation $ \var( \bY_{ij} ) = A''(\btheta_{ij}) = A''(\eta(\bX_{ij}))$ where $A''$ denotes the second derivative of $A$.
Then, we denote by $\btheta$ the $n \times m$ matrix whose entries are the $\btheta_{ij}$'s.

Examples of data satisfying model \eqref{eq:expfam} are the following ones: \\

\noindent {\bf Gaussian noise with known variance $\tau^{2}$:} \\
$$
q(y ; \bX_{ij}) = \frac{1}{\sqrt{2 \pi}} \exp \left( - \frac{(y - \bX_{ij})^{2}}{2 \tau^2} \right), \; \E[\bY_{ij}] = \bX_{ij}, \; \var(\bY_{ij}) = \tau^{2},
$$
$$
\YY = \R,  \; \XX = \R, \; \Theta = \R, \;  h(y) = \frac{1}{\sqrt{2 \pi} \tau} \exp\left( - \frac{y^2}{2 \tau^2} \right), \; \eta(x) = \frac{x}{\tau^{2}}, \; A(\theta) = \tau^{2} \frac{\theta^2}{2}. \vspace{0.5cm}
$$

\noindent {\bf Gamma-distributed measurements with known shape parameter $L > 0$:} \\
$$
q(y ; \bX_{ij}) = \frac{L^{L} y^{L-1}}{\Gamma(L) \bX_{ij}^{L}} \exp \left( -L \frac{y}{\bX_{ij}} \right) \1_{]0,+\infty[}(y), \; \E[\bY_{ij}] = \bX_{ij}, \; \var(\bY_{ij}) = \frac{\bX_{ij}^{2}}{L},
$$
$$
\YY = \R, \; \XX = ]0,+\infty[, \; \Theta = ]-\infty,0[,  \; h(y) =  \frac{L^{L} y^{L-1}}{\Gamma(L)} \1_{]0,+\infty[}(y), \; \eta(x) = -\frac{L}{x}, \; A(\theta) = -L \log\left(-\frac{\theta}{L}\right). \vspace{0.5cm}
$$



The matrix $\btheta = \eta(X)$ 
 can then be estimated via the $n \times m$ matrix $\hat{\btheta}^{f} = \hat{\btheta}^{f}( \bY ) $ whose entries are given by
\begin{equation} \label{eq:hatZf}
\hat{\btheta}^{f}_{ij}( \bY ) =  \eta \left( \hat{\bX}^f_{ij} \right),  \mbox{ for all } 1 \leq i \leq n, \; 1 \leq j \leq m,
\end{equation}
where $\hat{\bX}^f_{ij}$ is a spectral estimator as defined
in eq.~\eqref{eq:hatXf}.

In the rest of this section, we follow the arguments in \cite{Eldar09} and  \cite{Deledalle} to derive SURE-like formulas under the exponential family for the estimators $\hat{\btheta}^{f}$ and $\hat{\bX}^{f}$, using either the mean-squared error (MSE) risk or the  Kullback-Leibler (KL) risk.

\subsubsection{Unbiased estimation of the MSE risk}

We consider the following MSE risk which provides a measure of discrepancy in the space $\Theta$ of natural parameters, and then indirectly in the space of interest $\XX$.

\begin{defin}  \label{def:MSErisk}
{\it
The squared error (SE) risk of $\hat{\btheta}^{f}$ is
$\SE(\hat{\btheta}^{f}, \btheta) = \| \hat{\btheta}^{f} - \btheta \|^{2}_{F}$, and the mean-squared error (MSE) risk of $\hat{\btheta}^{f}$ is defined as
$
\MSE(\hat{\btheta}^{f}, \btheta)  = \E \left[ \SE(\hat{\btheta}^{f}, \btheta) \right] = \E \left[ \| \hat{\btheta}^{f} - \btheta \|^{2}_{F} \right].
$
}
\end{defin}

Using the above MSE risk to compare $\hat{\btheta}^{f}$ and $\btheta$ implies that the discrepancy between the estimator $\hat{\bX}^{f}$ and the matrix of interest $\bX$ is  measured by the  quantity
$
\MSE_{\eta}(\hat{\bX}^{f},\bX) = \MSE(\eta(\hat{\bX}^{f}),\eta(\bX) )
$
which is different from $\MSE(\hat{\bX}^{f},\bX)$.
For Gaussian noise, $\MSE_{\eta}(\hat{\bX}^{f},\bX)  = \frac{1}{\tau^{2}} \E \left[ \|  \hat{\bX}^{f} -  \bX \|^{2}_{F} \right]$, while for Gamma distributed measurements with known shape parameter $L > 0$, it follows that
\begin{align*}
\MSE_{\eta}(\hat{\bX}^{f},\bX) = L^{2} \sum_{i=1}^{n} \sum_{j=1}^{m}  \left(   \frac{\bX_{ij} - \hat{\bX}^{f}_{ij}}{\bX_{ij} \hat{\bX}^{f}_{ij} }  \right)^2.
\end{align*}
The following proposition gives a SURE formula for the  MSE risk introduced in Definition \ref{def:MSErisk}.

\begin{prop} \label{prop:SURE-MSE}
Suppose that the data are sampled from a continuous exponential family. Assume that the function $h$, in the definition \eqref{eq:expfamcan} of the exponential family, is twice continuously differentiable on $\YY = \R$. If the following condition holds
\begin{equation}
\E \left[ \left| \hat{\btheta}^{f}_{ij}(\bY)  \right|  \right] < + \infty, \mbox{ for all } 1 \leq i \leq n, \; 1 \leq j \leq m, \label{eq:condGSURE}
\end{equation}
then, the quantity
\begin{equation}\label{eq:GSURE}
\GSURE( \hat{\btheta}^{f} )  = \|  \hat{\btheta}^{f}( \bY ) \|^{2} +  \sum_{i=1}^{n} \sum_{j=1}^{m} \left( 2 \frac{h'(\bY_{ij})}{h(\bY_{ij})} \hat{\btheta}^{f}_{ij}(\bY) + \frac{h''(\bY_{ij})}{h(\bY_{ij})} \right) + 2 \div \hat{\btheta}^{f}( \bY ),
\end{equation}
where
$\displaystyle
 \div \hat{\btheta}^{f}( \bY ) =  \sum_{i=1}^{n} \sum_{j=1}^{m}  \frac{\partial \hat{\btheta}^{f}_{ij}( \bY ) }{ \partial \bY_{ij}},
$
is an unbiased estimator of $\MSE(\hat{\btheta}^{f}, \btheta)$
\end{prop}

Note that $\GSURE( \hat{\btheta}^{f} )$ is an unbiased estimator of $\MSE(\hat{\btheta}^{f}, \btheta)$ and not of  $\MSE(\hat{\bX}^{f},\bX)$. It is shown in Section \ref{sec:SUREGauss} that the results of Proposition \ref{eq:condGSURE} coincide with the approach in \cite{MR3105401} on the derivation of a SURE formula in the case of Gaussian noise for smooth spectral estimators.
In the case of Gamma noise, assuming $L > 2$ implies that the conditions on the function $h$ in Proposition \ref{prop:SURE-MSE} is satisfied,
hence assuming that conditions \eqref{eq:condGSURE} holds as well, and using that
\begin{align*}
 \hat{\btheta}^{f}_{ij}(\bY) = -\frac{L}{ f_{ij}(\bY)}
 \quad
 \text{and}
 \quad
 \frac{\partial \hat{\btheta}^{f}_{ij}(\bY) }{ \partial \bY_{ij}} = \frac{L}{| f_{ij}(\bY) |^2} \frac{\partial f_{ij}(\bY) }{ \partial \bY_{ij}},
\end{align*}
it follows that
\begin{align}\label{eq:GSUREexp}
  \GSURE( \hat{\btheta}^{f} ) &= \sum_{i=1}^{n} \sum_{j=1}^{m} \frac{L^2}{ |f_{ij}(\bY)|^2 } - \frac{2 L(L-1)}{\bY_{ij} f_{ij}(\bY)}  + \frac{ 2 L  }{| f_{ij}(\bY) |^2}\frac{\partial f_{ij}\left(\bY\right) }{ \partial \bY_{ij}} - \frac{(L-1)(L-2)}{ |\bY_{ij}|^2 }.
\end{align}

\subsubsection{Unbiased estimation of  KL risks}

Following the terminology in \cite{Deledalle}, let us now introduce two different notions of Kullback-Leibler risk, which arise from the non-symmetry of this discrepancy measure.

\begin{defin}  \label{def:KLrisk}
{\it
Let $f : \R^{n \times m} \to \R^{n \times m}$ be a smooth spectral function. Consider the estimator $\hat{\btheta}^{f}$ defined  by \eqref{eq:hatZf}, where $\bY$ is a matrix whose entries $\bY_{ij}$  are independent random variables sampled from the exponential family \eqref{eq:expfamcan} in canonical form:
\begin{description}
\item[$\bullet$] the  Kullback-Leibler synthesis (KLS) risk of $\hat{\btheta}^{f}$ is defined as
\begin{align*}
\KLS(\hat{\btheta}^{f}, \btheta) = \sum_{i=1}^{n} \sum_{j=1}^{m}  \int_{\R} \log\left( \frac{p(y ;  \hat{\btheta}^{f}_{ij}  )}{ p(y ; \btheta_{ij}) } \right) p(y ; \hat{\btheta}^{f}_{ij} ) \;\mathrm{d} y,
\end{align*}
and the mean KLS risk of $\hat{\btheta}^{f}$ is defined as
$
\MKLS(\hat{\btheta}^{f}, \btheta) = \E \left[ \KLS(\hat{\btheta}^{f}, \btheta) \right],
$
\item[$\bullet$] the   Kullback-Leibler analysis (KLA) risk    of $\hat{\btheta}^{f}$ is defined as
\begin{align*}
\KLA(\hat{\btheta}^{f}, \btheta) = \sum_{i=1}^{n} \sum_{j=1}^{m} \int_{\R} \log\left( \frac{p(y ;   \btheta_{ij}  )}{ p(y ; \hat{\btheta}^{f}_{ij}) } \right) p(y ; \btheta_{ij} ) \;\mathrm{d} y,
\end{align*}
and the mean KLA risk    of $\hat{\btheta}^{f}$ is defined as
$
\MKLA(\hat{\btheta}^{f}, \btheta) =  \E \left[ \KLA(\hat{\btheta}^{f}, \btheta) \right].
$
\end{description}
}
\end{defin}

A key advantage of the Kullback-Leibler risk is that it measures the discrepancy between  the unknown distribution  $p(y ;   \btheta_{ij} )$ and its estimate $p(y ; \hat{\btheta}^{f}_{ij})$. It is thus invariant with respect to the reparametrization $\hat{\btheta}^{f} = \eta(\hat{\bX}^{f})$ (unlike the MSE risk),   and we may also write $\MKLS(\hat{\btheta}^{f}, \btheta) = \MKLS(\hat{\bX}^{f}, \bX)$ and $\MKLA(\hat{\btheta}^{f}, \btheta) = \MKLA(\hat{\bX}^{f}, \bX)$. As suggested in  \cite{Deledalle}, the MKLA risk represents how well the distribution $p(y ; \hat{\btheta}^{f}_{ij})$ explain a random variable $\bY_{ij}$ sampled from the pdf $p(y ;   \btheta_{ij} )$. The MKLA risk is a natural loss function in many statistical problems since it takes as a reference measure the true distribution of the data, see e.g.\ \cite{MR913570}. The MKLS risk represents how well one may generate an independent copy of $\bY_{ij}$ by sampling a random variable from the pdf $p(y ; \hat{\btheta}^{f}_{ij})$. The MKLS risk has also been considered in various inference problems in statistics \cite{MR2181981,MR1272745}.

By simple calculation, it follows that
\begin{align} \label{eq:MKLS}
\MKLS(\hat{\btheta}^{f}, \btheta) &=  \sum_{i=1}^{n} \sum_{j=1}^{m}    \E \left[  \left(\hat{\btheta}^{f}_{ij}  - \btheta_{ij} \right)   A'( \hat{\btheta}^{f}_{ij})  \right] +  A(\btheta_{ij}) - \E  \left[ A(\hat{\btheta}^{f}_{ij}) \right],
\\
\text{and} \quad
\label{eq:MKLA}
\MKLA(\hat{\btheta}^{f}, \btheta) &=  \sum_{i=1}^{n} \sum_{j=1}^{m}    \E \left[  \left(\btheta_{ij}  -   \hat{\btheta}^{f}_{ij}\right)   A'( \btheta_{ij})  \right] +   \E  \left[ A( \hat{\btheta}^{f}_{ij} )  \right] -   A( \btheta_{ij} ).
\end{align}
Hence,   in the case of Gaussian measurements with known variance $\tau^2$, we easily retrieve that
$
\MKLS(\hat{\btheta}^{f}, \btheta) = \MKLA(\hat{\btheta}^{f}, \btheta) =   \frac{\tau^{2}}{2}  \MSE(\hat{\btheta}^{f}, \btheta)  = \frac{1}{2 \tau^2} \E \left[ \| \hat{\bX}^{f} - \bX \|^{2}_{F} \right].
$
In the case of Gamma distributed measurements with known shape parameter $L > 0$, it follows that
\begin{eqnarray*}
\MKLS(\hat{\btheta}^{f}, \btheta) & = &     L \sum_{i=1}^{n} \sum_{j=1}^{m}  \E \left[ \frac{\hat{\bX}^{f}_{ij}}{\bX_{ij}} - \log \left(  \frac{\hat{\bX}^{f}_{ij}}{\bX_{ij}}  \right) -1\right]    , \\
\MKLA(\hat{\btheta}^{f}, \btheta) & = &  L \sum_{i=1}^{n} \sum_{j=1}^{m}  \E \left[ \frac{\bX_{ij}}{\hat{\bX}^{f}_{ij}} - \log \left(  \frac{\bX_{ij}}{\hat{\bX}^{f}_{ij}}  \right) -1\right].
\end{eqnarray*}
Below, we use some of the results  in \cite{Deledalle} whose main contributions are the derivation of new unbiased estimators of the MKLS and MKLA risks. For continuous exponential family, the  risk estimate derived in \cite{Deledalle} is unbiased for the MKLS risk, while it is only asymptotically unbiased for  the MKLA risk with respect to the signal-to-noise ratio.  For data sampled from a continuous exponential family, this makes simpler the use of  the MKLS risk to derive data-driven shinkage in low rank matrix denoising, and we have therefore chosen to concentrate our study on this risk in this setting.  The following proposition establishes a SURE formula to estimate  the  MKLS risk in the continuous case.

\begin{prop} \label{prop:SURE-MKLS}
Suppose that the data are sampled from a continuous exponential family. Assume that the function $h$, in the definition \eqref{eq:expfamcan} of the exponential family, is continuously differentiable on $\YY = \R$. Suppose that the function $A$, in the definition \eqref{eq:expfamcan} of the exponential family, is twice continuously differentiable on $\Theta$. If the following condition holds
\begin{equation}
\E \left[ \left| A'( \hat{\btheta}^{f}_{ij}(\bY) ) \right| \right] < + \infty, \mbox{ for all } 1 \leq i \leq n, \; 1 \leq j \leq m, \label{eq:condSUKLS}
\end{equation}
then, the quantity
\begin{equation}\label{eq:SUKLS}
\SUKLS(\hat{\btheta}^{f}) = \sum_{i=1}^{n} \sum_{j=1}^{m}  \left( \left( \hat{\btheta}^{f}_{ij}(\bY) +  \frac{h'(\bY_{ij})}{h(\bY_{ij})} \right) A'(  \hat{\btheta}^{f}_{ij}(\bY) ) - A(\hat{\btheta}^{f}_{ij}(\bY) )   \right)  + \div f ( \bY ),
\end{equation}
where
$\displaystyle
\div f ( \bY ) = \sum_{i = 1}^{m} \sum_{j = 1}^{n}  \frac{\partial f_{ij}(\bY) }{ \partial \bY_{ij}},
$
is an unbiased estimator of $\displaystyle \MKLS(\hat{\btheta}^{f}, \btheta) - \sum_{i=1}^{n} \sum_{j=1}^{m} A(\btheta_{ij})$. 
\end{prop}

A key difference in the formula of unbiased estimates for the MSE and the KL risks is the computation of the divergence term in \eqref{eq:GSURE} and \eqref{eq:SUKLS}, when $\hat{\bX}^{f} = \sum_{k = 1}^{\min(n,m)} f_{k} ( \tilde{\sigma}_{k} ) \tilde{\bu}_{k} \tilde{\bv}_{k}^{t}$ is a smooth spectral estimator in the sense where each function $f_{k} : \R_{+} \to \R_{+}$ is assumed to be (almost everywhere) differentiable for $1 \leq k \leq \min(n,m)$. In this setting, the divergence term in the expression of $\GSURE( \hat{\btheta}^{f} )$ depends upon the matrix $\hat{\btheta}^{f}(\bY) = \eta(\hat{\bX}^{f})$. Therefore, when $\eta$ is a nonlinear mapping, it is generally not possible to obtain a simpler expression for $ \div \hat{\btheta}^{f}( \bY )$. To the contrary, for $\SUKLS( \hat{\btheta}^{f} )$, the divergence term is $\div f ( \bY )$ which has the following closed-form expression for any smooth spectral estimators
\begin{equation}
 \div f ( \bY ) = |m - n| \sum_{k = 1}^{\min(n,m)} \frac{f_{k}(\tilde{\sigma}_{k})}{\tilde{\sigma}_{k}} + \sum_{k = 1}^{\min(n,m)}f_{k}'(\tilde{\sigma}_{k}) + 2 \sum_{k = 1}^{\min(n,m)} f_{k}(\tilde{\sigma}_{k}) \sum_{\ell =1 ; \ell \neq k}^{\min(n,m)} \frac{\tilde{\sigma}_{k}}{\tilde{\sigma}_{k}^{2} - \tilde{\sigma}_{\ell}^{2} }~, \label{eq:divf}
\end{equation}
thanks to the results from Theorem IV.3 in \cite{MR3105401}.

Note that $\SUKLS( \hat{\bX}^{r}_{w} ) = \frac{\tau^2}{2} \SURE( \hat{\bX}^{r}_{w} )$ for Gaussian measurements, hence, the $\GSURE$ and $\SUKLS$ strategies match in this case.
In the case of Gamma measurements, assuming that $L > 2$ implies that the conditions on the function $h$ in Proposition \ref{prop:SURE-MKLS} is satisfied, and by
assuming that condition \eqref{eq:condSUKLS} holds as well,
it follows that
\begin{align*}
\SUKLS(\hat{\btheta}^{f}) &= \sum_{i=1}^{n} \sum_{j=1}^{m}  \left(  (L-1) \frac{ f_{ij}(\bY)  }{\bY_{ij}} - L \log \left( f_{ij}(\bY) \right)   \right) - L m n  + \div f ( \bY ),
\end{align*}
where the expression of $\div f ( \bY )$ is given by \eqref{eq:divf}.

Note that it is implicitly understood in the definition of $\div f ( \bY )$ that each mapping $f_{ij} : \R^{n \times m} \to \R$ is differentiable. The differentiability of the spectral function $f$ (and thus of its components $f_{ij}$) is a consequence of the assumption that the functions  $f_{1},\ldots,f_{\min(n,m)}$ (acting on the singular values) are supposed to be differentiable. For further details, on the differentiability of $f$ and the $f_{ij}$'s, we refer to Section IV in \cite{MR3105401}. From the arguments in \cite{MR3105401}, it  follows  that formula \eqref{eq:divf} for the divergence of $f$ is also valid under the assumption that each function $f_{k}$ is  differentiable on $\R_{+}$ except on a set of Lebesgue measure zero.

\subsection{The case of Poisson data}


For Poisson data, the key result to obtain unbiased estimate of a given risk is the following lemma which dates back to the work in \cite{hudson1978}.

\begin{lem} \label{lem:hudson}
Let $f : \Z^{n \times m} \to \R^{n \times m}$ be a measurable mapping. Let $1 \leq i \leq n$ and $1 \leq j \leq m$, and denote by $f_{ij} : \Z^{n \times m} \to \R$ a measurable function. Let $\bY \in \Z^{n \times m}$ be a matrix whose entries are independently sampled from a Poisson distribution on $\Z$.  Then,
\begin{align*}
\E \left[  \sum_{i=1}^{n} \sum_{j=1}^{m} \bX_{ij} f_{ij}(\bY) \right] = \E \left[   \sum_{i=1}^{n} \sum_{j=1}^{m} \bY_{ij} f_{ij}(\bY - \be_{i} \be_{j}^{t})  \right],
\end{align*}
where,  for each $1 \leq i \leq n$ and $1 \leq j \leq m$, $f_{ij}(\bY)$ denotes the $(i,j)$-th entry of the matrix $f(\bY)$, and $\be_{i}$ (resp.\ $\be_{j}$) denotes the vector of $\Z^{n}$ (resp.\ $\Z^{m}$) with the $i$-th entry (resp.\ $j$-th entry) equals to one and all others equal to zero.
\end{lem}

Hudson's lemma provides a way to estimate (in an unbiased way) the expectation of the Frobenius inner product between the matrix $\bX$ and the matrix $f(\bY)$. To see the usefulness of this result, one may consider the following mean-squared error
\begin{align*}
{\MSE}(\hat{\bX}^{f},\bX)  =  \E \left[ \left\| \hat{\bX}^{f}  - \bX \right\|^{2}_{F} \right] =
  \E \left[ \left\| \hat{\bX}^{f}   \right\|^{2}_{F} - 2 \sum_{i=1}^{n} \sum_{j=1}^{m} \bX_{ij} \hat{\bX}^{f}_{ij}(\bY)  +  \left\|  \bX \right\|^{2}_{F} \right].
\end{align*}
Therefore, by Lemma \ref{lem:hudson}, one immediately obtains that
\begin{equation} \label{eq:PURE}
\PURE(\hat{\btheta}^{f}) = \left\| \hat{\bX}^{f} \right\|^{2}_{F} -2  \sum_{i=1}^{n} \sum_{j=1}^{m}  \bY_{ij} f_{ij}(\bY - \be_{i} \be_{j}^{t}),
\end{equation}
is an unbiased estimate for the quantity ${\MSE}(\hat{\bX}^{f},\bX) -  \left\| \bX \right\|^{2}_{F}$.

For Poisson data, one may also define the following KL risks
\begin{eqnarray}
\MKLS(\hat{\btheta}^{f}, \btheta) & = & \sum_{i=1}^{n} \sum_{j=1}^{m}  \E \left[\bX_{ij} - \hat{\bX}^{f}_{ij} -   \hat{\bX}^{f}_{ij} \log \left( \frac{\bX_{ij}}{\hat{\bX}^{f}_{ij}}  \right)  \right], \label{eq:MKLSPoisson} \\
\MKLA(\hat{\btheta}^{f}, \btheta) & = &   \sum_{i=1}^{n} \sum_{j=1}^{m}  \E \left[ \hat{\bX}^{f}_{ij} - \bX_{ij} -  \bX_{ij} \log \left( \frac{\hat{\bX}^{f}_{ij}}{\bX_{ij}}  \right)  \right], \label{eq:MKLAPoisson}
\end{eqnarray}
which are in agreement with Definition \ref{def:KLrisk} of KL risks for data sampled from a Poisson distribution. From the arguments in \cite{Deledalle}, there does not currently exist an approach to derive a SURE formula for the MKLS risk in the Poisson case since they are no unbiased formula for
$\hat{\bX}^{f}_{ij} \log \bX_{ij}$.
Nevertheless, as shown  in   \cite{Deledalle},
Hudson's Lemma \ref{lem:hudson} provides an unbiased estimator for
$\bX_{ij} \log \hat{\bX}^{f}_{ij}$,
and then it is possible to unbiasedly estimate  the MKLA risk   as follows.
\begin{prop} \label{prop:SURE-MKLA}
For data sampled from a Poisson distribution, the quantity
\begin{equation}\label{eq:PUKLA}
\PUKLA(\hat{\btheta}^{f}) =  \sum_{i=1}^{n} \sum_{j=1}^{m} \hat{\bX}^{f}_{ij} - \bY_{ij} \log\left( f_{ij}(\bY - \be_{i} \be_{j}^{t}) \right),
\end{equation}
is an unbiased estimator of $\displaystyle \MKLA(\hat{\btheta}^{f}, \btheta) +  \sum_{i=1}^{n} \sum_{j=1}^{m}   \bX_{ij} -  \bX_{ij} \log \left( \bX_{ij}  \right)$.
\end{prop}

\subsection{Data-driven shrinkage in low-rank matrix denoising} \label{sec:optexp}

\label{sec:optexp_rank1}

For a matrix $\bX$ with entries $\bX_{ij} \in \mathcal{X} = \mathbb{R}$,
we consider shrinkage estimators of the form
\begin{align} \label{eq:hatXw1}
  &\hat{\bX}^{s}_{w} = f ( \bY ) = \sum_{k \in s}  w_{k} \tilde{\sigma}_{k} \tilde{\bu}_{k} \tilde{\bv}_{k}^{t},
\end{align}
with $s \subseteq \mathcal{I} = \{1,2,\ldots,\min(n, m)\}$ and
$w_{k} \in [0, 1]$, for all $k \in s$.

When the underlying matrix $\bX$ is constrained to have positive entries, e.g.\ $\mathcal{X} = ]0, +\infty[$ in the Gamma and Poisson cases, 
we consider instead estimators of the form
\begin{align} \label{eq:hatXw1_eps}
  \hat{\bX}^{s}_{w} = f ( \bY ) =
  \max\left[
    \sum_{k \in s}  w_{k} \tilde{\sigma}_{k} \tilde{\bu}_{k} \tilde{\bv}_{k}^{t},
    \varepsilon
    \right],
\end{align}
where $\varepsilon > 0$ is an {\it a priori} lower bound on the smallest value of $\bX_{ij}$, where for any matrix $\bX$, $\max[\bX, \varepsilon]_{ij} = \max[\bX_{ij}, \varepsilon]$, for all $1 \leq i \leq m$ and $1 \leq j \leq n$.

The construction of the subset $s$ is postponed to Section \ref{sec:bulk_exp_fam},
and we focus here in selecting the weights in a data-driven way
for a fixed given $s$.
In the following, we denote by $s^c$ the complementary set of $s$ in $\mathcal{I}$,
{\it i.e.}, $s^c = \mathcal{I} \;\backslash\; s$,
and we let $\hat{\btheta}^{s}_{w} = \eta\left( \hat{\bX}^{s}_{w} \right)$. When  $\mathcal{X} = ]0, +\infty[$, we have found that considering estimators of the form \eqref{eq:hatXw1_eps} is more appropriate than trying to find shrinking weights $(w_{k})_{k \in s}$ such that all the entries of the matrix $\sum_{k \in s}  w_{k} \tilde{\sigma}_{k} \tilde{\bu}_{k} \tilde{\bv}_{k}^{t}$ are positive,  for a given subset $s$.

\subsubsection*{Gaussian noise with known homoscedastic variance $\tau^{2}$}

By applying the GSURE formula \eqref{eq:GSURE} for Gaussian distributed measurements and thanks to the expression \eqref{eq:divf} for the divergence of smooth spectral estimators, we obtain for $\hat{\bX}^{s}_{w}$, as defined in \eqref{eq:hatXw1}, the SURE expression given by
\begin{gather*}
  \hspace{-2em}
  \SURE ( \hat{\bX}^{s}_{w} ) = - m n \tau^2 + \sum_{k \in s}  (w_{k}-1)^2  \tilde{\sigma}_{k}^2  +  \sum_{k \in s^c}  \tilde{\sigma}_{k}^2
  + 2 \tau^2  \sum_{k = 1}^{s} \left(  1+ |m - n|  +  2 \sum_{\ell =1 ; \ell \neq k}^{\min(n,m)} \frac{ \tilde{\sigma}_{k}^{2}}{\tilde{\sigma}_{k}^{2} - \tilde{\sigma}_{\ell}^{2} } \right) w_{k}
\end{gather*}
which unbiasedly estimate $\MSE(\hat{\bX}^{s}_{w}, \bX)$.
Hence, for each $k \in s$, by differentiating the above expression with respect to $w_{k}$, it follows that a data-driven weight for the $k$-th empirical singular value is given by
\begin{equation}
w_{k}(\bY) = \left( 1 - \frac{\tau^2}{ \tilde{\sigma}_{k}^{2}}  \left(  1+ |m - n| +  2    \sum_{\ell =1 ; \ell \neq k}^{\min(n,m)} \frac{ \tilde{\sigma}_{k}^{2}}{\tilde{\sigma}_{k}^{2} - \tilde{\sigma}_{\ell}^{2} } \right) \right)_{+}, \label{eq:optGaussSURE}
\end{equation}
which fullfils the requirement that $w_k(\bY) \in [0, 1]$.
Note that as $\SUKLS( \hat{\bX}^{s}_{w} ) = \frac{\tau^2}{2} \SURE( \hat{\bX}^{s}_{w} )$ for Gaussian measurements, the exact same data-driven weight would be obtained by minimizing an estimate of the $\MKLS(\hat{\bX}^{s}_{w}, \bX)$.

\bigskip
{\it The case of estimators with rank one.}
Consider the case of estimators with rank $1$, {\it i.e.}, let $s = \{ 1 \}$.
It follows that
$\hat \bX^1_w = \hat \bX^{\{1\}}_w = w_1 \hat \bX^1$ where $w_1 \in [0, 1]$ is given by
\begin{equation*}
w_{1}(\bY) = \left( 1 - \frac{\tau^2}{ \tilde{\sigma}_{1}^{2}}  \left(  1+ |m - n| +  2    \sum_{\ell =1 ; \ell \neq 1}^{\min(n,m)} \frac{ \tilde{\sigma}_{1}^{2}}{\tilde{\sigma}_{1}^{2} - \tilde{\sigma}_{\ell}^{2} } \right) \right)_{+}.
\end{equation*}

\subsubsection*{Gamma and Poisson distributed measurements}

In Gamma and Poisson cases, it is not possible to follow the same strategy as in the Gaussian case to derive optimal weights for \eqref{eq:hatXw1_eps} in a closed-form using the established SURE-like formulas.
We shall investigate how data-driven shinkage can be approximated in Section \ref{sec:num} on numerical experiments using fast algorithms.
Nevertheless, when the estimator is restricted to rank $1$, optimizing KL risk estimators lead to closed-form expressions under the assumption that all the entries of the data matrix $\bY$ are strictly positive.

\bigskip

{\it The case of estimators with rank one under Gamma noise.}
Consider again the case of estimators with rank $1$, {\it i.e.}, let $s = \{ 1 \}$,
and let $\hat \bX^1 = \tilde{\sigma}_{1} \tilde{\bu}_{1} \tilde{\bv}_{1}^{t}$
denote the PCA approximation of rank 1 of $\bX$.
If all the entries of the matrix $\bY$ are strictly positive, by the Perron-Frobenius theorem, all the entries of the first singular vectors  $\tilde{\bu}_{1}$ and $\tilde{\bv}_{1}$ are strictly positive. Therefore, all the entries of $\hat{\bX}^{1}$ belong to the set $\XX = ]0,+\infty[$, and
we can consider $\hat \bX^1_w = \hat \bX^{\{1\}}_w = w_1 \tilde{\sigma}_{1} \tilde{\bu}_{1} \tilde{\bv}_{1}^{t}$ as defined in \eqref{eq:hatXw1} instead of \eqref{eq:hatXw1_eps}.
Assuming $L > 2$ for the $\SUKLS$ formula to hold, it follows by simple calculations that
\begin{align*}
\SUKLS(\hat{\btheta}^{1}_{w})   =    \sum_{i=1}^{n} \sum_{j=1}^{m}  (L-1)   w_{1} \frac{\hat \bX^1_{ij}}{\bY_{ij} }  - m n L  &  \log \left( w_{1} \right)   - L   \log \left(  \frac{\hat \bX^1_{ij}}{\bY_{ij} } \right)  - L m n \\
&  + (1 + |m - n| ) w_{1} +  2 w_{1} \sum_{\ell =2  }^{\min(n,m)} \frac{\tilde{\sigma}_{1}^{2}}{\tilde{\sigma}_{1}^{2} - \tilde{\sigma}_{\ell}^{2} }.
\end{align*}
Hence, by differentiating the above expression with respect to $w_{1}$
and as it is monotonic on both sides of its unique minimum,
the optimal value of $w_{1} \in [0, 1]$ minimizing $\SUKLS(\hat{\btheta}^{1}_{w}) $
is given by
\begin{align*}
w_{1}(\bY) =  \min \left[ 1, \left(    \frac{L - 1}{L m n } \sum_{i=1}^{n} \sum_{j=1}^{m}   \frac{\hat \bX^1_{ij}}{\bY_{ij} } + \frac{1}{L m n} \left(1 + |m - n| + 2   \sum_{\ell =2  }^{\min(n,m)} \frac{\tilde{\sigma}_{1}^{2}}{\tilde{\sigma}_{1}^{2} - \tilde{\sigma}_{\ell}^{2} } \right) \right)^{-1} \right],
\end{align*}
which yields the  shrinking rule \eqref{eq:optgamma} stated in the introduction of this paper.
Note that it is not possible to obtain, in a closed-form, the optimal value of the weight $w_{1}$ that minimizes the criterion $\GSURE( \hat{\btheta}^{1}_{w})$.

\bigskip

{\it The case of estimators with rank one under Poisson noise.}
Using again that all the assumption that the entries of $\bY$ are positive, 
we can consider (by the Perron-Frobenius theorem) $\hat \bX^1_w = \hat \bX^{\{1\}}_w = w_1 \tilde{\sigma}_{1} \tilde{\bu}_{1} \tilde{\bv}_{1}^{t}$ as defined in \eqref{eq:hatXw1} instead of \eqref{eq:hatXw1_eps}.
Then, the PURE formula \eqref{eq:PURE} and Proposition \ref{prop:SURE-MKLA} apply to the estimator $\hat{\btheta}^{1}_{w} = \log\left( \hat{\bX}^{1}_{w} \right)$ which yield to
\begin{align*}
&\PURE(\hat{\btheta}^{1}_{w}) = w_{1}^{2} \tilde{\sigma}_{1}^{2} -2  \sum_{i=1}^{n} \sum_{j=1}^{m}  \bY_{ij} w_{1}  \tilde{\sigma}_{1}^{(ij)} \tilde{\bu}_{1,i}^{(ij)} \tilde{\bv}_{1,j}^{(ij)},\\
\text{and} \quad &
\PUKLA(\hat{\btheta}^{1}_{w}) =  \sum_{i=1}^{n} \sum_{j=1}^{m} w_{1} \hat \bX^1_{ij} - \bY_{ij} \left( \log\left( w_{1}  \right) +  \log\left( \tilde{\sigma}_{1}^{(ij)} \tilde{\bu}_{1,i}^{(ij)} \tilde{\bv}_{1,j}^{(ij)}  \right) \right),
\end{align*}
where $\hat \bX^1_{ij} = \tilde{\sigma}_1 \tilde{\bu}_{1,i} \tilde{\bv}_{1,j}$,  $\tilde{\sigma}_{1}^{(ij)}$ is the largest singular value of the matrix $\bY - \be_{i} \be_{j}^{t}$, and $\tilde{\bu}_{1}^{(ij)}$ (resp.\ $\tilde{\bv}_{1}^{(ij)}$) denotes its left (resp.\ right) singular vectors. Therefore,
by differentiating the above expression with respect to $w_{1}$
and as it is monotonic on both sides of its unique minimum,
an optimal value for $w_1 \in [0, 1]$ which minimizes $\PURE(\hat{\btheta}^{1}_{w})$ is given by
\begin{align*}
w_{1}(\bY) = \min\left[1, \frac{1}{\tilde{\sigma}_{1}^{2}} \sum_{i=1}^{n} \sum_{j=1}^{m}  \bY_{ij}   \tilde{\sigma}_{1}^{(ij)} \tilde{\bu}_{1,i}^{(ij)} \tilde{\bv}_{1,j}^{(ij)}\right].
\end{align*}
However, this optimal shrinking rule cannot be used in practice since evaluating the values of $\tilde{\sigma}_{1}^{(ij)},\tilde{\bu}_{1}^{(ij)},\tilde{\bv}_{1}^{(ij)}$ for all $1 \leq i \leq n$ and $1 \leq j \leq m$ is not feasible from a computational point of view for large values of $n$ and $m$. Nevertheless, a fast algorithm to find a numerical approximation of the optimal value ${w}_{1}(\bY)$ is proposed  in Section \ref{sec:num}.

 To the contrary, using again that all the $\hat \bX^1_{ij}$ are positive by the Perron-Frobenius theorem, the value of $w_1 \in [0, 1]$ minimizing $\PUKLA(\hat{\btheta}^{1}_{w})$ is
\begin{align*}
w_{1}(\bY) = \min\left[1, \frac{ \sum_{i=1}^{n} \sum_{j=1}^{m}  \bY_{ij} }{  \sum_{i=1}^{n} \sum_{j=1}^{m} \hat \bX^1_{ij} } \right],
\end{align*}
which is straightforward to compute. This corresponds to the  shrinkage rule \eqref{eq:optpoisson} given in the introduction.

\section{Gaussian spiked population model} \label{sec:gauss}

In this section, we restrict our analysis to the Gaussian spiked population model and the asymptotic setting introduced in Definition \ref{def:spiked}.

%

\subsection{Asymptotic location of empirical singular values}
\label{sec:asymptotic_sv}

We summarize below the asymptotic behavior of the singular values of  the data matrix
$
\bY =  \sum_{k = 1}^{\min(n,m)} \tilde{\sigma}_{k} \tilde{\bu}_{k} \tilde{\bv}_{k}^{t}
$
in the Gaussian spiked population model.

In the case where $\bX = 0$, it is well known \cite{MR2760897,MR2567175} that the empirical distribution of the singular values of $\bY = \bW$ (with $\tau = \frac{1}{\sqrt{m}}$) converges, as $n \to + \infty$, to the quarter circle distribution if $c = 1$ and to its generalized version if $c < 1$. This distribution is supported on the compact interval $[c_{-},c_{+}]$ with
\begin{align*}
c_{\pm} = 1 \pm \sqrt{c}
\end{align*}
where $c_{+}$ is the so-called bulk (right) edge.

 When $\bX \neq 0$ has a low rank structure, the asymptotic behavior of the singular values of $\bY = \bX + \bW$  is also well understood \cite{Benaych-GeorgesN12,MR2322123,MR3054091}, and generalizations to noise matrix $\bW$ whose distribution is orthogonally invariant have also been recently considered in \cite{Benaych-GeorgesN12}. Below, we recall some of these results that will be needed  in this paper. To this end, let us introduce the real-valued function $\rho$ defined by
\begin{align*}
\rho\left( \sigma \right) = \sqrt{\frac{(1+\sigma^{2})(c+\sigma^{2})}{\sigma^{2}}} \mbox{ for any } \sigma > 0.
\end{align*}
Then, the following result holds (see e.g.\ Theorem 2.8 in  \cite{Benaych-GeorgesN12} and Proposition 9 in \cite{MR3054091}).
\begin{prop} \label{prop:sv}
Assume that $\bY  = \bX + \bW$ is a random matrix sampled from the Gaussian spiked population model with $\tau = \frac{1}{\sqrt{m}}$ and $\bX = \sum_{k = 1}^{r^{\ast}} \sigma_{k} \bu_{k} \bv_{k}^{t}$. Then, for any fixed $k \geq 1$, one has that, almost surely,
\begin{align*}
\lim_{n \to + \infty} \tilde{\sigma}_{k} =
\left\{
\begin{array}{cc}
\rho\left( \sigma_{k} \right) & \mbox{ if } k \leq r^{\ast} \mbox{ and }   \sigma_{k} > c^{1/4}, \\
c_{+} & \mbox{ otherwise.}
\end{array}
\right.
\end{align*}
Moreover,
\begin{align*}
\lim_{n \to + \infty} \tilde{\sigma}_{\min(n,m)} = c_{-}.
\end{align*}
\end{prop}
In what follows, we shall also use the relation
\begin{equation}
\frac{1}{\sigma^{2}} = \frac{\rho^{2}(\sigma) -(c+1) - \sqrt{(\rho^{2}(\sigma)-(c+1))^2-4c} }{2 c} \mbox{ that holds for any } \sigma > c^{1/4}, \label{eq:link}
\end{equation}
which is a consequence of e.g.\  the results in Section 3.1 in \cite{Benaych-GeorgesN12}.


\subsection{Existing asymptotically optimal shrinkage rules}

Below, we briefly summarize some results in  \cite{GavishDonoho} and \cite{MR3200641} on the construction of asymptotically optimal spectral estimators. 
Let
\begin{equation}
\hat{\bX}^{f} = f ( \bY ) =  \sum_{k = 1}^{\min(n,m)} f_{k} ( \tilde{\sigma}_{k} ) \tilde{\bu}_{k} \tilde{\bv}_{k}^{t} \label{eq:smoothspectralestim}
\end{equation}
be a given smooth spectral estimator,  
and consider the standard squared error
$
\SE(\hat{\bX}^{f}, \bX) = \| \hat{\bX}^{f} - \bX \|^{2}_{F}
$
as a measure of risk. The set of spectral functions minimizing $\SE(\hat{\bX}^{f}, \bX)$ is given by $f_{k} ( \tilde{\sigma}_{k} ) = \tilde{\bu}_{k}^{t} \bX \tilde{\bv}_{k}$, for $1 \leq k \leq \min(n,m)$. However, it cannot be used in practice since $\bX$ is  obviously unknown. A first alternative suggested in  \cite{GavishDonoho} and \cite{MR3200641} is to rather study the asymptotic risk
\begin{equation}
\SE_{\infty}(\hat{\bX}^{f}) = \lim_{n \to \infty} \SE(\hat{\bX}^{f}, \bX) \mbox{ (in the almost sure sense)} \label{eq:SE}
\end{equation}
 in the Gaussian spiked population model. 
Then, it is proposed  in  \cite{GavishDonoho} and \cite{MR3200641}  to find an asymptotically optimal choice of $f$ by minimizing $ \SE_{\infty}(\hat{\bX}^{f}) $ among a given class of smooth spectral functions. The results in \cite{GavishDonoho} show that, among spectral estimators of the form
$
\hat{\bX}^{\eta} =   \sum_{k = 1}^{\min(n,m)}    \eta(\tilde{\sigma}_{k}) \tilde{\bu}_{k} \tilde{\bv}_{k}^{t},
$
where $\eta : \R_{+} \to \R_{+}$ is a continuous shrinker such that $\eta(\sigma) = 0$ whenever $\sigma \leq c_{+}$,  an asymptotically optimal shrinkage rule is given by the choice
\begin{equation} \label{eq:optshrinkGD}
\eta^{\ast}(\sigma) = \left\{
\begin{array}{cll}
\frac{1}{\sigma} \sqrt{\left(\sigma^2-(c+1) \right)^{2}-4c}& \mbox{ if } \sigma > c_{+}, \\
0 & \mbox{ otherwise} . \\
\end{array}\right.
\end{equation}
In \cite{MR3200641}, it is proposed to consider spectral estimators of the form
$
\hat{\bX}^{\delta} = \sum_{k = 1}^{r} \delta_{k}  \tilde{\bu}_{k} \tilde{\bv}_{k}^{t}
$
where  $\delta_1,\ldots,\delta_{r}$ are positive weights.  By Theorem 2.1 in \cite{MR3200641}, it follows that, if $\sigma_{k} > c^{1/4}$ for all $1 \leq k \leq r$ with $r \leq r^{\ast}$, then the weights which minimize $ \SE_{\infty}(\hat{\bX}^{\delta}) $ over $\R_{+}^{r}$ are given by
\begin{equation} \label{eq:optshrinkRao}
\delta_k^{\ast} = \delta_k(\sigma_{k}) =
\frac{\sigma_{k}^{4} - c}{\sigma_{k} \sqrt{(1+\sigma_{k}^{2})(c+\sigma_{k}^{2})}} ,  \mbox{ for all } 1 \leq k \leq r.
\end{equation}

 In what follows,  the shrinkage rules  \eqref{eq:optshrinkGD}  and \eqref{eq:optshrinkRao} are shown to be equivalent, and they will serve as a reference of asymptotic optimality. It should be stressed that the estimators in  \cite{GavishDonoho}  and \cite{MR3200641}  are not equivalent. Indeed, the method in \cite{MR3200641} requires an estimate of the rank, while the approach in \cite{GavishDonoho} applies the same  shrinker to all empirical singular values. Nevertheless, the shrinkage function that is applied to significant singular values (either above the bulk edge in  \cite{GavishDonoho} or up to a given rank in \cite{MR3200641}) is the same.

\subsection{Asymptotic behavior of data-driven estimators based on SURE} \label{sec:SUREGauss}

Following the principle of SURE, a second alternative to choose a smooth spectral estimator of the form \eqref{eq:smoothspectralestim} is to study the problem of selecting a set of functions $(f_{k})_{1 \leq k \leq \min(n,m)}$ that minimize an unbiased estimate of
$
\MSE(\hat{\bX}^{f}, \bX) = \E \left[ \| \hat{\bX}^{f} - \bX \|^{2}_{F} \right].
$
For any $1 \leq i \leq m$ and $1 \leq j \leq n$, we recall that $f_{ij}(\bY)$ denotes the $(i,j)$-th entry of the matrix $\hat{\bX}^{f} = f ( \bY )$.  Under the  condition that
\begin{equation}
\E \left[ \left| \bY_{ij}  f_{ij}(\bY) \right| + \left| \frac{\partial f_{ij}(\bY) }{ \partial \bY_{ij}} \right|\right] < + \infty, \mbox{ for all } 1 \leq i \leq n, \; 1 \leq j \leq m. \label{eq:condSURE}
\end{equation}
it follows from the results in \cite{MR3105401} (or equivalently from Proposition \ref{prop:SURE-MSE} for Gaussian noise with $\tau^{2} = 1/m$) that
\begin{equation}
\SURE \left( \hat{\bX}^{f} \right) = - n + \|  f ( \bY )  - \bY \|^{2}_{F} + \frac{2}{m} \div f ( \bY ),
\label{eq:SURE}
\end{equation}
is an unbiased estimate of $\MSE(\hat{\bX}^{f}, \bX)$, where the divergence
$
\div f ( \bY )
$
admits the closed-form expression \eqref{eq:divf}. The SURE formula \eqref{eq:SURE} has been used in \cite{MR3105401}  to find a data-driven value for $\lambda = \lambda(\bY)$ in the the case of singular values shrinkage by soft-thresholding  which corresponds to the choice
\begin{align*}
f_{k}(\tilde{\sigma}_{k}) = (\tilde{\sigma}_{k} - \lambda)_{+}, \mbox{ for all } 1 \leq k \leq \min(n,m).
\end{align*}

We study now the asymptotic behavior of the SURE formula \eqref{eq:SURE}. To this end, we shall use Proposition \ref{prop:sv}, but  we will also need the following result (whose proof can be found in the Appendix) to study some of the terms in expression \eqref{eq:divf} of the divergence of $f(\bY)$.

\begin{prop} \label{prop:sti}
Assume that $\bY  = \bX + \bW$ is a random matrix sampled from the Gaussian spiked population model with $\tau = \frac{1}{\sqrt{m}}$ and $\bX = \sum_{k = 1}^{r^{\ast}} \sigma_{k} \bu_{k} \bv_{k}^{t}$. Then, for any fixed $1 \leq k \leq r^{\ast}$ such that $\sigma_{k} > c^{1/4}$, one has that, almost surely,
\begin{align*}
\lim_{n \to + \infty} \frac{1}{n}  \sum_{\ell =1 ; \ell \neq k}^{n} \frac{\tilde{\sigma}_{k}}{\tilde{\sigma}_{k}^{2} - \tilde{\sigma}_{\ell}^{2} } = \frac{1}{\rho\left( \sigma_{k} \right) } \left( 1 + \frac{1}{\sigma_{k}^{2}} \right).
\end{align*}
\end{prop}

In what follows, we restrict  our analysis to the following class of spectral estimators (the terminology in the definition below is borrowed from \cite{GavishDonoho}).

\begin{defin}
Let $\hat{\bX}^{f} = f(\bY) =  \sum_{k = 1}^{\min(n,m)} f_{k} ( \tilde{\sigma}_{k} ) \tilde{\bu}_{k} \tilde{\bv}_{k}^{t}$ be a smooth spectral estimator. For a given $1 \leq r \leq \min(n,m)$, the  estimator $f$ is said to be a spectral shrinker of order $r$ that collapses the bulk to 0 if
\begin{align*}
\left\{\begin{array}{cl}
f_{k}(\sigma) = 0 & \mbox{ whenever }  \sigma \leq c_{+} \mbox{ and } 1 \leq k \leq r,\\
f_{k}(\sigma)  =  0 & \mbox{ for all } \sigma \geq 0 \mbox{ and }  k > r.
\end{array}\right.
\end{align*}
\end{defin}

 The reason for restricting the study to spectral estimators such that $f_{k}(\tilde{\sigma}_{k}) = 0$ whenever $\tilde{\sigma}_{k} <   c_{+} $ is linked to the choice of the active set $s^{\ast}$ \eqref{eq:sastGauss} of singular values in the Gaussian case, as detailed in Section \ref{sec:bulk_exp_fam}. Now, for a spectral shrinker $\hat{\bX}^{f}$ of order $r$ that collapses the bulk to 0, we study the asymptotic behavior of the terms in expression \eqref{eq:SURE}  that only depend on $f$, namely
\begin{eqnarray}
\overline{\SURE} \left( \hat{\bX}^{f} \right) & = & \sum_{k = 1}^{r} (  f_{k}( \tilde{\sigma}_{k})  - \tilde{\sigma}_{k} )^{2}  + 2  \left(1 - \frac{n}{m}\right) \sum_{k = 1}^{r} \frac{f_{k}(\tilde{\sigma}_{k})}{\tilde{\sigma}_{k}} + \frac{2}{m} \sum_{k = 1}^{r}f_{k}'(\tilde{\sigma}_{k}) \nonumber \\
& & +  4 \frac{n}{m} \sum_{k = 1}^{r} f_{k}(\tilde{\sigma}_{k}) \left( \frac{1}{n} \sum_{\ell =1 ; \ell \neq k}^{n} \frac{\tilde{\sigma}_{k}}{\tilde{\sigma}_{k}^{2} - \tilde{\sigma}_{\ell}^{2} } \right)  \label{eq:barSURE}
\end{eqnarray}

The reason for studying $\overline{\SURE} \left( \hat{\bX}^{f} \right) $ is that finding an optimal shrinkage rule that minimizes  $\SURE \left( \hat{\bX}^{f} \right)$  is equivalent to minimizing expression \eqref{eq:barSURE} over spectral shrinkers  of order $r$ that collapses the bulk to 0, since
$
\SURE \left( \hat{\bX}^{f} \right) - \overline{\SURE} \left( \hat{\bX}^{f} \right)  = - n +  \sum_{k = r+1}^{n}  \tilde{\sigma}_{k}^{2}
$
for such $ \hat{\bX}^{f}$.

Then, using Proposition \ref{prop:sv},  Proposition \ref{prop:sti}, and the assumption that  the $f_{k}$'s are continuously differentiable functions on $\R_{+}$, we immediately obtain the following result.

\begin{lem} \label{lem:limSURE}
Assume that $\bY  = \bX + \bW$ is a random matrix sampled from the Gaussian spiked population model with $\tau = \frac{1}{\sqrt{m}}$ and $\bX = \sum_{k = 1}^{r^{\ast}} \sigma_{k} \bu_{k} \bv_{k}^{t}$. Let $\hat{\bX}^{f}$ be a spectral shrinker  of order $r \leq r^{\ast}$  that collapses the bulk to 0, such that  each function $f_{k}$, for $1 \leq k \leq r$, is continuously differentiable on $]c_{+},+\infty[$. Moreover, assume that $\sigma_{k} > c^{1/4}$ for all $1 \leq k \leq r$. Then, one has that, almost surely,
\begin{equation}
\lim_{n \to + \infty} \overline{\SURE} \left( \hat{\bX}^{f} \right)  = \sum_{k = 1}^{r} (  f_{k}( \rho(\sigma_{k}))  - \rho(\sigma_{k}) )^{2}  + 2  f_{k}(\rho(\sigma_{k}))  \left( \frac{\sigma_{k}^{2}(1+c) + 2c}{\sigma_{k}^{2}\rho(\sigma_{k})}  \right)   \label{eq:limSURE}
\end{equation}
\end{lem}

\noindent {\bf Asymptotically optimal shrinkage of singular values.}  Thanks to Lemma \ref{lem:limSURE}, one may determine an asymptotic optimal spectral shrinker as the one minimizing  $\lim_{n \to + \infty} \overline{\SURE} \left( \hat{\bX}^{f} \right)$. For this purpose, let us define the class of estimators
\begin{equation} \label{eq:hatXw}
\hat{\bX}^{r}_{w} =   \sum_{k = 1}^{r}    w_{k}  \tilde{\sigma}_{k}  \1_{\left\{ \tilde{\sigma}_{k} > c_{+}\right\}} \tilde{\bu}_{k} \tilde{\bv}_{k}^{t},
\end{equation}
where $1 \leq r \leq r^{\ast}$ is a given integer, and the $w_{k}$'s are positive weights. In practice, the estimator $\hat{\bX}^{r}_{w}$ is computed by replacing the bulk edge $c_+$ by its approximation $c_+^{n,m}  = 1 + \sqrt{\frac{n}{m}}$ in eq.~\eqref{eq:hatXw}. For moderate to large values of $n$ and $m$, the quantities $c_+$ and $c_+^{n,m}$ are very close, and this replacement does not change the numerical performances of $\hat{\bX}^{r}_{w}$.

Then, provided that $\sigma_{k} > c^{1/4}$ for all $1 \leq k \leq r$, it follows from Lemma \ref{lem:limSURE}   that
\begin{align*}
\lim_{n \to + \infty} \overline{\SURE} \left( \hat{\bX}^{r}_{w} \right)  = \sum_{k = 1}^{r} \rho^{2}(\sigma_{k})(  w_{k}  - 1 )^{2}  + 2  w_{k} \left( \frac{\sigma_{k}^{2}(1+c) + 2c}{\sigma_{k}^{2} }  \right).
\end{align*}
Differentiating the above  expression with respect to each weight $w_{k}$ leads to the following choice of asymptotically optimal weights
\begin{equation} \label{eq:wkast}
w_{k}^{\ast} = 1 -  \frac{\sigma_{k}^{2}(1+c) + 2c}{\sigma_{k}^{2} \rho^{2}(\sigma_{k})} \mbox{ for all } 1 \leq k \leq r.
\end{equation}
Therefore, if the singular values of the matrix $\bX$ to be estimated are sufficiently large (namely $\sigma_{k} > c^{1/4}$ for all $1 \leq k \leq r$), by using Proposition \ref{prop:sv} and eq.~\eqref{eq:wkast},  one has that an asymptotically optimal spectral shrinker of order $r \leq r^{\ast}$  is given by the  choice of functions
\begin{equation}
f_{k}^{\ast}(\rho(\sigma_{k}))   = \left\{
\begin{array}{cc}
 \left(  1 -  \frac{\sigma_{k}^{2}(1+c) + 2c}{\sigma_{k}^{2} \rho^{2}(\sigma_{k})} \right) \rho(\sigma_{k})  & \mbox{ if }  \rho(\sigma_{k})  > c_{+}, \\
0 & \mbox{ otherwise,}
\end{array}
\right.
\mbox{ for all } 1 \leq k \leq r. \label{eq:fast}
\end{equation}
Using, the relation \eqref{eq:link} one may also express the asymptotically optimal shrinking rule \eqref{eq:fast} either as a function of  $\rho(\sigma_{k})$ only,
\begin{equation}
\label{eq:fastrho}
f_{k}^{\ast}(\rho(\sigma_{k})) =
\left\{
\begin{array}{cc}
\frac{1}{\rho(\sigma_{k})} \sqrt{(\rho^{2}(\sigma_{k})-(c+1))^2-4c} & \mbox{ if }  \rho(\sigma_{k})  > c_{+}, \\
0 & \mbox{ otherwise.}
\end{array}
\right.
\end{equation}
or as  function  of $\sigma_{k}$
only (using that $\rho(\sigma_{k})  > c_{+}$ is equivalent to $\sigma_{k} > c^{1/4}$),
\begin{equation} \label{eq:fastsigma}
f_{k}^{\ast}(\rho(\sigma_{k})) =
\left\{
\begin{array}{cc}
 \frac{\sigma_{k}^{4} -c}{\sigma_{k} \sqrt{(1+\sigma_{k}^{2})(c+\sigma_{k}^{2})}} & \mbox{ if } \sigma_{k} > c^{1/4}, \\
0 & \mbox{ otherwise.}
\end{array}
\right.
\end{equation}

 Therefore, for spectral shrinker of order $r$, we remark that the shrinkage rule \eqref{eq:fastrho}   coincides with the rule  \eqref{eq:optshrinkGD} which has been obtained in \cite{GavishDonoho}.
 Similarly, when the quantity $f_{k}^{\ast}(\rho(\sigma_{k}))$   is expressed as a function of $\sigma_{k}$ only  in \eqref{eq:fastsigma},  then we retrieve the shrinking rule \eqref{eq:optshrinkRao} derived in  \cite{MR3200641}. Therefore, it appears that minimizing either  the asymptotic behavior of the $\SURE$, that is $\lim_{n \to + \infty} \overline{\SURE} \left( \hat{\bX}^{f} \right)$, or the limit of $\SE$ risk \eqref{eq:SE} leads to the same  choice of an asymptotically optimal spectral estimator. \\

\noindent{\bf Data-driven shrinkage of empirical singular values.} From the results in Section \ref{sec:optexp}, the principle of SURE minimisation leads to the following data-driven choice of spectral shrinker of order $r$ that collapses the bulk to 0
\begin{equation} \label{eq:hatXast}
\hat{\bX}^{r}_{w} =   \sum_{k = 1}^{r}    f_{k}(\tilde{\sigma}_{k}) \tilde{\bu}_{k} \tilde{\bv}_{k}^{t},
\end{equation}
where
$
 f_{k}(\tilde{\sigma}_{k})  = w_{k}(\bY) \tilde{\sigma}_{k} \1_{\left\{ \tilde{\sigma}_{k} > c_{+}\right\}}, \mbox{ for all } 1 \leq k \leq r,
$
with  $w_{k}(\bY)$ given by   \eqref{eq:optGaussSURE}.
From  Proposition \ref{prop:sv} and Proposition \ref{prop:sti} it follows that, if  $\sigma_{k} > c^{1/4}$,  then, almost surely,
\begin{align*}
\lim_{n \to + \infty}  f_{k}(\tilde{\sigma}_{k}) =  \left(  1 -  \frac{\sigma_{k}^{2}(1+c) + 2c}{\sigma_{k}^{2} \rho^{2}(\sigma_{k})} \right) \rho(\sigma_{k}), \mbox{ for all } 1 \leq k \leq r \leq r^{\ast}.
\end{align*}
Therefore, the data-driven spectral estimator $\hat{\bX}^{r}_{w}$ \eqref{eq:hatXast} asymptotically leads to the optimal shrinking rule of singular values given by  \eqref{eq:fast}  which has been obtained by minimizing the asymptotic behavior of the SURE.

Note that when $\tau \neq 1 /\sqrt{m}$, it suffices to replace the condition $\tilde{\sigma}_{k} > c_{+}$ by $\tilde{\sigma}_{k} > \tau (\sqrt{m} + \sqrt{n})$ in the definition of $\hat{\bX}^{r}_{w}$, which yields the  shrinking rule \eqref{eq:optGauss} stated in the introduction of this paper.

\section{Estimating active sets of singular values in exponential families}
\label{sec:bulk_exp_fam}

In this section, we propose to formulate a new Akaike information criterion (AIC) to select an appropriate set of singular values over which a shrinkage procedure might be applied. To this end, we shall consider the estimator
$
\tilde{\bX}^s = \sum_{k \in s} \tilde{\sigma}_k \tilde{\bu}_k \tilde{\bv}^t_k
$
defined for a subset $s \subseteq \mathcal{I} = \{1,2,\ldots,\min(n, m)\}$,
and we address the problem of selecting an optimal subset $s^\star$ from the data $\bY$.

In the case of Gaussian measurements, the shrinkage estimators that we use in our numerical experiments are of the form
$
\hat{\bX}^f = \sum_{k \in s^\star} f_{k}(\tilde{\sigma}_k) \tilde{\bu}_k \tilde{\bv}^t_k
$
where
\begin{align*}
s^\star = \{k \; ; \; \tilde{\sigma}_k > c_+^{n,m}  \} \mbox{ with } c_+^{n,m}  = 1 + \sqrt{\frac{n}{m}},
\end{align*}
for some (possibly data-dependent) shrinkage functions $f_{k}$. The set $s^\star$ 
is based on the knowledge of an approximation $c_+^{n,m}$ of the bulk edge $c_{+}$. Thanks to Proposition \ref{prop:sv}, the bulk edge  $c_{+}$ is interpreted as the threshold which allows to distinguish the locations of significant singular values in the data from those  due to the presence of additive noise. 
Interestingly, the following result shows that  the active set $s^\star$ may be interpreted through the prism of model selection using the minimisation of a penalized log-likelihood criterion.

\begin{prop}\label{prop:bulk_aic_gaussian}
  Assume that $\bY = \bX + \bW$ where the entries of $\bW$ are iid Gaussian variables with zero mean and standard deviation  $\tau = 1/\sqrt{m}$.  Then, we have 
\begin{equation} \label{eq:AICGauss}
 s^\ast =  \uargmin{s \subseteq \mathcal{I} } m \| \bY - \tilde{\bX}^{s} \|^{2}_{F} + 2 |s| p_{n,m} \mbox{ with } p_{n,m} =   \left(  \frac{1}{2} \left(\sqrt{m} + \sqrt{n}\right)^2 \right),
\end{equation}
where $\tilde{\bX}^{s} = \sum_{k \in s}  \tilde{\sigma}_k  \tilde{\bu}_k \tilde{\bv}^t_k$  for $s \in \mathcal{I} = \{1,2,\ldots,\min(n, m)\}$, and $|s|$ is the cardinal of  $s$.
\end{prop}

\begin{proof}
We  remark that $ \bY - \tilde{\bX}^s  = \sum_{k \notin s} \tilde{\sigma}_k \tilde{\bu}_k \tilde{\bv}^t_k$.
  It results that
  \begin{align} \label{eq:opstar}
    m \norm{ \bY - \tilde{\bX}^s}^2_{F} + 2 |s| p_{n,m}
    =
    m \sum_{k \notin s} \tilde{\sigma}_k^2
    + 2 |s| p_{n,m}
    =
    \sum_{k=1}^{n} \left\{
    \begin{array}{ll}
      m \tilde{\sigma}_{k}^2 & \text{if } k \notin s\\
      2 p_{n,m} & \text{otherwise}
    \end{array}\right.
    .
  \end{align}
Using that  $\sqrt{2 p_{n,m} / m} = c_+^{n,m}$, it follows that the set $s^\star = \{k \; ; \; \tilde{\sigma}_k > c_+^{n,m}\}$ is by definition such that $k \in s^\star$ if and only if $2 p_{n,m} < m \tilde{\sigma}_k^2$. Therefore, by \eqref{eq:opstar}, the criterion $s \mapsto m \norm{ \bY - \tilde{\bX}^s}^2_{F} + 2 |s| p_{n,m}$ is minimum at $s = s^\star$ which
concludes the proof.
\end{proof}

In the model $\bY =  \bX + \bW$, where the entries of $\bW$ are iid  Gaussian variables with zero mean and variance $\tau^{2}$,  it is well known that the degrees of freedom (DOF)  of a given estimator $\hat{\bX}$  is defined as
\begin{align*}
\DOF( \hat{\bX} ) = \frac{1}{\tau^{2}} \sum_{i=1}^{n} \sum_{j=1}^{m} \cov ( \hat{\bX}_{ij} ,  \bY_{ij}) = \frac{1}{\tau^{2}} \sum_{i=1}^{n} \sum_{j=1}^{m} \E [ \hat{\bX}_{ij}  \bW_{ij}].
\end{align*}
The DOF is  widely used in statistics to define various criteria for model selection among a collection of estimators, see e.g.\ \cite{Efron04}. In low rank matrix denoising, the following proposition shows that it is possible to derive the asymptotic behavior of the DOF of spectral estimators.

\begin{prop} \label{prop:dof}
Assume that $\bY  = \bX + \bW$ is a random matrix sampled from the Gaussian spiked population model with $\tau = \frac{1}{\sqrt{m}}$ and $\bX = \sum_{k = 1}^{r^{\ast}} \sigma_{k} \bu_{k} \bv_{k}^{t}$.  Let $\hat{\bX}^{f}$ be a spectral shrinker  of order $r \leq r^{\ast}$  that collapses the bulk to 0, such that  each function $f_{k}$, for $1 \leq k \leq r$, is continuously differentiable on $]c_{+},+\infty[$. Moreover, assume that $\sigma_{k} > c^{1/4}$ for all $1 \leq k \leq r$. Then, one has that, almost surely,
\begin{align*}
\lim_{n \to + \infty}  \frac{1}{m} \DOF( \hat{\bX}^{f}) = \sum_{k = 1}^{r} \frac{f_{k}(\rho(\sigma_{k}))}{\rho(\sigma_{k})} \left(1 + c +    \frac{2 c }{\sigma_{k}^{2}} \right).
\end{align*}
\end{prop}
\begin{proof}
Thanks to the derivation of the SURE in \cite{MR630098} and formula \eqref{eq:divf} on the divergence of spectral estimators, one has that
\begin{align*}
\DOF( \hat{\bX}^{f}) =  \E \left[ \div \hat{\bX}^{f} \right] = \E \left[ |m - n| \sum_{k = 1}^{r} \frac{f_{k}(\tilde{\sigma}_{k})}{\tilde{\sigma}_{k}} + \sum_{k = 1}^{r}f_{k}'(\tilde{\sigma}_{k}) + 2 \sum_{k = 1}^{r} f_{k}(\tilde{\sigma}_{k}) \sum_{\ell =1 ; \ell \neq k}^{n} \frac{\tilde{\sigma}_{k}}{\tilde{\sigma}_{k}^{2} - \tilde{\sigma}_{\ell}^{2} } \right].
\end{align*}
By  Proposition \ref{prop:sv},  Proposition \ref{prop:sti}, and our assumptions on  the $f_{k}$'s, one has that, almost surely,
\begin{align*}
\lim_{n \to + \infty} \frac{1}{m} \div \hat{\bX}^{f} =  \sum_{k = 1}^{r} \frac{f_{k}(\rho(\sigma_{k}))}{\rho(\sigma_{k})} \left(1 + c +    \frac{2 c }{\sigma_{k}^{2}} \right).
\end{align*}
which completes the proof.
\end{proof}
Hence, in the Gaussian spiked population model, by Proposition \ref{prop:dof} and using that $\sigma_{k}^{2} > \sqrt{c}$ for all $1 \leq k \leq r$, it follows that if $s \subseteq \{1,\ldots,r \}$ then
\begin{equation}
\lim_{n \to + \infty}  \frac{1}{m} \DOF( \tilde{\bX}^{s}) = |s| \left(1 + c +    \frac{2 c }{\sigma_{k}^{2}} \right) \leq  |s| \left(1 + \sqrt{c} \right)^2 = |s| c_{+}^{2}.
\end{equation}
Hence, the quantity $2 |s| \left(  \frac{1}{2} \left(\sqrt{m} + \sqrt{n}\right)^2 \right)$ is asymptotically an upper bound  of $\DOF( \tilde{\bX}^{s})$ (when normalized by $1/m$) for any given set $s \subseteq \{1,\ldots,r \}$.

Let us now consider the more general case where the entries of $\bY$ are sampled from an exponential family. To the best of our knowledge, extending the notion of the bulk edge to non-Gaussian data sampled from an exponential family has not been considered so far in the literature on random matrices and low rank perturbation model. Therefore, except in the Gaussian case, it is far from being trivial to find an appropriate threshold value $\bar{c}$ to define an active set  in the form $\bar{s} = \{k \; ; \; \tilde{\sigma}_k > \bar{c}\}$.

Nevertheless, to select an appropriate active set of singular values, we introduce the following criterion that is inspired by the previous results on the DOF of the estimator $\tilde{\bX}^{s}$ in the Gaussian case and the statistical literature on the well known AIC for model selection \cite{akaike1974new}.

\begin{defin}
{\it
The AIC associated to $\tilde{\bX}^s = \sum_{k \in s} \tilde{\sigma}_k \tilde{\bu}_k \tilde{\bv}^t_k$ is
\begin{align}
  \label{eq:aic}
  \mathrm{AIC}(\tilde{\bX}^s) = -2 \log q(\bY; \tilde{\bX}^s) + 2 |s| p_{n,m}
  \quad
  \text{with}
  \quad p_{n,m} = \frac{1}{2} \left(\sqrt{m} + \sqrt{n}\right)^2~.
\end{align}
where $|s|$ is the cardinal of $s$, and $q(\bY; \tilde{\bX}^s) = \prod_{i=1}^{n} \prod_{j=1}^{m} q( \bY_{ij} ; \tilde{\bX}^s_{ij})$ is the likelihood of the data in the general form \eqref{eq:expfam} at the estimated parameters $ \bX_{ij} =   \tilde{\bX}^s_{ij}$.
}
\end{defin}

In the above definition of $\mathrm{AIC}(\tilde{\bX}^s)$, the quantity $2 |s| p_{n,m}$ is an approximation of the degree of freedom of
$\tilde{\bX}^s$, i.e., of the numbers of its free parameters as it is justified by Proposition \ref{prop:dof} in the case of Gaussian measurements.
The AIC allows us to define an optimal subset of active variables as
\begin{align*}
s^\ast  =  \uargmin{ s \subseteq \mathcal{I}  }
  \mathrm{AIC}(\tilde{\bX}^s).
\end{align*}
For Gaussian measurements, Proposition \ref{prop:bulk_aic_gaussian} gives the value of the optimal set  $s^\ast$ in a closed-form. 

Following the arguments in Section \ref{sec:optexp}, for Gamma or Poisson measurements and for a given subset $s$, we consider the estimator 
\begin{align}  
  \tilde{\bX}^s_{\epsilon} = 
  \max\left[
    \sum_{k \in s} \tilde{\sigma}_{k} \tilde{\bu}_{k} \tilde{\bv}_{k}^{t},
    \varepsilon
    \right],
\end{align}
when $\epsilon > 0$ is an {\it a priori} value to satisfy the positivity constraint on the entries of an estimator in this setting. However, contrary to the case of Gaussian noise, the search of an optimal subset $s^\star \subset \uargmin{ s \subseteq \mathcal{I}  }
  \mathrm{AIC}(\tilde{\bX}^s_{\epsilon})$ becomes a combinatorial problem in this context. In our numerical experiments,  we thus choose to construct an approximation $\tilde{s}$ of $s^\star$ with a greedy search strategy that reads as follows
\begin{align} \label{eq:tildes}
  \tilde{s} =
  \mathcal{I} \setminus \left\{ k \in \mathcal{I} \; ; \;
  \mathrm{AIC}(\tilde{\bX}^{\mathcal{I} \backslash \{k\}}_{\epsilon}) \leq \mathrm{AIC}(\tilde{\bX}^{\mathcal{I}}_{\epsilon})
  \right\}.
\end{align}

For Gaussian measurements, $\tilde{s} = s^\star$
since the optimisation problem \eqref{eq:tildes} becomes separable.
In our numerical experiments, we have found that $\tilde s$ selects
a relevant set of active singular values which separates well the structural content of $\bX$
while removing most of the noise component. Further details are given in Section \ref{sec:num} below.


 For Gaussian noise, the computation of the active set $s^{\ast}$ of singular values may also be  interpreted as a way to estimate the unknown rank $r^{\ast}$ of the signal matrix $\bX$. In this setting, one has that $s^\star = \{k \; ; \; \tilde{\sigma}_k > c_+^{n,m}\}$ which suggests the choice
\begin{equation}
\hat{r} = \max \{k \; ; \; \tilde{\sigma}_k > c_+^{n,m}\}, \label{eq:estrankbulk}
\end{equation}
as an estimator of $r^{\ast}$.

There exists an abundant literature of the problem of estimating the rank of an empirical covariance matrix for the purpose of selecting the appropriate number of significant components to be kept  in PCA or factor analysis. It is much beyond the scope of this paper to give an overview of this topic. We point to the review in \cite{MR2036084} for a summary of existing methods to determine the number of components in PCA that are grouped into three categories: subjective methods, distribution-based test tools, and computational procedures.
 For recent contributions in the matrix denoising model \eqref{eq:deformodel} with Gaussian noise, we refer to the  works \cite{CTT2015,GDIEEE14} and references therein.  For example for Gaussian data with know variance $\tau^{2} = 1/m$, Eq.~(11) in  \cite{GDIEEE14}  on optimal hard thresholding of singular values suggest to take
\begin{equation}
\hat{r} = \max \{k \; ; \; \tilde{\sigma}_k > \lambda(c) \}, \mbox{ with }  \lambda(c) = \sqrt{2 (c+1) + \frac{8 c}{(c+1) + \sqrt{c^2+14c+1}}}, \label{eq:estrankHT}
\end{equation}
as a simple method to estimate  the rank. It should be remarked that the problem of estimating the true rank $r^{\ast}$ of $\bX$ in model  \eqref{eq:deformodel}  is somewhat ill-posed as, in the Gaussian spiked population model, Proposition \ref{prop:sv} implies that one may only expect to estimate the so-called effective rank $r_{\rm{eff}} = \max \{k \; ; \;  \sigma_k > c^{1/4}\}$  (see e.g.\ Section II.D in \cite{MR3200641}).

  In our numerical experiments, we shall compare different choices  for the active set of singular values of the form $\hat{s} =\{1,\ldots,\hat{r}\}$ where $\hat{r}$ is either given by \eqref{eq:estrankbulk},  \eqref{eq:estrankHT}, or by the ``oracle choices'' $\hat{r} = r^{\ast}$ and $\hat{r} = r_{\rm{eff}}$.

  Other methods based on hypothesis testing \cite{CTT2015} 
  could be used for rank estimation in the Gaussian model \eqref{eq:deformodel}, but it is beyond the purpose of this paper to give a detailed comparison.

For Poisson or Gamma noise, it is more difficult to interpret the computation of $s^{\ast}$ as a way to estimate the rank of $\bX$ since, in our numerical experiments, we have found that  the cardinality of $s^{\ast}$ is generally not equal to $\max \{k \; ; k \in s^{\ast} \}$. Moreover, to the best of our knowledge, there is not so much work on the estimation of the true rank of a noisy matrix beyond the Gaussian case. Therefore, we have not included a numerical comparison with other methods for the choice of the active set of singular values in these two cases.

\section{Numerical experiments} \label{sec:num}

In this section, we assess of the performance of data-driven srhinkage rules
under various numerical experiments involving Gaussian, Gamma and Poisson measurements.


\subsection{The case of a signal matrix of rank one}

We consider the simple setting where the rank $r^{\ast}$ of the matrix $\bX$ is known and equal to one meaning that
\begin{align*}
\bX = \sigma_{1} \bu_{1} \bv_{1}^{t},
\end{align*}
where $\bu_{1} \in \R^{n} $ and $\bv_{1} \in \R^{m}$ are vectors with unit norm that are fixed in this numerical experiment, and $\sigma_{1}$ is a positive real that we will let varying. We also choose to fix $n = m = 100$, and so to take $c = \frac{n}{m} = 1$ and $c_{+} = 2$. For the purpose of sampling data from Gamma and Poisson distribution, we took singular vectors $\bu_{1}$ and $\bv_{1}$ with positive entries. The $i$-th entry (resp.\  $j$-th entry) of $\bu_{1}$ (resp.\ $\bv_{1}$) is chosen to be proportional to $1 -(i/n-1/2)^2$ (resp.\ $1 -(j/m-1/2)^2$).
Let $\bY = \sum_{k = 1}^{\min(n,m)} \tilde{\sigma}_{k} \tilde{\bu}_{k} \tilde{\bv}_{k}^{t}$ be an $n \times m$ matrix whose entries are sampled from  model \eqref{eq:expfam}
and then satisfying $\E[\bY] = \bX$.

\subsubsection*{Gaussian measurements}

We first consider the case of Gaussian measurements,
where
$
\bY = \bX + \bW
$
with $\E[\bW_{ij}] = 0$, $\mathrm{Var}(\bW_{ij}) =\tau^2$ with $\tau = \frac{1}{\sqrt{m}}$.
In this context, we compare the following spectral shrinkage estimators:
\begin{description}
\item[$\bullet$] Rank-1 PCA shrinkage
\begin{flalign*}
\hat{\bX}^1 = \tilde \sigma_{1} \tilde{\bu}_{1} \tilde{\bv}_{1}^{t}, &&
\end{flalign*}
\item[$\bullet$] Rank-1 SURE-driven soft-thresholding
\begin{flalign*}
\hat{\bX}_{\mathrm{soft}}^1 = \hat{\sigma}_1 \tilde{\bu}_{1} \tilde{\bv}_{1}^{t}
\quad \text{with} \quad
\hat{\sigma}_1 =
( \tilde \sigma_{1} - \lambda(\bY))_{+}, &&
  \end{flalign*}
\item[$\bullet$] Rank-1 asymptotically optimal shrinkage proposed in \cite{MR3200641} and \cite{GavishDonoho}
  \begin{flalign*}
    \hat{\bX}_{\ast}^{1} = \hat{\sigma}_{1} \tilde\bu_{1} \tilde\bv_{1}^{t}
    \quad \text{with} \quad
    \hat{\sigma}_{1} =
    \sqrt{ \tilde{\sigma}_{1}^2 - 4} \; \1_{\left\{ \tilde{\sigma}_{1} > 2 \right\} }, &&
  \end{flalign*}
  \item[$\bullet$] Rank-1 SURE-driven weighted estimator that we have derived in Section \ref{sec:optexp}
\begin{flalign*}
\hat{\bX}_w^1 = \hat{\sigma}_1 \tilde{\bu}_{1} \tilde{\bv}_{1}^{t}
\quad \text{with} \quad
\hat{\sigma}_1 = \left( 1 - \frac{1}{ \tilde{\sigma}_{1}^2}  \left(  \frac{1}{m}+  \frac{2}{m}    \sum_{\ell = 2}^{n} \frac{ \tilde{\sigma}_{1}^{2}}{\tilde{\sigma}_{1}^{2} - \tilde{\sigma}_{\ell}^{2} } \right) \right)_{+} \tilde \sigma_1 \1_{\left\{ \tilde{\sigma}_{1} > 2 \right\} }, &&
\end{flalign*}
\end{description}
where the above formula follows from the results in Section \ref{sec:SUREGauss} using that $c = 1$ and $c_{+} = 2$ in these numerical experiments, and where, for the soft-thresholding,
the value $\lambda(\bY) > 0$ is obtained by a numerical solver in order to minimize the $\SURE$.
As a benchmark, we will also consider the oracle
estimator $\bX_*^1$ that performs shrinkage by using
the knowledge of the true singular-value $\sigma_1$ defined as
\begin{align*}
{\bX}_{\ast}^{1} = \hat{\sigma}_1 \tilde\bu_{1} \tilde\bv_{1}^{t}
\quad \text{with} \quad
\hat{\sigma}_1 = \sqrt{ \rho(\sigma_{1})^{2} - 4} \; \1_{\left\{ \rho(\sigma_{1}) > 2\right\}} 
\end{align*}
which corresponds to  the asymptotically optimal shrinking rule \eqref{eq:fastrho}  as a function of  $\rho(\sigma_{k})$  in the setting $c = 1$ and $c_+ = 2$.
Note that form the formula above $\hat w_1 = \hat{\sigma_1} / \tilde{\sigma_1}$
is necessary in the range $[0, 1]$ for all considered estimators.


In Figure \ref{fig:GaussQuantities}, we compare the
estimated singular-values $\hat{\sigma}_1$
and the estimated weights $\hat{w}_1 = \hat{\sigma}_1 / \tilde{\sigma}_1$
as functions of $\sigma_1$ for the four aforementionned estimators.
Because all estimators are subject to noise variance,
we display, for all estimators,
the median values and the 80\% confidence intervals
obtained from $M= 100$ noise realizations.
It can be seen that the median curves for the eigenvelues and the weights of $\hat{\bX}_{w}^{1}$ and $\hat{\bX}_{\ast}^{1}$ coincide (up to variations that are slightly larger for the former) which is in agreement with the asymptotic analysis of shrinkage rules that has been carried out in Section \ref{sec:SUREGauss}.
Spectral estimator obtained by SURE-driven soft-thresholding also leads to an optimal shrinkage rule. 

\begin{figure}[!t]
\centering%
\subfigure[]{\includegraphics[width=0.32\linewidth]{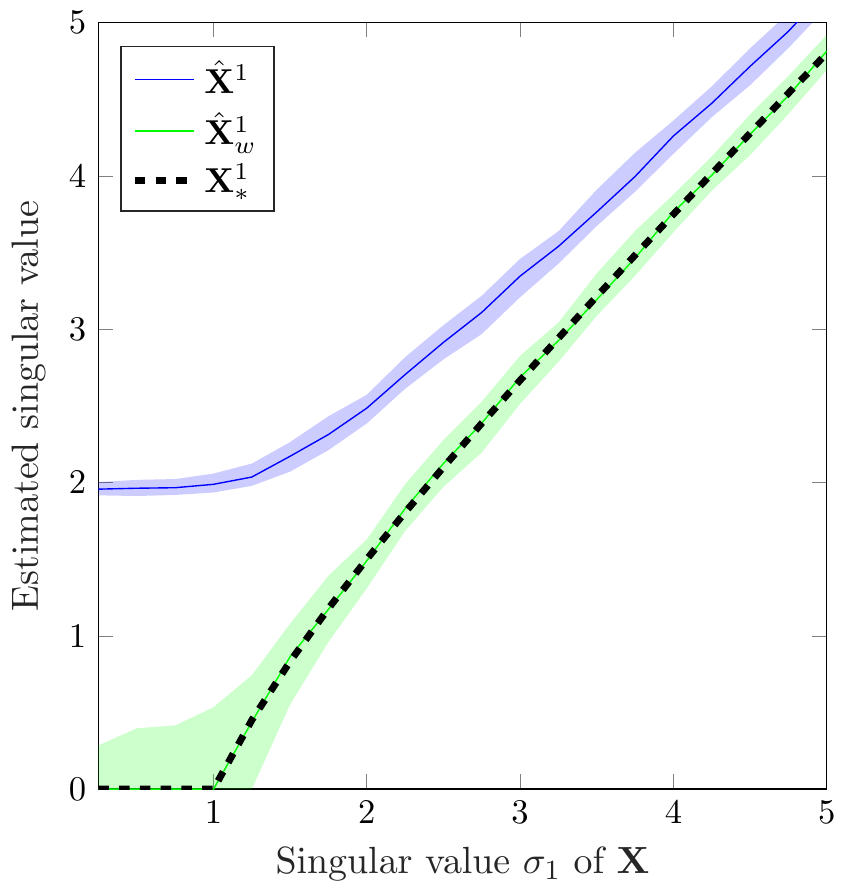}}\hfill%
\subfigure[]{\includegraphics[width=0.32\linewidth]{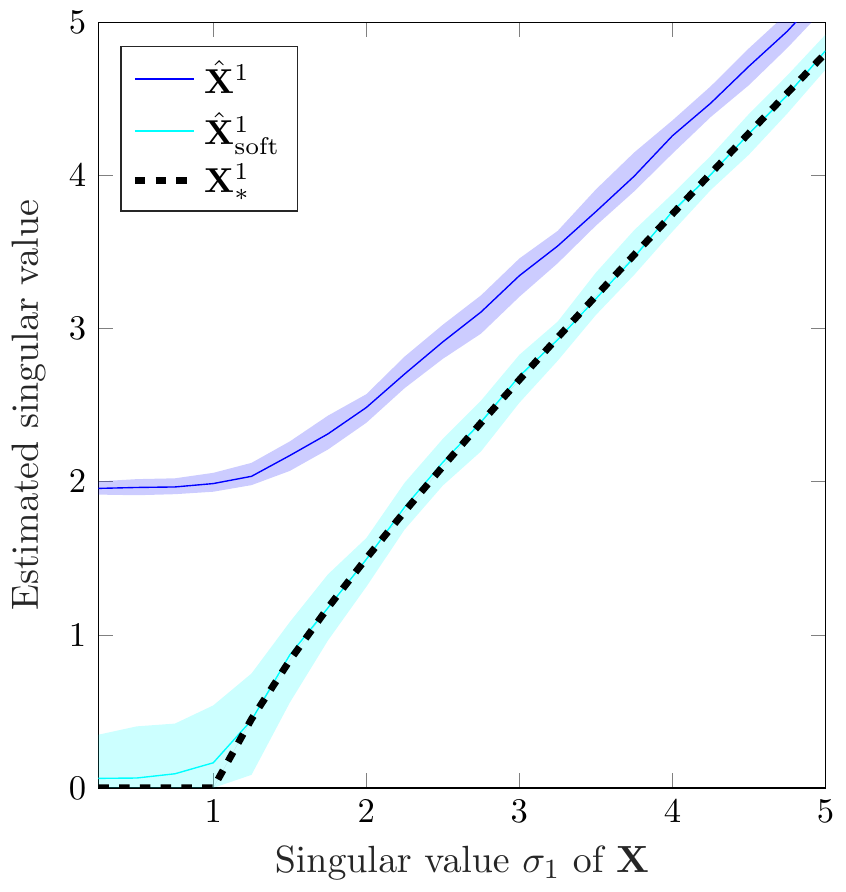}}\hfill%
\subfigure[]{\includegraphics[width=0.32\linewidth]{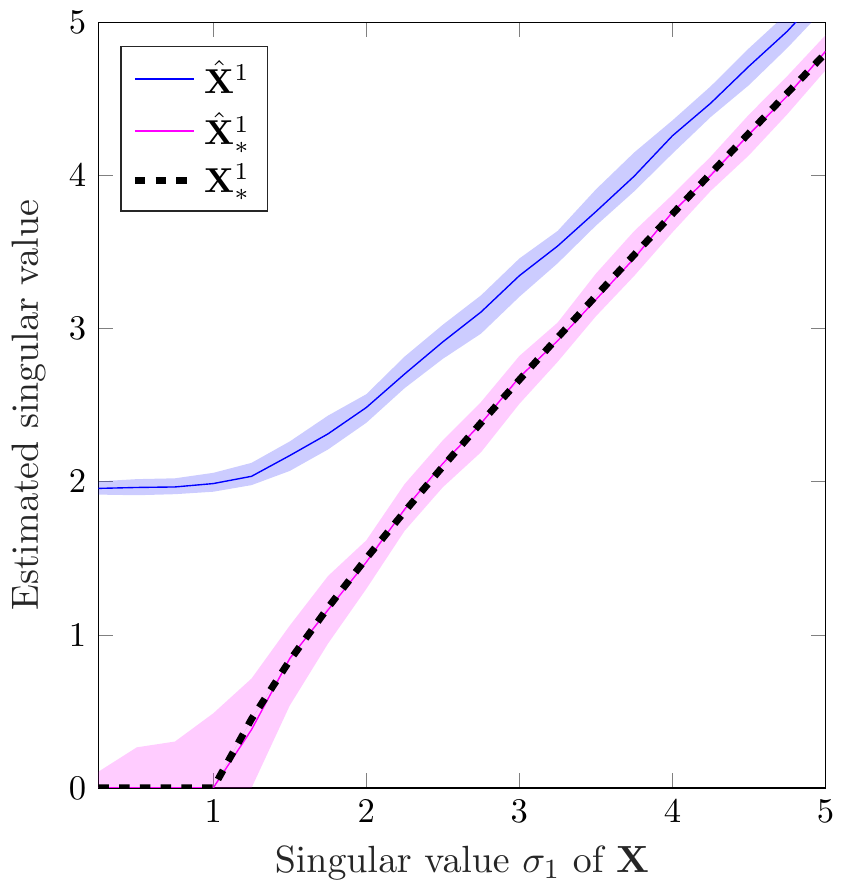}}\\
\subfigure[]{\includegraphics[width=0.32\linewidth]{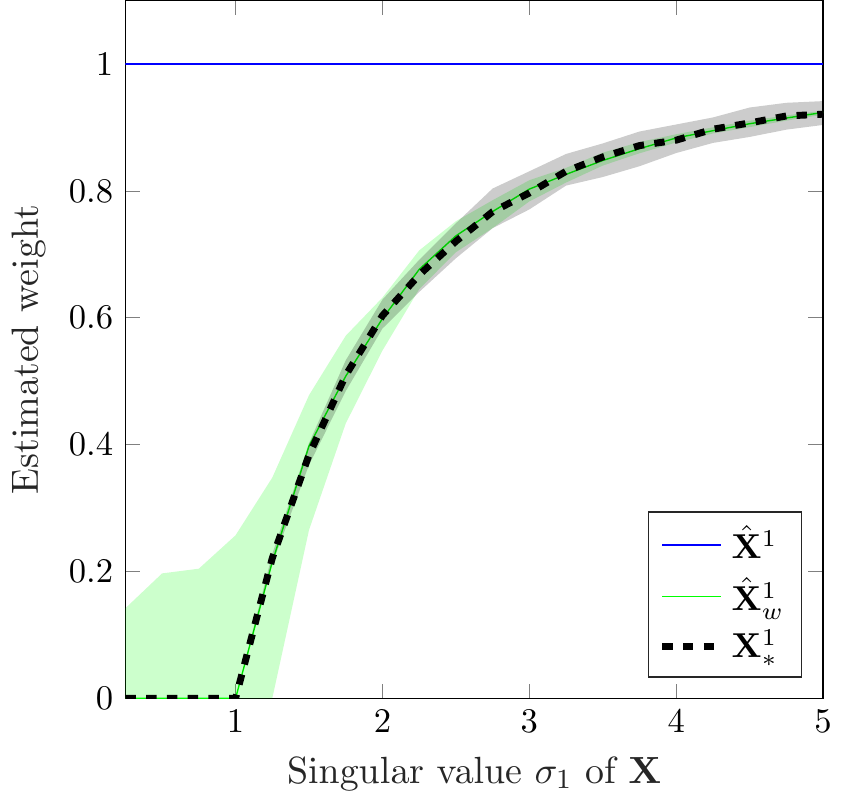}}\hfill%
\subfigure[]{\includegraphics[width=0.32\linewidth]{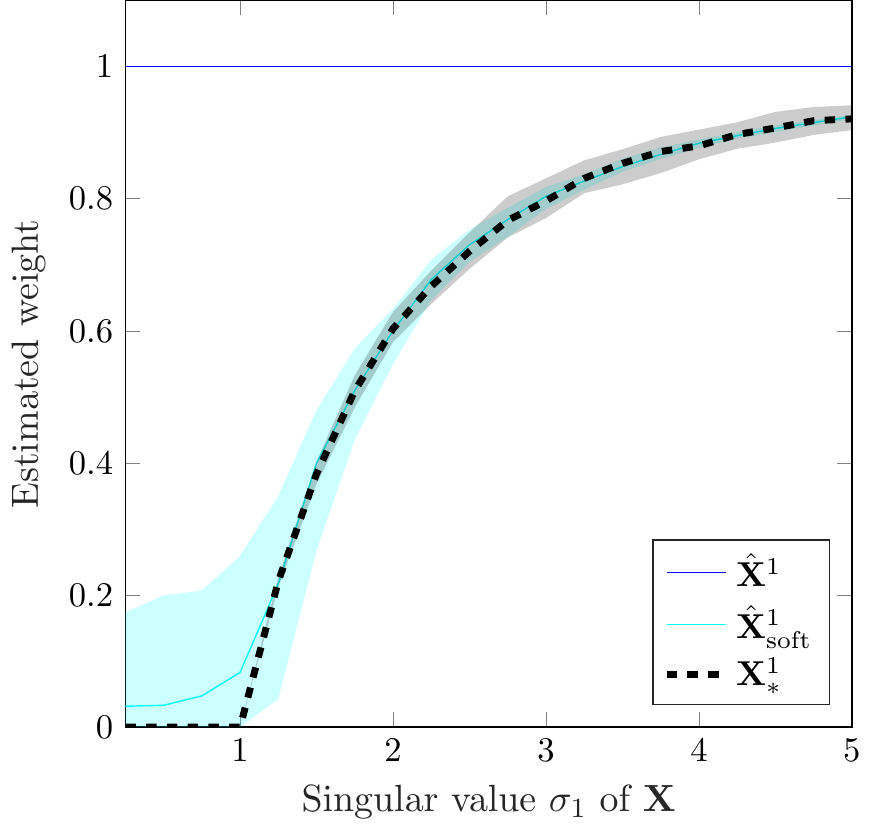}}\hfill%
\subfigure[]{\includegraphics[width=0.32\linewidth]{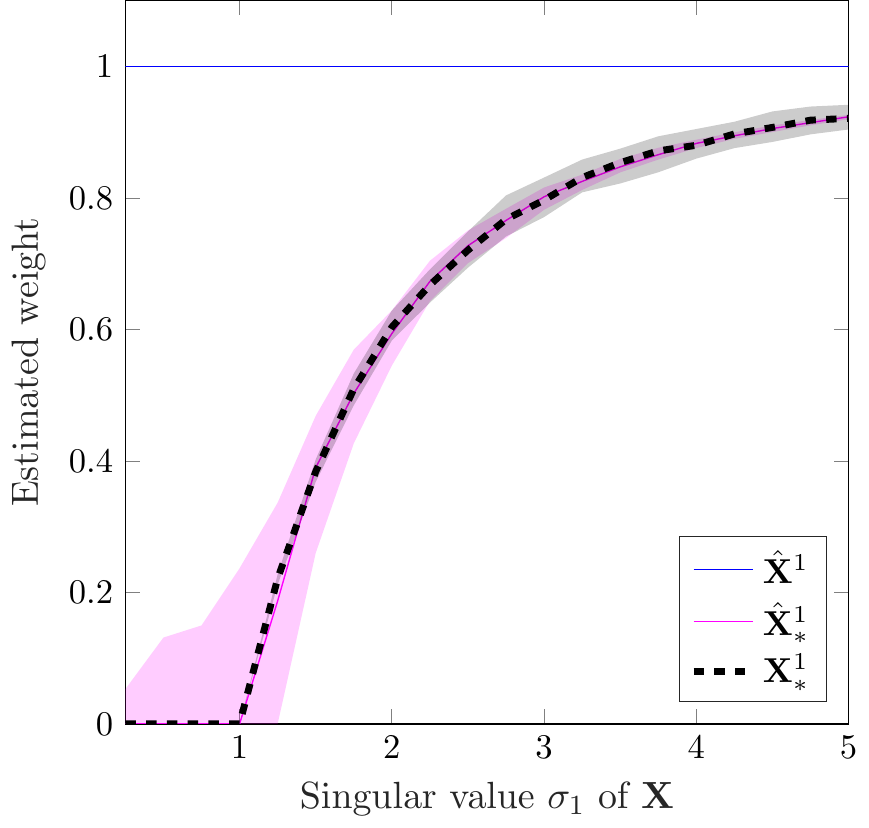}}
\caption{The case of Gaussian measurements with $m=n=100$.
  Estimated first singular value $\hat \sigma_1$ as a function of the true underlying one $\sigma_1$,
  for (a) our proposed estimator $\hat{\bX}_w^1$, (b) the soft-thresholding $\hat{\bX}_{\mathrm{\mathrm{soft}}}^1$ and (c) the asymptotical one $\hat{\bX}_{*}^1$.
  All of them are compared to the first singular value $\tilde \sigma_1$ of $\bY^1$ and the one of the oracle asymptotical estimator $\bX_{*}^1$.
  (c,d,e) Same but for the corresponding weight $\hat w_1 = \hat \sigma_1 / \tilde \sigma_1$.
  Curves have been computed on $M=100$ noise realizations, only the median and an 80\% confidence interval
  are represented respectively by a stroke and a shadded area of the same color.
  }
  \label{fig:GaussQuantities}
\end{figure}

In Figure \ref{fig:GaussRisk}, for each of the four spectral estimators above, we  display
for $M=100$ noise realizations, as functions of $\sigma_1$, the following normalized MSE
\begin{align*}
\NMSE(\hat{\bX})  =  \frac{\| \hat{\bX} - \bX \|^2_{F}}{ \| \bX \|^{2}_{F}}.
\end{align*}
The normalized MSE of the estimators $\hat{\bX}_{\mathrm{soft}}^{1}, \hat{\bX}_{\ast}^{1}$ and $\hat{\bX}_{w}^{1}$ are the same for values of $\sigma_1$ larger than $c^{1/4} = 1$, and they only differ for values of $\sigma_1$ close or below the threshold $c^{1/4} = 1$ (corresponding to values of $\rho(\sigma_1)$ below the bulk edge $c_+ = 2$).
More remarkably, above $c^{1/4} = 1$, they offer similar NMSE values to the oracle shrinkage estimator ${\bX}_{\ast}^{1}$, not only in terms of median but also
in terms of variability, as assessed by the confidence intervals.
The performances of the estimator $\hat{\bX}^{1}$ (standard PCA) are clearly poorer. These numerical experiments also illustrate that, for finite-dimensional low rank matrix denoising with $r^{\ast} = 1$, data-driven spectral estimators  obtained by minimizing a SURE criterion achieve performances that are similar to asymptotically optimal shrinkage rules.

\begin{figure}[!t]
\centering
\subfigure[]{\includegraphics[height=0.31\linewidth]{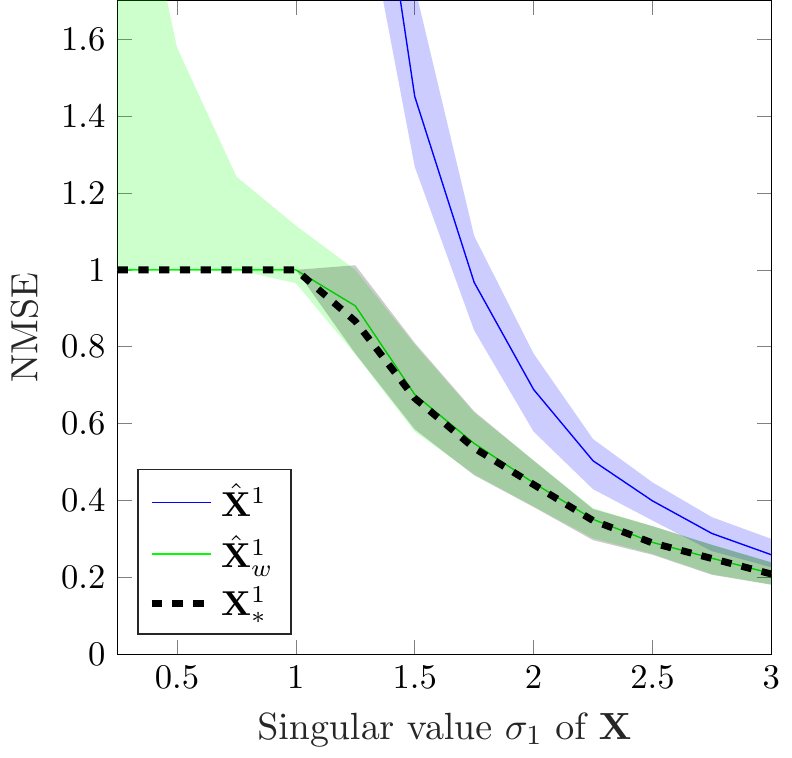}}\hfill%
\subfigure[]{\includegraphics[height=0.31\linewidth]{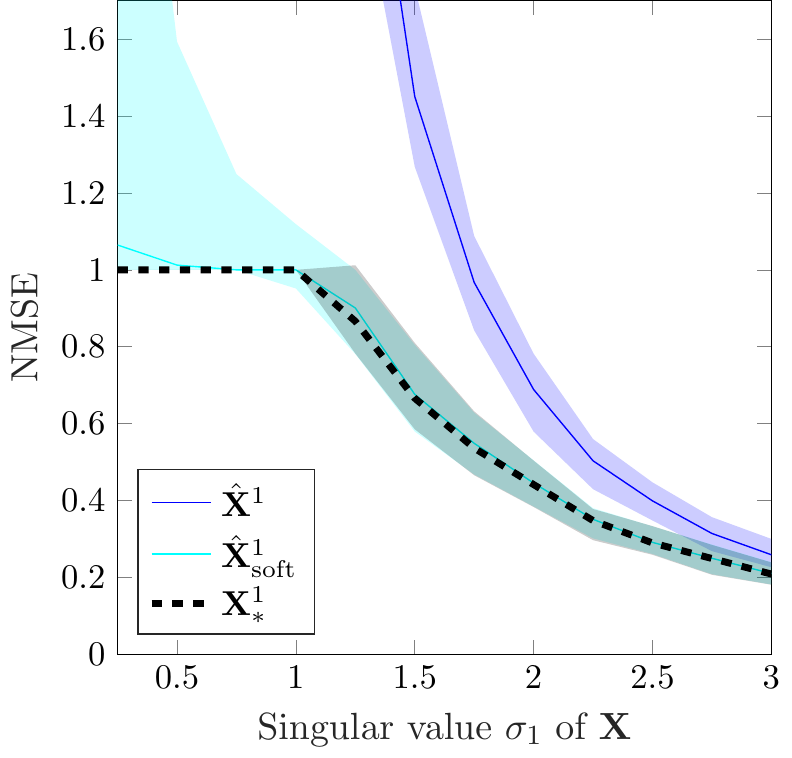}}\hfill%
\subfigure[]{\includegraphics[height=0.31\linewidth]{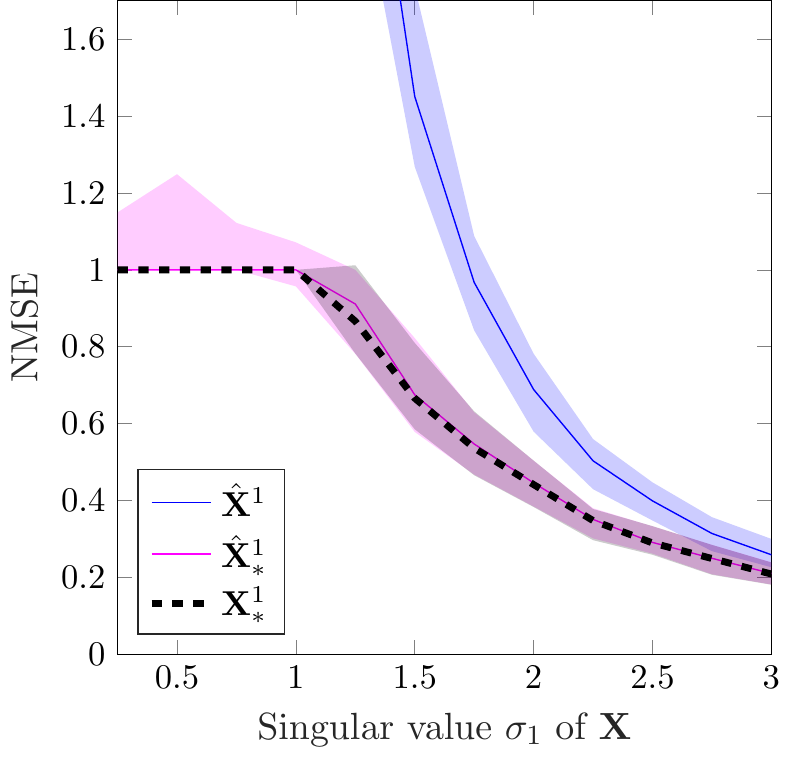}}\\
\caption{
  Same as Fig.~\ref{fig:GaussQuantities} but for the normalized MSE
  of the corresponding estimators.
}
 \label{fig:GaussRisk}
\end{figure}

\begin{figure}[!t]
\centering%
\subfigure[]{\includegraphics[height=0.24\linewidth]{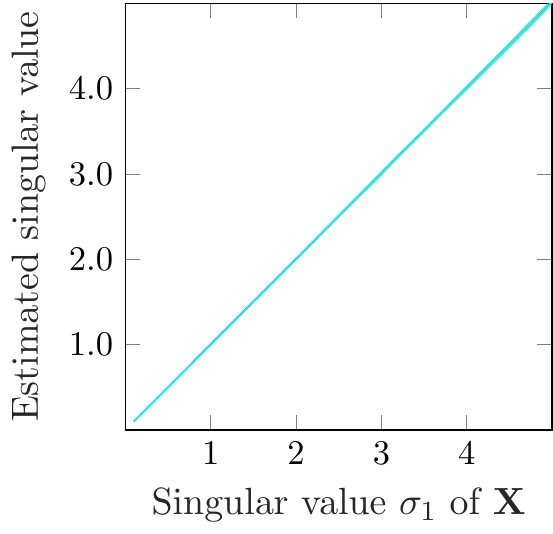}}\hfill%
\subfigure[]{\includegraphics[height=0.24\linewidth]{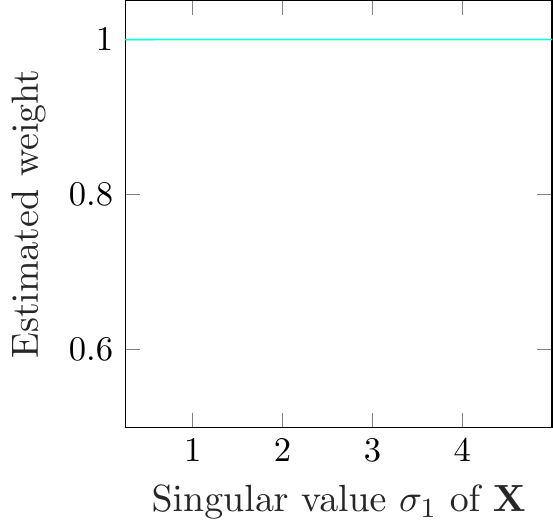}}\hfill%
\subfigure[]{\includegraphics[height=0.24\linewidth]{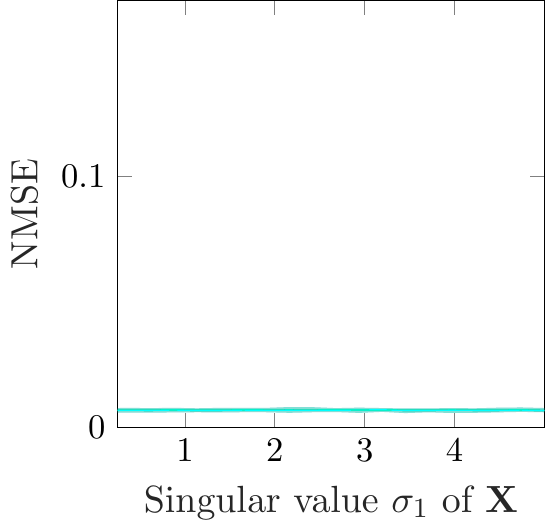}}\hfill%
\subfigure[]{\includegraphics[height=0.24\linewidth]{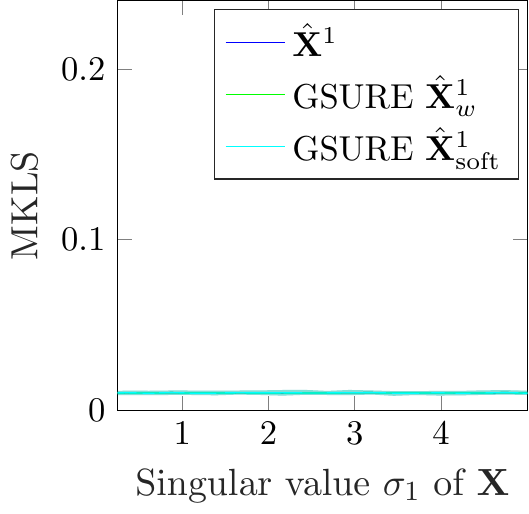}}\\
\subfigure[]{\includegraphics[height=0.24\linewidth]{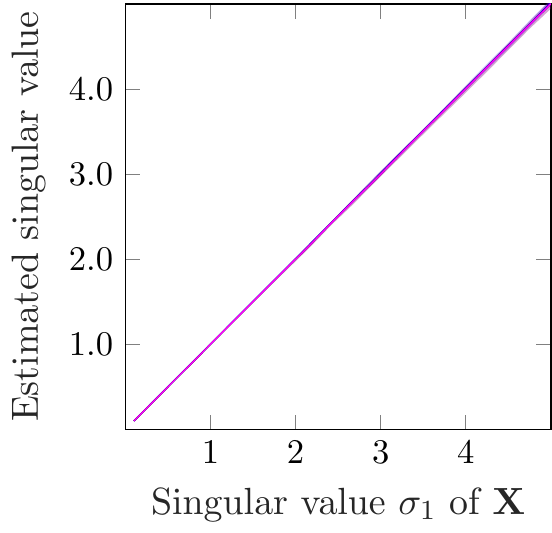}}\hfill%
\subfigure[]{\includegraphics[height=0.24\linewidth]{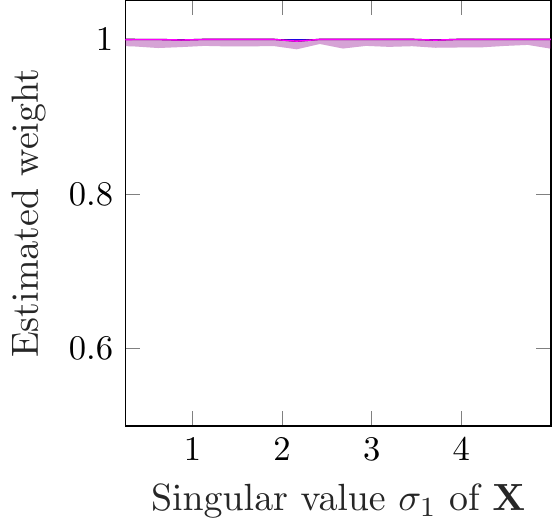}}\hfill%
\subfigure[]{\includegraphics[height=0.24\linewidth]{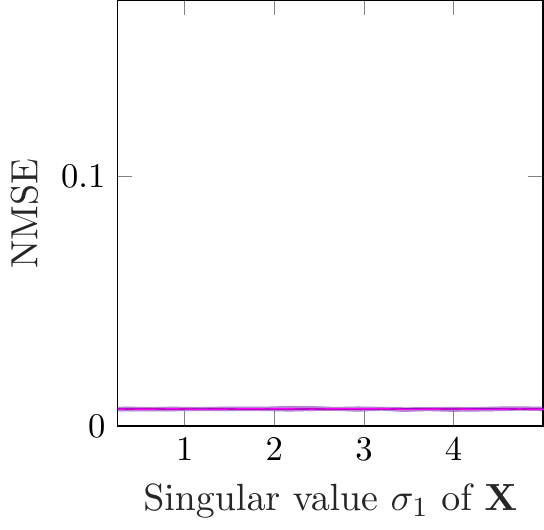}}\hfill%
\subfigure[]{\includegraphics[height=0.24\linewidth]{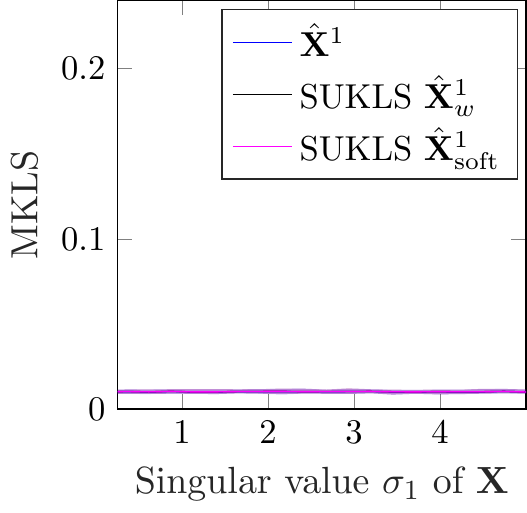}}\\
\caption{The case of Gamma measurements with $m=n=100$.
  (a) Estimated first eigenvalue $\hat \sigma_1$ as a function of the true underlying one $\sigma_1$
  for our proposed estimator $\hat{\bX}_w^1$ and the soft-thresholding $\hat{\bX}_{\mathrm{\mathrm{soft}}}^1$
  when both are guided by the GSURE.
  Both of them are compared to the first singular value $\tilde \sigma_1$ of $\bY^1$.
  Same but for (b) the corresponding weights $\hat w_1 = \hat \sigma_1 / \tilde \sigma_1$,
  (c) the NMSE risk and (d) the MKLS risk.
  (e-h) Exact same esperiments but when our proposed estimator and
  the soft-thresholding are both guided by SUKLS.
  Curves have been computed on $M=100$ noise realizations, only the median and an 80\% confidence interval
  are represented respectively by a stroke and a shadded area of the same color.
} \label{fig:GSURE}
\label{fig:SUKLS}
\end{figure}

\begin{figure}[!t]
\centering
\subfigure[]{\includegraphics[height=0.24\linewidth]{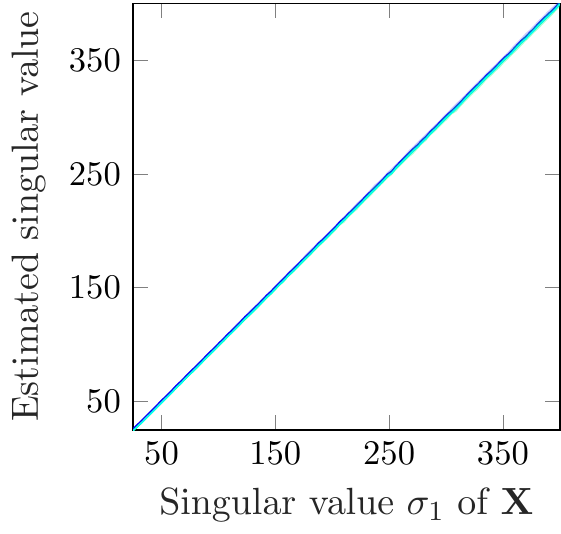}}\hfill%
\subfigure[]{\includegraphics[height=0.24\linewidth]{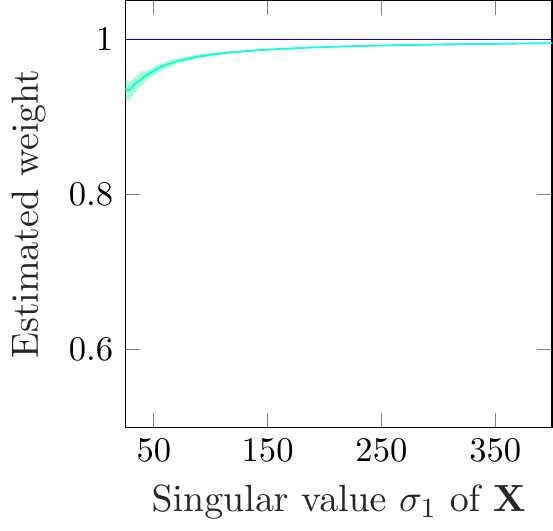}}\hfill%
\subfigure[]{\includegraphics[height=0.24\linewidth]{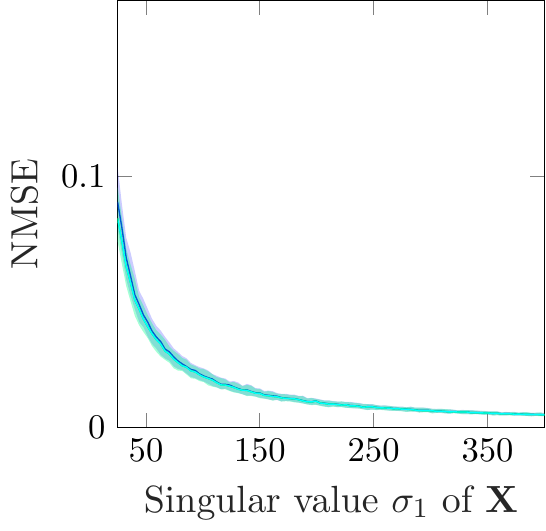}}\hfill%
\subfigure[]{\includegraphics[height=0.24\linewidth]{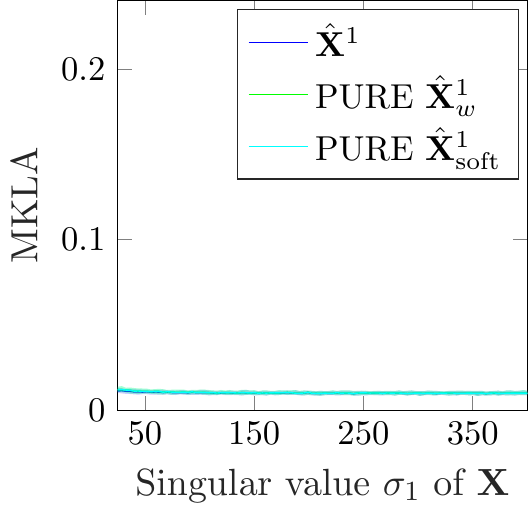}}\\
\subfigure[]{\includegraphics[height=0.24\linewidth]{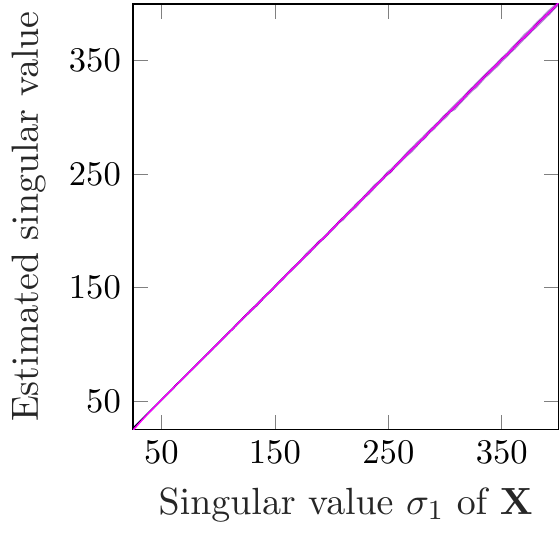}}\hfill%
\subfigure[]{\includegraphics[height=0.24\linewidth]{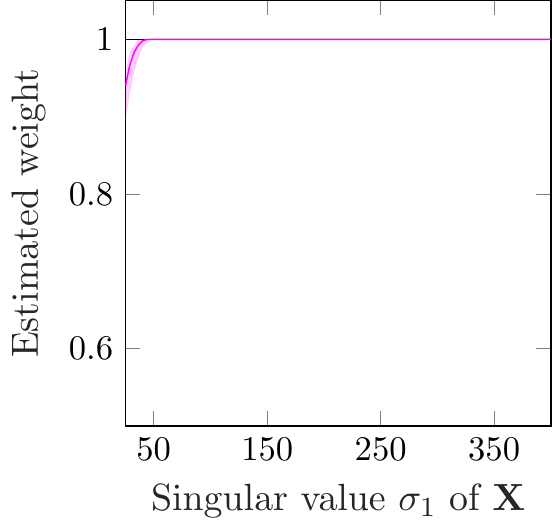}}\hfill%
\subfigure[]{\includegraphics[height=0.24\linewidth]{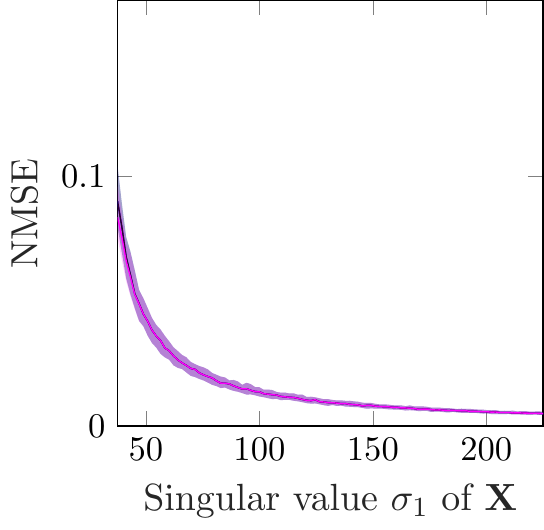}}\hfill%
\subfigure[]{\includegraphics[height=0.24\linewidth]{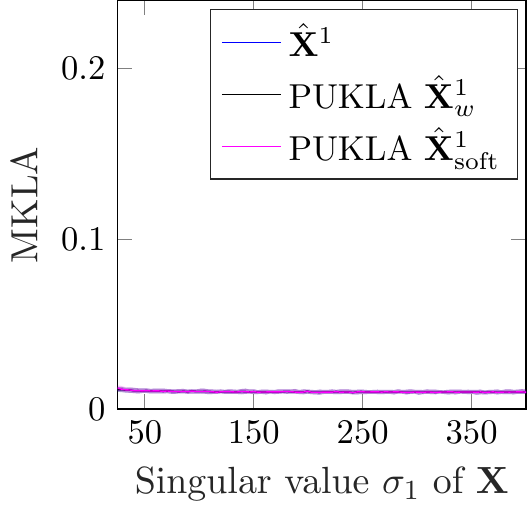}}\\
\caption{
  The case of Poisson measurements with $m=n=100$.
  (a) Estimated first eigenvalue $\hat \sigma_1$ as a function of the true underlying one $\sigma_1$
  for our proposed estimator $\hat{\bX}_w^1$ and the soft-thresholding $\hat{\bX}_{\mathrm{\mathrm{soft}}}^1$
  when both are guided by the PURE.
  Both of them are compared to the first singular value $\tilde \sigma_1$ of $\bY^1$.
  Same but for (b) the corresponding weights $\hat w_1 = \hat \sigma_1 / \tilde \sigma_1$,
  (c) the NMSE risk and (d) the MKLA risk.
  (e-h) Exact same esperiments but when our proposed estimator and
  the soft-thresholding are both guided by PUKLA.
  Curves have been computed on $M=100$ noise realizations, only the median and an 80\% confidence interval
  are represented respectively by a stroke and a shadded area of the same color.
} \label{fig:PURE}
\end{figure}

\subsubsection*{Gamma and Poisson distributed measurements}

Let us now consider the case where the entries of $\bY_{ij} \geq 0$ of the data matrix $\bY$ are independently sampled from a Gamma or Poisson distribution with mean $\bX_{ij} > 0$. To satisfy the constraint that the estimators must be matrices with positive entries, we consider estimators of the form \eqref{eq:hatXw1_eps}.
In this context, we compare the following spectral shrinkage estimators,
set for $\varepsilon = 10^{-6}$, as:
\begin{description}
\item[$\bullet$] Rank-1 PCA shrinkage
\begin{flalign*}
\hat{\bX}^1 = \max\left[\tilde \sigma_{1} \tilde{\bu}_{1} \tilde{\bv}_{1}^{t}, \varepsilon \right], &&
\end{flalign*}
\item[$\bullet$] Rank-1 GSURE/SUKLS/PURE/PUKLA-driven soft-thresholding
\begin{flalign*}
\hat{\bX}_{\mathrm{soft}}^1 = \max\left[\hat{\sigma}_1 \tilde{\bu}_{1} \tilde{\bv}_{1}^{t}, \varepsilon \right]
\quad \text{with} \quad
\hat{\sigma}_1 =
( \tilde \sigma_{1} - \lambda(\bY))_{+}, &&
  \end{flalign*}
  \item[$\bullet$] Rank-1 GSURE/SUKLS/PURE/PUKLA-driven weighted estimator
\begin{flalign*}
\hat{\bX}_w^1 = \max\left[\hat{\sigma}_1 \tilde{\bu}_{1} \tilde{\bv}_{1}^{t}, \varepsilon \right]
\quad \text{with} \quad
\hat{\sigma}_1 = w_1(\bY) \tilde{\sigma}_1 \1_{\left\{ 1 \in \tilde{s} \right\} }, &&
\end{flalign*}
\end{description}
where $\tilde s$ is the approximated active subset
as defined in Section \ref{sec:bulk_exp_fam}.
For the soft-thresholding,
the value $\lambda(\bY) > 0$ is obtained by a numerical solver in order to minimize
either the $\GSURE$ or the $\SUKLS$ criterion (in the Gamma case)
and either the $\PURE$ or the $\PUKLA$ criterion (in the Poisson case).
The weight
$w_1(\bY) \in [0, 1]$ is obtained by a numerical solver in order to minimize the $\GSURE$ and the $\PURE$,
as described in Section \ref{sec:algo}.
According to Section \ref{sec:optexp},
the weight $w_1(\bY) \in [0, 1]$,
minimizing the $\SUKLS$ criterion,
has the following closed-form formula
\begin{align*}
w_{1}(\bY) = \min\left[1, \left(    \frac{L - 1}{L m n } \sum_{i=1}^{n} \sum_{j=1}^{m}   \frac{\hat \bX^1_{ij}}{\bY_{ij} } + \frac{1}{L m n} \left(1  + 2   \sum_{\ell =2  }^{n} \frac{\tilde{\sigma}_{1}^{2}}{\tilde{\sigma}_{1}^{2} - \tilde{\sigma}_{\ell}^{2} } \right) \right)^{-1} \right],
\end{align*}
and for the $\PUKLA$ criterion, we have
\begin{align*}
w_{1}(\bY) = \min\left[1, \frac{ \sum_{i=1}^{n} \sum_{j=1}^{m}  \bY_{ij} }{  \sum_{i=1}^{n} \sum_{j=1}^{m} \hat \bX^1_{ij} } \right].
\end{align*}
To evaluate the performances of these estimators, we perform again a study
involving $M = 100$ noise realizations.

In the Gamma case with shape parameter $L = 3$,
results are reported in Figure \ref{fig:GSURE}
where $\sigma_1$ ranges from $0.1$ to $5$.
In the Poisson case,
results are reported in Figure \ref{fig:PURE}.
To generate data from a Poisson distribution with mean value $\bX = \sigma_{1} \bu_{1} \bv_{1}^{t}$, we took $\sigma_{1}$
ranging from $25$ to $400$. In this context, the entries $\bX_{i,j}$ are in average ranging from $0.25$ to $4$.
When $\sigma_{1} = 25$, about $78\%$ of the entries of $\bY$ are $0$ and $20\%$ are equals to $1$ which correspond to an extreme level of noise,
while when $\sigma_{1} = 400$, the entries of $\bY$ concentrate around $4$ with a standard deviation of $2$
which correspond to a simpler noisy setting.

In these experiments, it can be seen that all the data-dependent spectral estimators achieve comparable results
with really small errors in terms of MSE and MKL risks. Their performances are similar to $\hat{\bX}^{1} = \tilde{\sigma}_{1} \tilde{\bu}_{1} \tilde{\bv}_{1}^{t}$ meaning that optimizing either SURE-like criteria leads to a spectral estimator closed to correspond to matrix denoising by ordinary PCA.
However, unlike the Gamma case, it might be observed in the Poisson case
that when reaching a stronger noise level, {\it i.e}, for small value of $\sigma_1$,
the NMSE of all estimator increases as the denoising problem becomes more challenging.
Nevertheless, only the weight of $\hat{\bX}_w^1$ driven by PUKLA does not present a drop
wich allows reaching a slightly smaller MKLA.
In the Gamma case, the noise level being proportional to the signal level,
the NMSE/MKLS remain constant for all $\sigma_1$.

\bigskip

Finally, as mentionned  by \cite{GavishDonoho}, to use the estimator $\hat{\bX}_{\ast}^{1}$ in a Gaussian model with homoscedastic variance $\tau^2 \neq  \frac{1}{m}$, one may take the estimator $\hat{\bX}_{\ast}^{1} = \sqrt{m} \tau f_{1}^{\ast}( \tilde{\sigma}_{1} / (\sqrt{m} \tau)) \tilde\bu_{1} \tilde\bv_{1}^{t}$. Hence, provided the variance of the entries of the data matrix $\bY$ is known, it is always possible to use a scaled version of the shrinkage rule from \cite{GavishDonoho} when $\tau^2 \neq  \frac{1}{m}$. However, in the setting of Gamma or Poisson noise, the variance of the additive noise varies from one entry to another and depends on the entries of the unknown signal matrix $\bX$ to recover. For this reason, it is not possible to use a scaled version of the shrinkage rule from \cite{GavishDonoho} as this would require to use scaling factors depending on the unknown values of the entries of $\bX$. Therefore, a comparison between our approach and the asymptotically optimal shrinkage proposed in \cite{MR3200641}  and \cite{GavishDonoho} (for Gaussian noise) is not possible in the case of Gamma or Poisson measurements. Note that for Gamma measurements one has that $\var(\bY_{ij}) = \bX_{ij}^{2}/L$, and, in our numerical experiments, it is assumed that the constant $L$ is known.

A typical example where this assumption is reasonable,
is the one of the statistical models of speckle used in coherent imagery,
such as, Synthetic Aperture Radar (SAR) and
SOund Navigation And Ranging (SONAR) imagery. In such imaging systems,
the observed irradiance $\bY_{ij}$ of a pixel with indices $(i,j)$
is obtained as the square modulus of a complex signal modeled as
being zero-mean circular complex Gaussian distributed
(consequence of the Central Limit Theorem)
\cite{goodman1976some}.
It follows that $\bY_{ij}$ has an exponential distribution%
\footnote{The exponential distribution is a particular instance of the Gamma distribution with parameter $L=1$}
with mean $\bX_{ij}$ corresponding to the underlying irradiance to be estimated.
In order to improve the contrast of such images (namely, the signal to noise ratio),
an average of $L$ independent and identically distributed images is often performed,
and the resulting pixel value becomes Gamma distributed with parameter $L$
\cite{ulaby1989handbook}.
Because the number $L$ of images to be averaged is chosen by the practitioner,
the parameter $L$ is absolutely known without uncertainties, and
for this reason it does not require to be estimated.
Nevertheless the variance $\var(\bY_{ij}) = \bX_{ij}^{2}/L$ remains unknown.

\bigskip

While all estimators behave similarly in the rank 1 setting,
we will see in the next section that they can significantly differ when
the rank is let to be larger than $2$.


\subsection{The case of a signal matrix of rank larger than two} \label{sec:rankmoretwo}

We now consider the more complex an realistic setting where the rank $r^{\ast}$ of the matrix $\bX$ is unknown
and potentially larger than two, i.e.,
\begin{align*}
\bX = \sum_{k=1}^{r^{\ast}} \sigma_{k} \bu_{k} \bv_{k}^{t},
\end{align*}
where $\bu_{k} \in \R^{n} $ and $\bv_{k} \in \R^{m}$ are vectors with unit norm that are fixed in this numerical experiment, and $\sigma_{k}$ are positive real values also fixed in this experiment. We also choose to fix $n = 100$ and $m = 200$, while the true rank is $r^{\ast} = 9$ as shown by the red curve in Figure  \ref{fig:gaussian_higher_rank}(i).
Again, let $\bY = \sum_{k = 1}^{\min(n,m)} \tilde{\sigma}_{k} \tilde{\bu}_{k} \tilde{\bv}_{k}^{t}$ be an $n \times m$ matrix whose entries are sampled from  model \eqref{eq:expfam}
and then satisfying $\E[\bY] = \bX$.

\newcommand{\nmakebox}[2][1em]{\hspace{-#1}\makebox[#1][l]{#2}}
\newlength{\plop}
\setlength{\plop}{0.6em}
\begin{figure}[!t]
  \centering
  \hfill
  \subfigure[]{%
    \includegraphics[width=0.18\linewidth,viewport=1 130 100 200, clip]{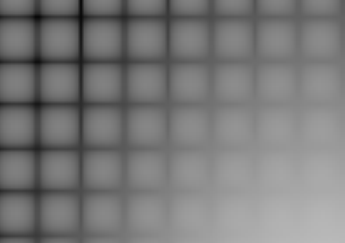}%
    \nmakebox[\plop]{\textcolor{red}{$\bullet$}}%
  }\hspace{1em}
  \subfigure[]{%
    \includegraphics[width=0.18\linewidth,viewport=1 130 100 200, clip]{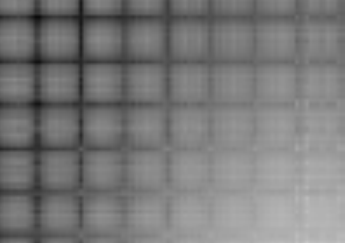}%
    \nmakebox[\plop]{\textcolor{magenta}{$\bullet$}}%
  }\hspace{1em}
  \subfigure[]{%
    \includegraphics[width=0.18\linewidth,viewport=1 130 100 200, clip]{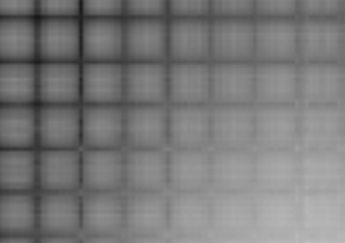}%
    \nmakebox[\plop]{\textcolor{magenta}{$\boldsymbol\circ$}}%
  }\hspace{1em}
  \subfigure[]{%
    \includegraphics[width=0.18\linewidth,viewport=1 130 100 200, clip]{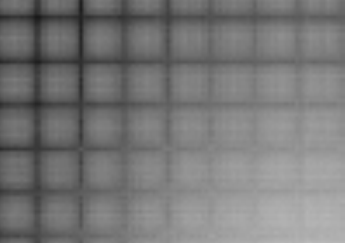}%
    \nmakebox[\plop]{\textcolor{green}{$\bullet$}}%
  }%
  \hfill{ }
  \\
  \hfill
  \subfigure[]{%
    \includegraphics[width=0.18\linewidth,viewport=1 130 100 200, clip]{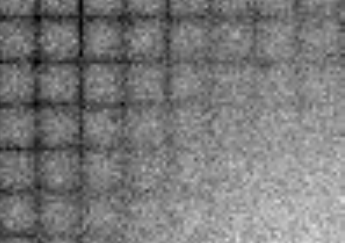}%
    \nmakebox[\plop]{\textcolor{blue}{$\bullet$}}%
  }\hspace{1em}
  \subfigure[]{%
    \includegraphics[width=0.18\linewidth,viewport=1 130 100 200, clip]{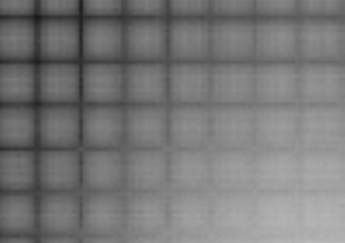}%
    \nmakebox[\plop]{\textcolor{black}{$\bullet$}}%
  }\hspace{1em}
  \subfigure[]{%
    \includegraphics[width=0.18\linewidth,viewport=1 130 100 200, clip]{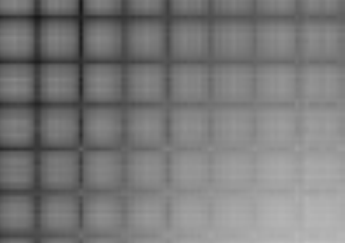}%
    \nmakebox[\plop]{\textcolor{black}{$\boldsymbol\circ$}}%
  }\hspace{1em}
  \subfigure[]{%
    \includegraphics[width=0.18\linewidth,viewport=1 130 100 200, clip]{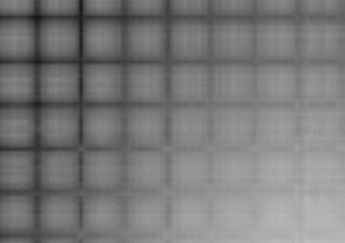}%
    \nmakebox[\plop]{\textcolor{chocolate}{$\bullet$}}%
  }%
  \hfill{ }
  \\%
  \hfill%
  \subfigure[]{\includegraphics[height=0.24\linewidth]{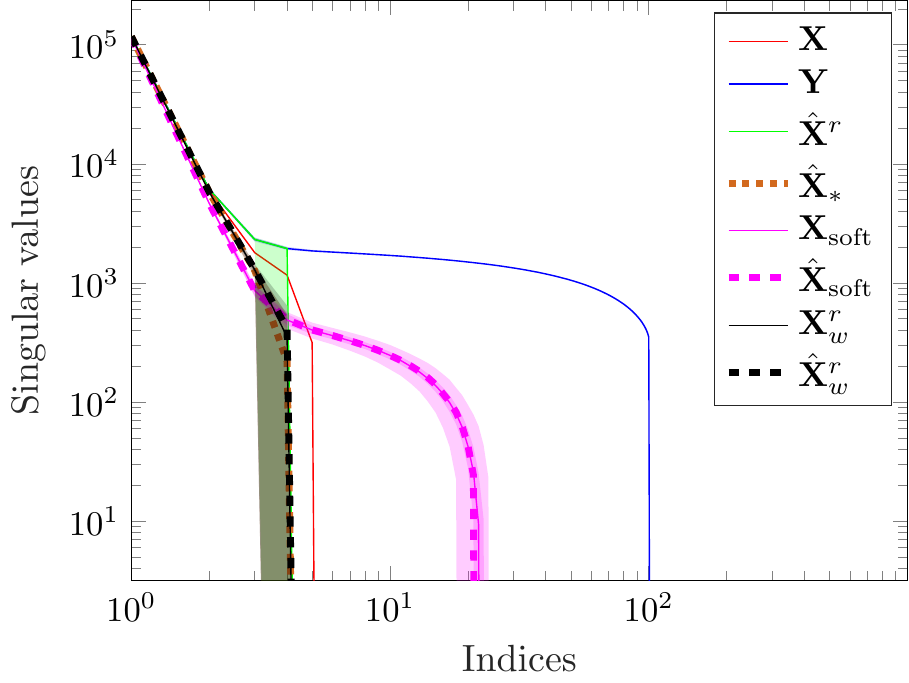}}\hfill%
  \subfigure[]{\includegraphics[height=0.24\linewidth]{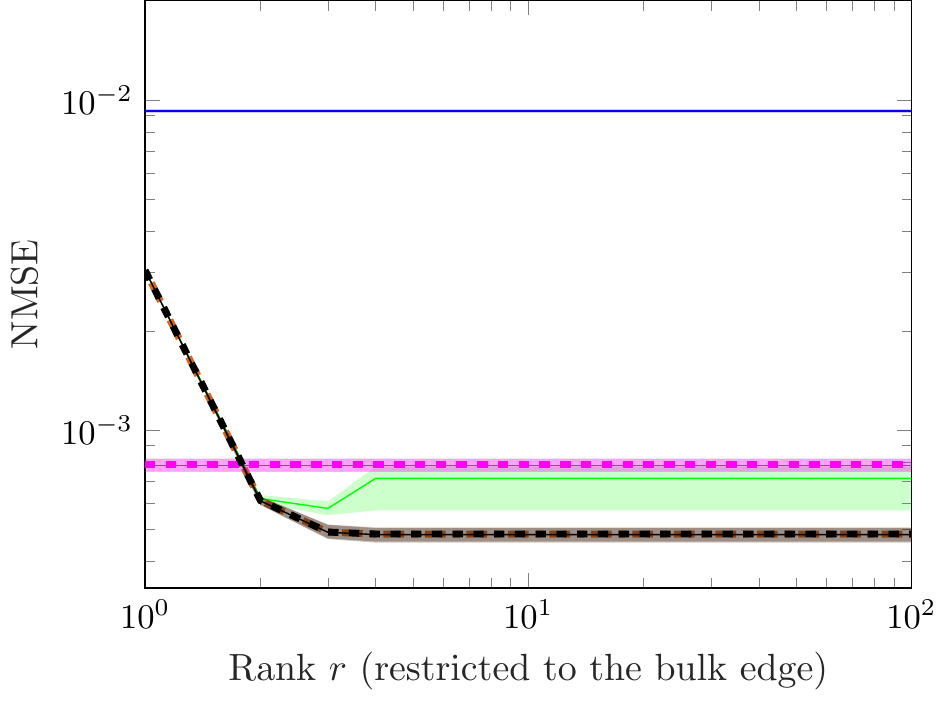}}\hfill%
  \subfigure[]{\includegraphics[height=0.24\linewidth]{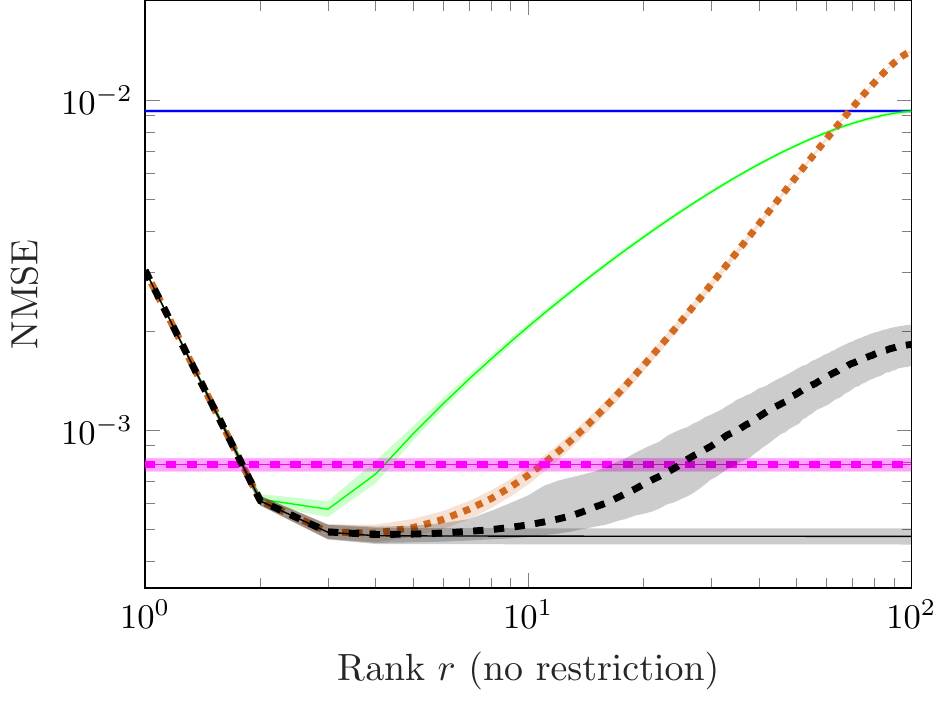}}%
  \hfill{ }
  \\%
  \hfill%
  \subfigure[]{\includegraphics[height=0.24\linewidth]{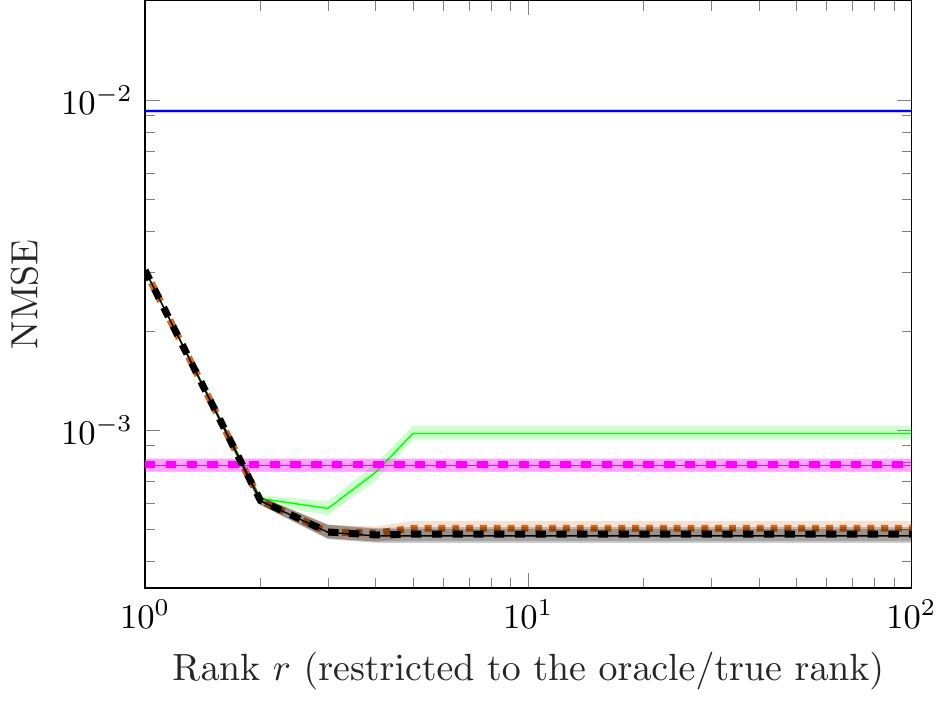}}\hfill%
  \subfigure[]{\includegraphics[height=0.24\linewidth]{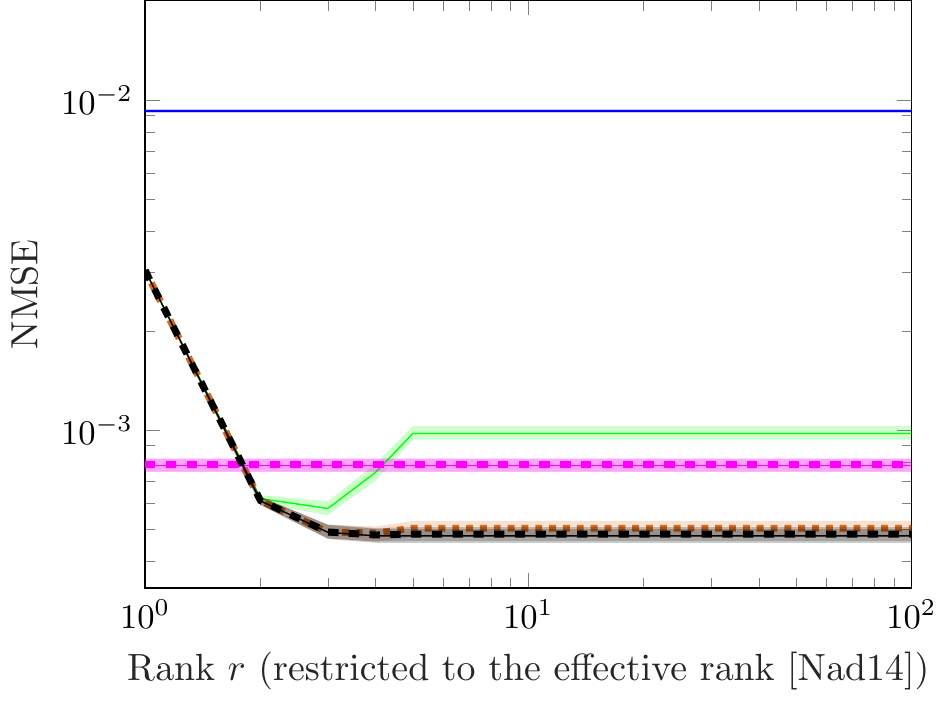}}\hfill%
  \subfigure[]{\includegraphics[height=0.24\linewidth]{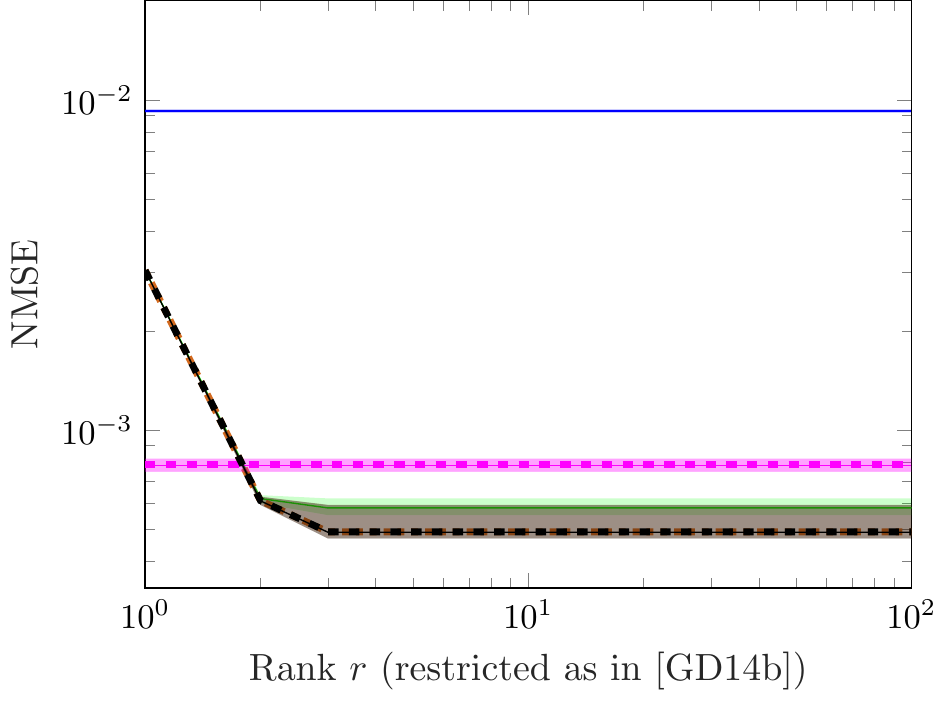}}%
  \hfill{ }
  \\[-0.5em]
\caption{
  (a) Zoom  on a $100 \times 200$ noise-free matrix and
  (e) a single realization of corrupted version by Gaussian noise ($\tau=80$).
  (b,c) Oracle soft-thresholding $ \bX_{\mathrm{soft}}$
  and data-driven soft-thresholding $ \hat{\bX}_{\mathrm{soft}}$.
  (d) PCA full rank $\hat{\bX}^{r_{\max}}$, i.e., $r_{\max} = \min(n, m)$.
  (f,g,h) Oracle full rank approximation
  $\bX^{r_{\max}}_w$, and data-driven full rank estimation
  $\hat \bX^{r_{\max}}_{w}$ and $\hat \bX^{r_{\max}}_{*}$.
  (i) Their corresponding singular values.
  (j) NMSE of the various approximations
  as a function of the rank $r$.
  (k) Same but without knowledge the bulk edge, namely $c_+ = 0$. (l,m,n) Same when the active set of singular values is of the form $\hat{s} =\{1,\ldots,\hat{r}\}$ where $\hat{r}$ is either given by $\hat{r} = r^{\ast}$ (oracle/true rank), $\hat{r} = r_{\rm{eff}}$ (effective rank) or  by \eqref{eq:estrankHT}.  In all the figures, the solid curves correspond to oracle estimators and the dashed curves correspond to data-driven estimators, obtained over $M = 1,000$ noise realizatrions. The grey areas represent a 80\% confidence interval.
}
\vspace{-2em}
\label{fig:gaussian_higher_rank}
\end{figure}

\subsubsection*{Gaussian distributed measurements}

We first consider the case of Gaussian measurements,
where
$
\bY = \bX + \bW
$
with $\E[\bW_{ij}] = 0$, $\mathrm{Var}(\bW_{ij}) =\tau^2$ with $\tau = 80$.
In the following numerical experiments, we study the behavior of the spectral estimator:
\begin{description}
\item[$\bullet$] PCA shrinkage
\begin{flalign*}
\hat{\bX}^r = \sum_{k=1}^r \tilde \sigma_{k} \tilde{\bu}_{k} \tilde{\bv}_{k}^{t} \; \1_{\left\{ k \;\leq\; \hat{r} \right\} }, &&
\end{flalign*}
\item[$\bullet$] SURE-driven soft-thresholding
\begin{flalign*}
\hat{\bX}_{\mathrm{soft}} = \sum_{k=1}^{\min(m, n)} \hat{\sigma}_k \tilde{\bu}_{k} \tilde{\bv}_{k}^{t}
\quad \text{with} \quad
\hat{\sigma}_k =
( \tilde \sigma_{k} - \lambda(\bY))_{+}, &&
  \end{flalign*}
\item[$\bullet$] Asymptotically optimal shrinkage proposed in \cite{MR3200641} and \cite{GavishDonoho}
  \begin{flalign*}
    \hat{\bX}_{\ast}^r = \sum_{k=1}^r \hat{\sigma}_{k} \tilde\bu_{k} \tilde\bv_{k}^{t}
    \quad \text{with} \quad
    \hat{\sigma}_{k} =
    \frac{1}{ \tilde{\sigma}_{k}} \sqrt{\left(\tilde{\sigma}_{k}^2-(c+1) \right)^{2}-4c} \; \1_{\left\{ k \;\leq\; \hat{r} \right\} }, &&
  \end{flalign*}
  \item[$\bullet$] SURE-driven weighted estimator that we have derived in Section \ref{sec:optexp}
\begin{flalign*}
\hat{\bX}_w^r = \sum_{k=1}^r \hat{\sigma}_k \tilde{\bu}_{k} \tilde{\bv}_{k}^{t}
\quad \text{with} \quad
\hat{\sigma}_k = \left( 1 - \frac{1}{ \tilde{\sigma}_{k}^2}  \left(  \frac{k}{m}+  \frac{2}{m}    \sum_{\ell = 2}^{n} \frac{ \tilde{\sigma}_{k}^{2}}{\tilde{\sigma}_{k}^{2} - \tilde{\sigma}_{\ell}^{2} } \right) \right)_{+}\tilde \sigma_k \1_{\left\{ k \;\leq\; \hat{r} \right\} }, &&
\end{flalign*}
\end{description}
where $r \in [1, \min(n, m)]$, and for the soft-thresholding,
the value $\lambda(\bY) > 0$ is obtained by a numerical solver
in order to minimize the $\SURE$.
Otherwise specified,
we consider $\hat{r} = \max \{k \; ; \; \tilde{\sigma}_k > c_+^{n,m}\}$, {\it i.e.},
an estimator of the rank using knowledge of the bulk edge $c_{+} \approx c_+^{n,m}$,
hence, $\1_{\left\{ k \;\leq\; \hat{r} \right\} } = \1_{\left\{ \tilde{\sigma}_k > c_+^{n,m} \right\} }$. As discussed in Section \ref{sec:bulk_exp_fam}, we compare, in these experiments, the influence of rank estimation by analyzing the performances of the same estimators when either $\hat{r} = r_{\max} = \min(n,m)$ (i.e.\ without knowledge the bulk edge, namely $c_+ = 0$), $\hat{r}  = r^{\ast}$ (oracle/true rank), $\hat{r} = r_{\rm{eff}}$ (effective rank \cite{MR3200641}) or  by \eqref{eq:estrankHT} (from hard-thresholding of singular values in  \cite{GDIEEE14}).

In order to assess the quality of $\SURE$ as an estimator of the $\MSE$, we also compare
the aforementioned approach with their oracle counterparts given by
\begin{align*}
&{\bX}_{\mathrm{soft}} =
\sum_{k=1}^{\min(m, n)} \hat{\sigma}_k \tilde\bu_{k} \tilde\bv_{k}^{t}
\quad \text{with} \quad \hat{\sigma}_k = ( \tilde{\sigma}_{k} - \lambda^{\mathrm{oracle}}(\bY))_{+}, \quad \text{and}\\
&{\bX}^r_w =
\sum_{k=1}^{r} \hat{\sigma}_k \tilde\bu_{k} \tilde\bv_{k}^{t}
\quad \text{with} \quad \hat{\sigma}_k = \tilde{\bv}_k^t \bX \tilde{\bu}_k,
\end{align*}
where
$\lambda^{\mathrm{oracle}}(\bY)$ minimizes the squared error
$
\SE
$
(non-expected risk)  over the sets and soft-thresholding approximations respectively.
Note that ${\bX}_w^r$ and ${\bX}_{\mathrm{soft}}$ are  ideal approximations of $\bX$ that cannot be used in practice but serve as benchmarks to evaluate the performances of the data-driven estimators $\hat{\bX}^r$, $\hat{\bX}_{\mathrm{soft}}$, $\hat{\bX}_{\ast}^{r}$ and $\hat{\bX}_{w}^{r}$.
In order to shed some light on the variance of these estimators, and indirectly on the variance of the SURE, we perform this experiments over $M=1000$ independent realizations of
$\bY$.

The results are reported
on Figure \ref{fig:gaussian_higher_rank}. For an estimator of the rank given either by  $\hat{r} = \max \{k \; ; \; \tilde{\sigma}_k > c_+^{n,m}\}$ (knowledge of the bulk edege), $\hat{r}  = r^{\ast}$ (oracle/true rank), $\hat{r} = r_{\rm{eff}}$ (effective rank) or  by \eqref{eq:estrankHT}, it can be observed that $\hat{\bX}^r_{w}$, $\hat{\bX}^r_{*}$ and
${\bX}^r_w$ achieve comparable performances for all
$r \in [1, \min(m, n)]$ even though the two first do not rely on
the unknown matrix $\bX$.
Similarly $\hat{\bX}_{\mathrm{soft}}$ and ${\bX}_{\mathrm{soft}}$
achieve also comparable performances showing again that the $\SURE$ accurately estimates
the $\MSE$.
In terms of error bands for the $\NMSE$, $\hat{\bX}^r_{w}$, $\hat{\bX}^r_{*}$ and ${\bX}^r_w$
outperform
$\hat{\bX}_{\mathrm{soft}}$ and ${\bX}_{\mathrm{soft}}$ provided that $r$ is large enough.
Moreover, the performance of $\hat{\bX}^r_{w}$ plateaus to its optimum
when the rank $r$ becomes large. This allows us to choose $r = \min(n, m)$
when we do not have {\it a priori} on the true or effective rank.

Interestingly, Fig.~\ref{fig:gaussian_higher_rank}.(k) shows that
when the above estimators are used without the knowledge of the bulk edge (i.e.\ by taking $c_+^{n,m} = 0$ in their computation instead of $c_+^{n,m} = 1 + \sqrt{\frac{n}{m}} $, which corresponds to the choice $\hat{r} = r_{\max} = \min(n,m)$), the performance of $\hat{\bX}^r_{w}$ actually decreases when the rank $r$ becomes too large. Indeed, it is clear from Fig.~\ref{fig:gaussian_higher_rank}.(k), that the  the error band of the NMSE of $\hat{\bX}^r_{w}$ becomes much larger as the rank $r$ increases. This illustrates that the $\SURE$ suffers from estimation variance in the case of over parametrization when $r$ becomes too large, and thus it cannot be used to estimate jointly a too large number of weights. Therefore, the knowledge of an appropriate estimator $\hat{r}$ of the rank (e.g.\ using the bulk edge) seems to provide a relevant upper bound on the number
of weights that can be jointly and robustly estimated with the $\SURE$.

\begin{figure}[!t]
\centering
\subfigure[]{%
  \includegraphics[width=0.18\linewidth,viewport=1 130 100 200, clip]{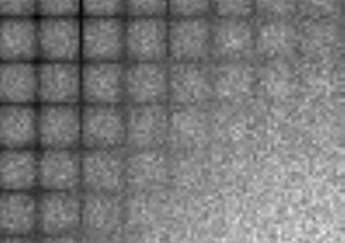}%
  \hspace{-0.6em}\textcolor{blue}{$\bullet$}%
}\hfill
\subfigure[]{%
  \includegraphics[width=0.18\linewidth,viewport=1 130 100 200, clip]{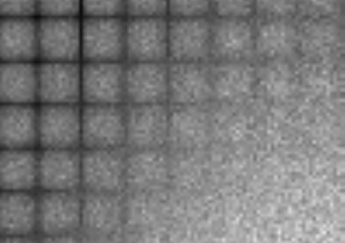}%
  \hspace{-0.6em}\textcolor{cyan}{$\bullet$}%
}\hfill
\subfigure[]{%
  \includegraphics[width=0.18\linewidth,viewport=1 130 100 200, clip]{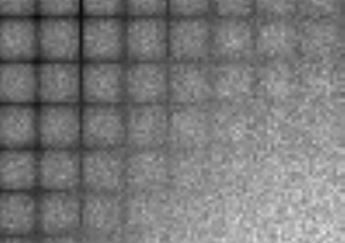}%
  \hspace{-0.6em}\textcolor{cyan}{$\boldsymbol\circ$}%
}\hfill
\subfigure[]{%
  \includegraphics[width=0.18\linewidth,viewport=1 130 100 200, clip]{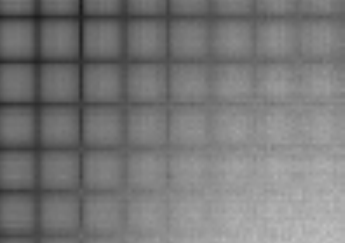}%
  \hspace{-0.6em}\textcolor{magenta}{$\bullet$}%
}\hfill
\subfigure[]{%
  \includegraphics[width=0.18\linewidth,viewport=1 130 100 200, clip]{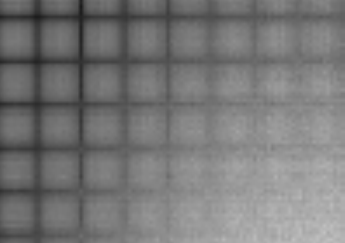}%
  \hspace{-0.6em}\textcolor{magenta}{$\boldsymbol\circ$}%
}\\
\subfigure[]{%
  \includegraphics[width=0.18\linewidth,viewport=1 130 100 200, clip]{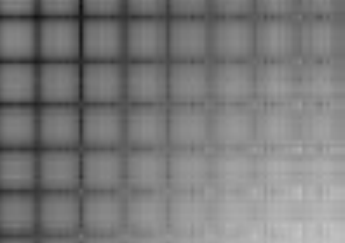}%
  \hspace{-0.6em}\textcolor{chocolate}{$\bullet$}%
}\hfill
\subfigure[]{%
  \includegraphics[width=0.18\linewidth,viewport=1 130 100 200, clip]{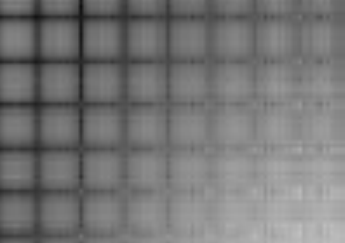}%
  \hspace{-0.6em}\textcolor{green}{$\bullet$}%
}\hfill
\subfigure[]{%
  \includegraphics[width=0.18\linewidth,viewport=1 130 100 200, clip]{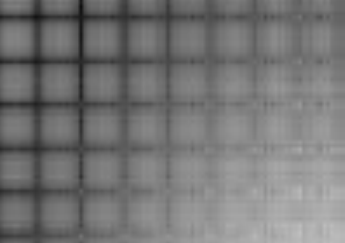}%
  \hspace{-0.6em}\textcolor{green}{$\boldsymbol\circ$}%
}\hfill
\subfigure[]{%
  \includegraphics[width=0.18\linewidth,viewport=1 130 100 200, clip]{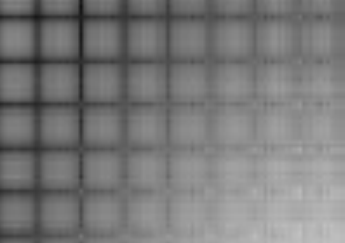}%
  \hspace{-0.6em}\textcolor{black}{$\bullet$}%
}\hfill
\subfigure[]{%
  \includegraphics[width=0.18\linewidth,viewport=1 130 100 200, clip]{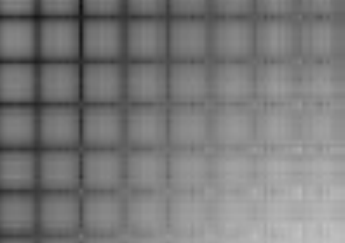}%
  \hspace{-0.6em}\textcolor{black}{$\boldsymbol\circ$}%
}\\
\hfill%
\subfigure[]{\includegraphics[height=0.24\linewidth]{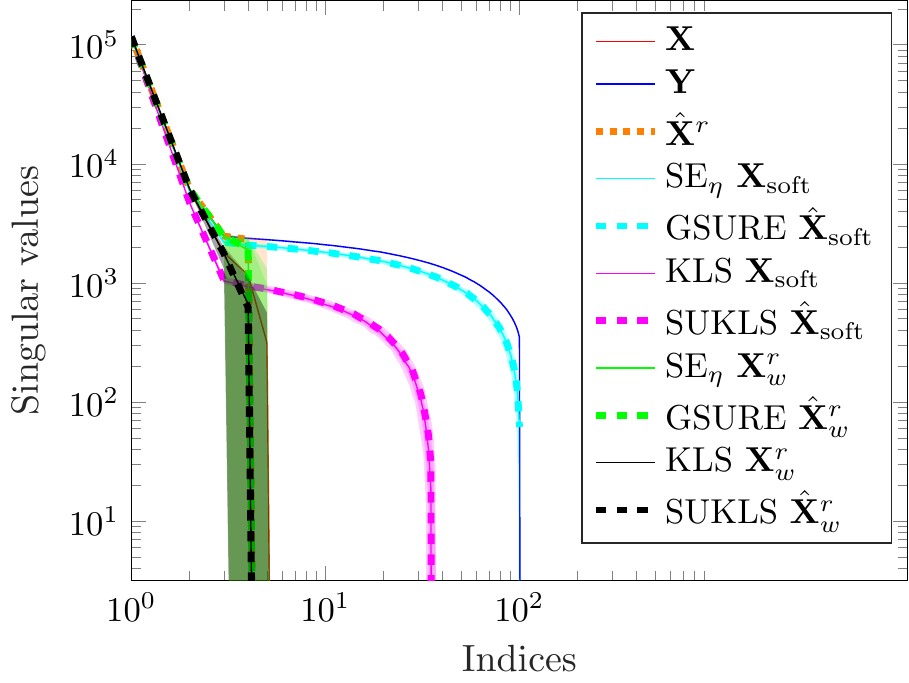}}\hfill%
\subfigure[]{\includegraphics[height=0.24\linewidth]{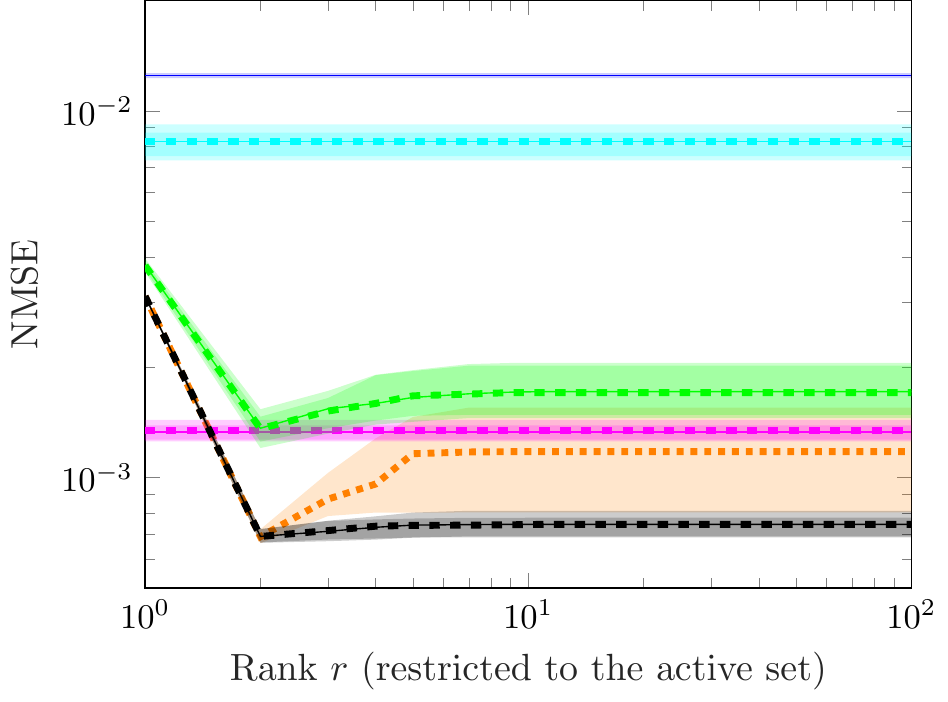}}\hfill%
\subfigure[]{\includegraphics[height=0.24\linewidth]{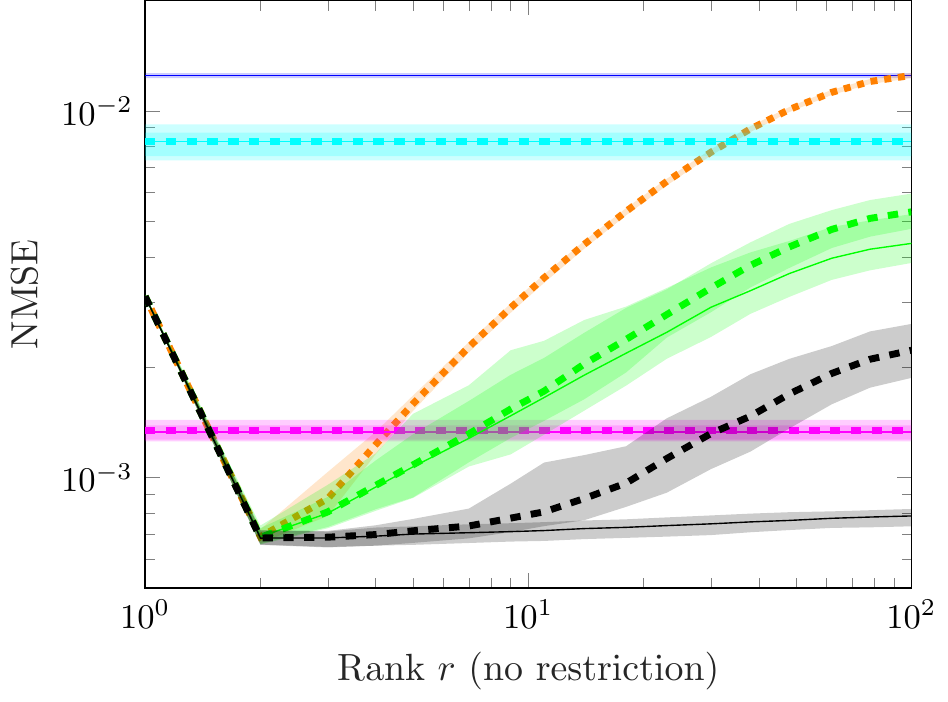}}%
\hfill{ }
\\
\hfill%
\hspace{0.32\linewidth}\hfill%
\subfigure[]{\includegraphics[height=0.24\linewidth]{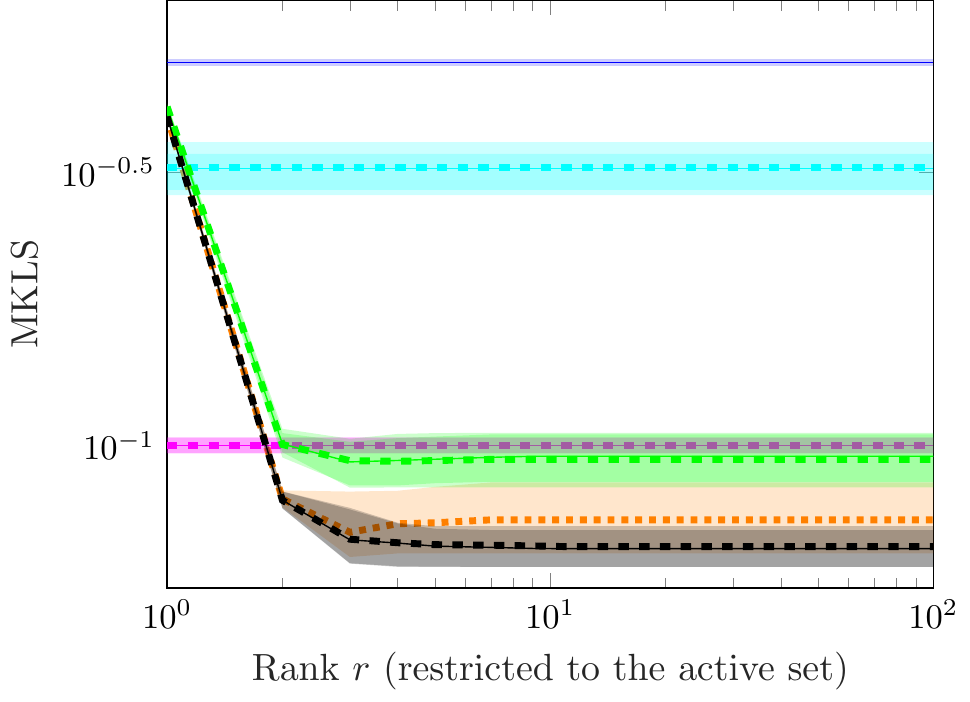}}\hfill%
\subfigure[]{\includegraphics[height=0.24\linewidth]{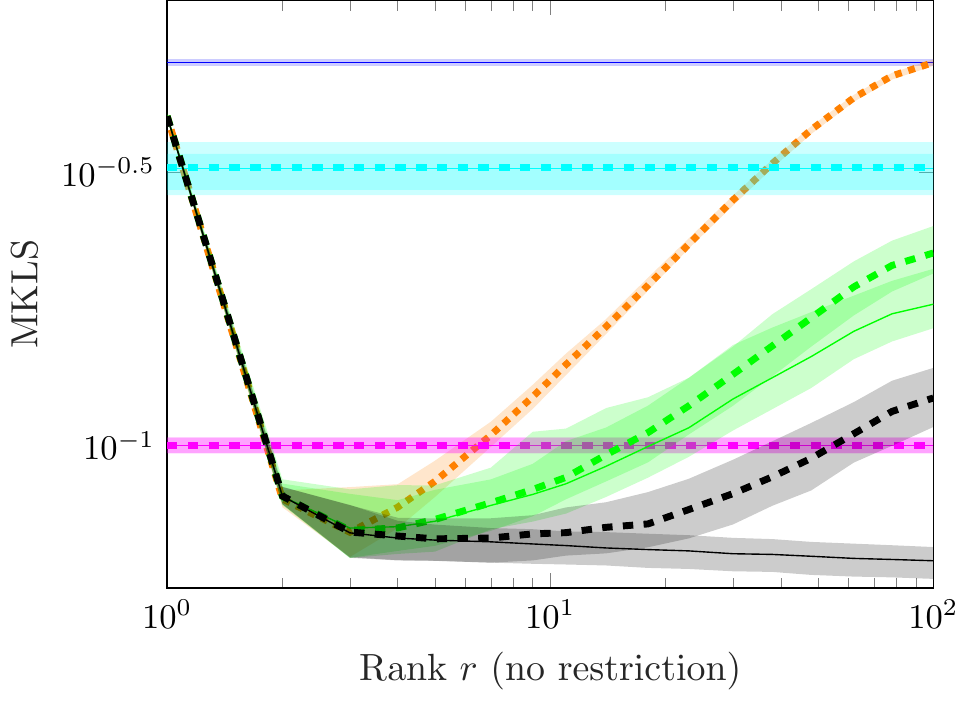}}%
\hfill{ }
\\[-0.5em]
\caption{
  (a) A single realization of corrupted version by Gamma noise ($L=80$) with zoom  on a $100 \times 200$ matrix.
  (b,c,d,e) Oracle soft-thresholding $ \bX_{\mathrm{soft}}$
  and data-driven soft-thresholding $\hat {\bX}_{\mathrm{soft}}$
  respectively for $\SE_\eta$, $\GSURE$, $\KLS$ and $\SUKLS$.
  (f) PCA $\hat{\bX}^{r_{\max}}$ with full rank approximation  i.e.\ $r_{\max} = \min(n, m)$.
  (g,h,i,j) Oracle full rank approximation
  $\bX^{r_{\max}}_w$, and data-driven full rank estimation $\hat \bX^{r_{\max}}_{w}$
  respectively for $\SE_\eta$, $\GSURE$, $\KLS$ and $\SUKLS$.
  (k) Their corresponding singular values averaged over $M = 100$ noise realizations.
  (l,m) $\NMSE$ averaged over $M = 100$ noise realizations
  as a function of the rank $r$
  with and without using the active set.
  (n,o) Same but with respect to $\MKLS$.
}
\label{fig:gamma_higher_rank}
\vspace{1em}
\end{figure}

\begin{figure}[!t]
\centering
\subfigure[]{%
  \includegraphics[width=0.18\linewidth,viewport=1 130 100 200, clip]{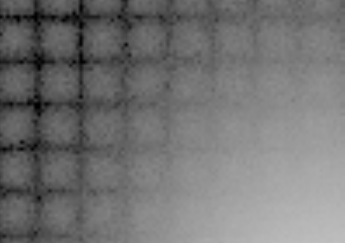}%
  \hspace{-0.6em}\textcolor{blue}{$\bullet$}%
}\hfill
\subfigure[]{%
  \includegraphics[width=0.18\linewidth,viewport=1 130 100 200, clip]{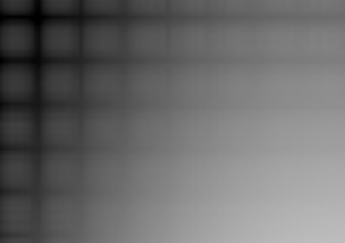}%
  \hspace{-0.6em}\textcolor{cyan}{$\bullet$}%
}\hfill
\subfigure[]{%
  \includegraphics[width=0.18\linewidth,viewport=1 130 100 200, clip]{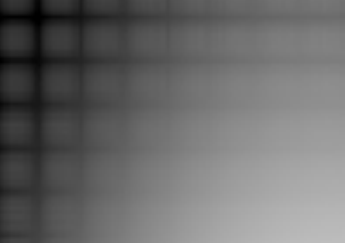}%
  \hspace{-0.6em}\textcolor{cyan}{$\boldsymbol\circ$}%
}\hfill
\subfigure[]{%
  \includegraphics[width=0.18\linewidth,viewport=1 130 100 200, clip]{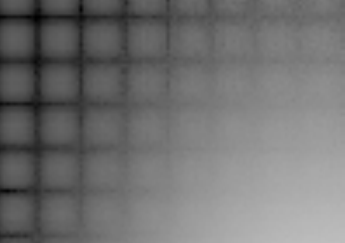}%
  \hspace{-0.6em}\textcolor{magenta}{$\bullet$}%
}\hfill
\subfigure[]{%
  \includegraphics[width=0.18\linewidth,viewport=1 130 100 200, clip]{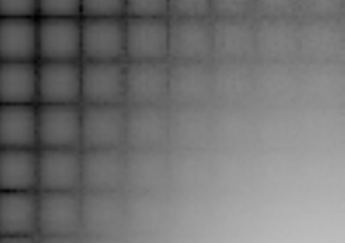}%
  \hspace{-0.6em}\textcolor{magenta}{$\boldsymbol\circ$}%
}\\
\subfigure[]{%
  \includegraphics[width=0.18\linewidth,viewport=1 130 100 200, clip]{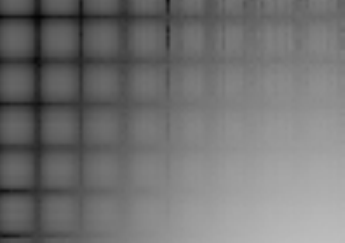}%
  \hspace{-0.6em}\textcolor{chocolate}{$\bullet$}%
}\hfill
\subfigure[]{%
  \includegraphics[width=0.18\linewidth,viewport=1 130 100 200, clip]{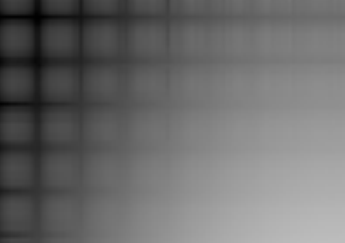}%
  \hspace{-0.6em}\textcolor{green}{$\bullet$}%
}\hfill
\subfigure[]{%
  \includegraphics[width=0.18\linewidth,viewport=1 130 100 200, clip]{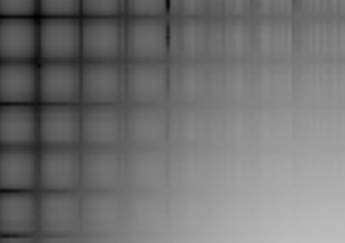}%
  \hspace{-0.6em}\textcolor{green}{$\boldsymbol\circ$}%
}\hfill
\subfigure[]{%
  \includegraphics[width=0.18\linewidth,viewport=1 130 100 200, clip]{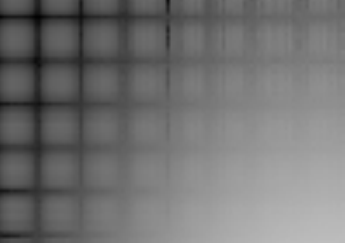}%
  \hspace{-0.6em}\textcolor{black}{$\bullet$}%
}\hfill
\subfigure[]{%
  \includegraphics[width=0.18\linewidth,viewport=1 130 100 200, clip]{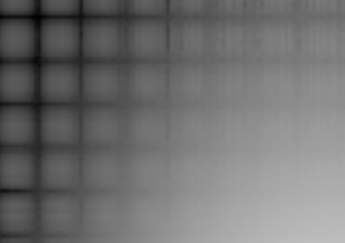}%
  \hspace{-0.6em}\textcolor{black}{$\boldsymbol\circ$}%
}\\
\hfill%
\subfigure[]{\includegraphics[height=0.24\linewidth]{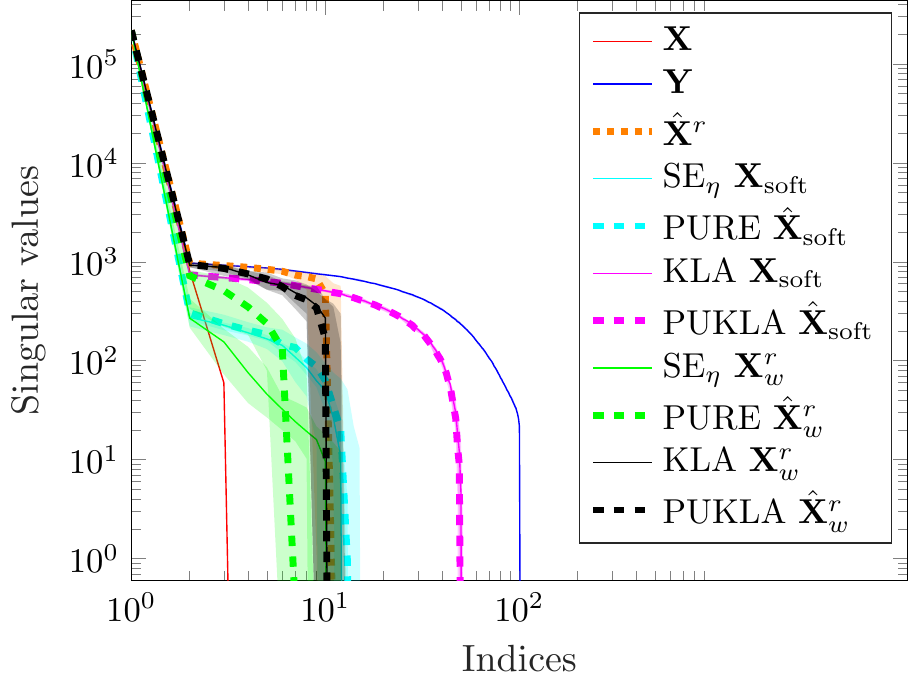}}\hfill%
\subfigure[]{\includegraphics[height=0.24\linewidth]{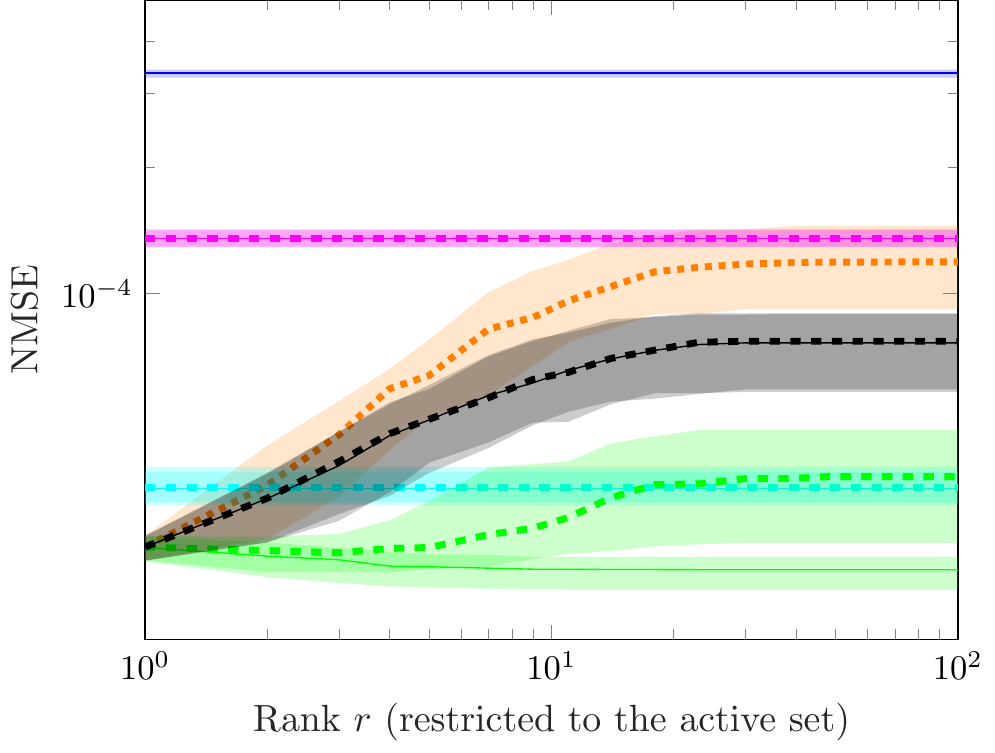}}\hfill%
\subfigure[]{\includegraphics[height=0.24\linewidth]{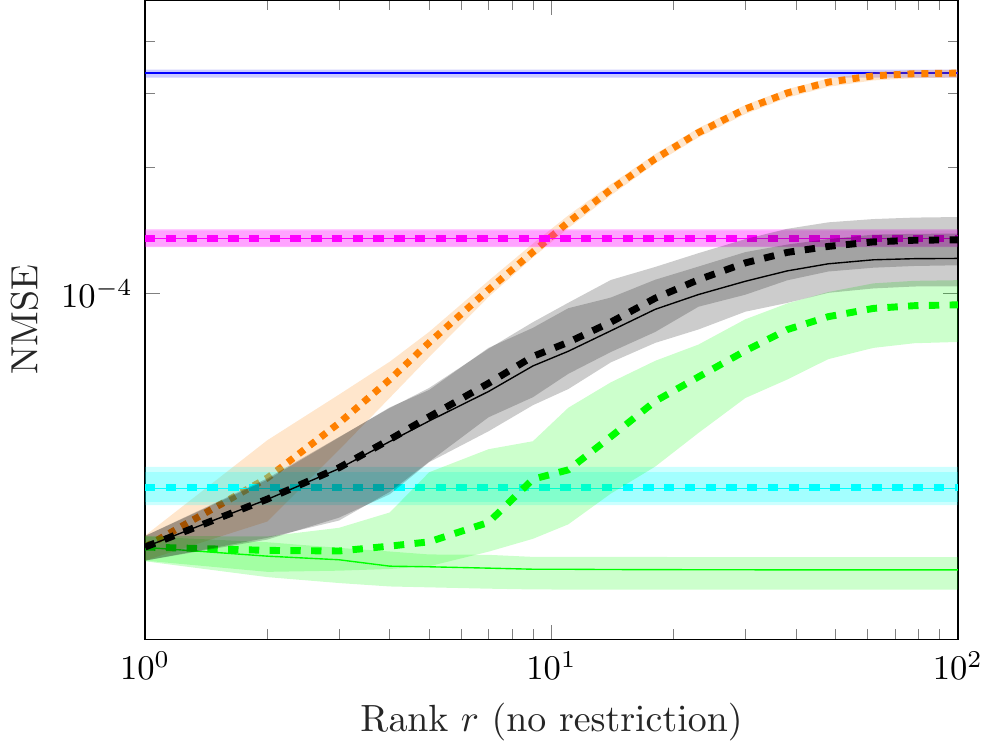}}%
\hfill{ }
\\
\hfill%
\hspace{0.32\linewidth}\hfill%
\subfigure[]{\includegraphics[height=0.24\linewidth]{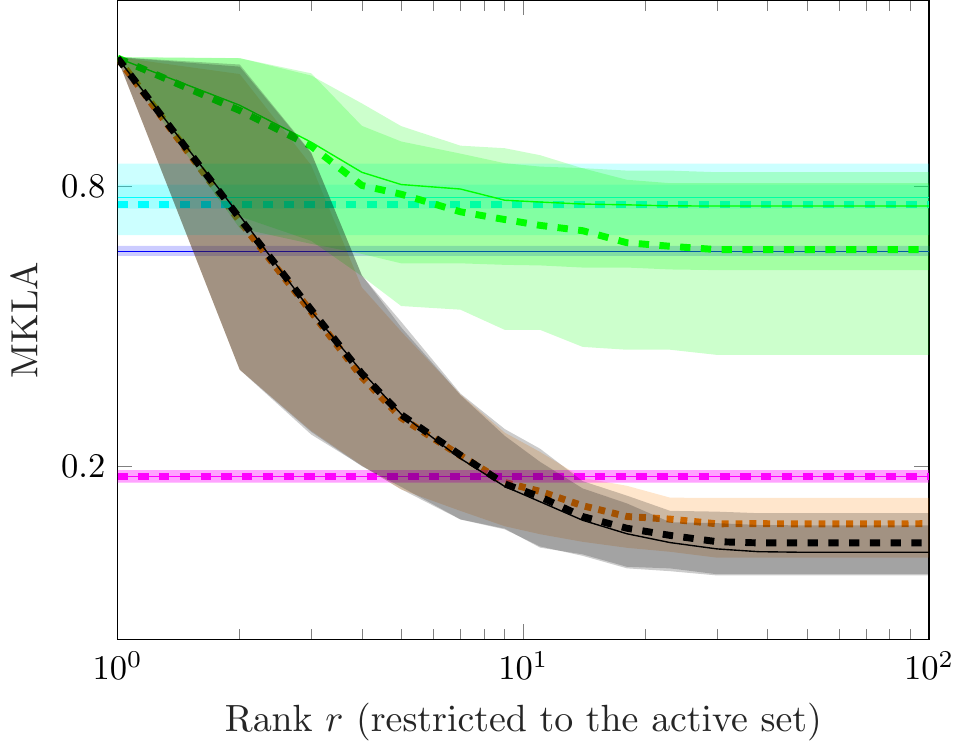}}\hfill%
\subfigure[]{\includegraphics[height=0.24\linewidth]{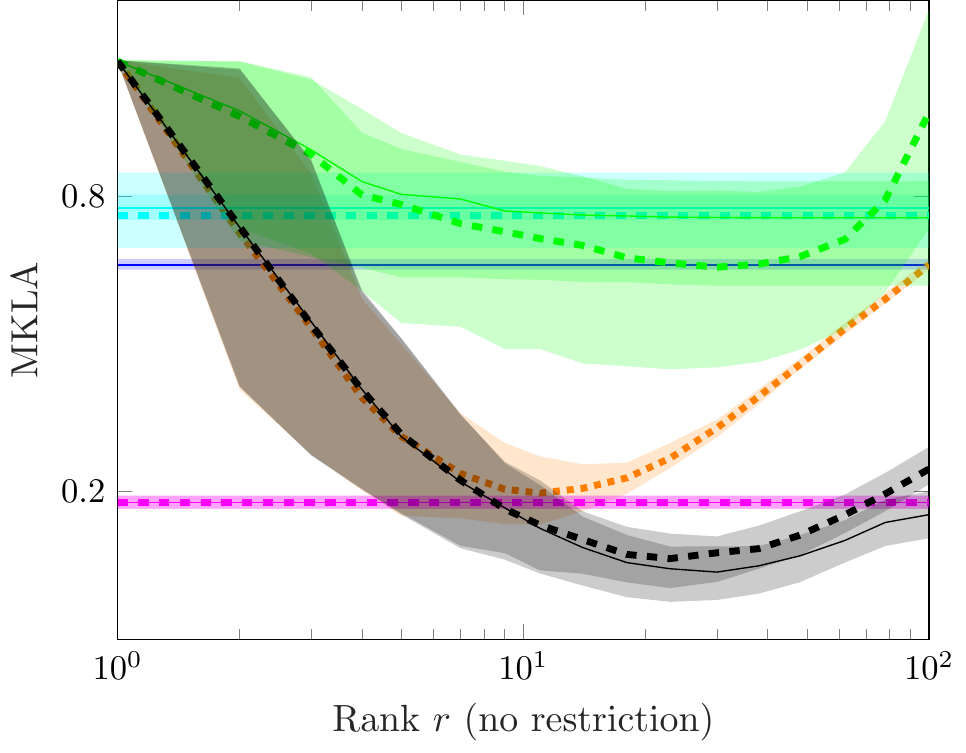}}%
\hfill{ }
\\[-0.5em]
\caption{
  (a) A single realization of corrupted version by Poisson noise with zoom  on a $100 \times 200$ noise-free matrix
  (b,c,d,e) Oracle soft-thresholding $ \bX_{\mathrm{soft}}$
  and data-driven soft-thresholding $\hat {\bX}_{\mathrm{soft}}$
  respectively for $\SE$, $\PURE$, $\KLA$ and $\PUKLA$.
  (f) PCA $\hat{\bX}^{r_{\max}}$ with full rank approximation  i.e.\ $r_{\max} = \min(n, m)$.
  (g,h,i,j) Oracle full rank approximation
  $\bX^{r_{\max}}_w$, and data-driven full rank estimation $\hat \bX^{r_{\max}}_{w}$
  respectively for $\SE$, $\PURE$, $\KLA$ and $\PUKLA$.
  (k) Their corresponding singular values averaged over $200$ noise realizations.
  (l,m) $\NMSE$ averaged over $200$ noise realizations
  as a function of the rank $r$
  with and without using the active set.
  (n,o) Same but with respect to $\MKLA$.
  (Matrix entries are displayed in log-scale for better
  visual assessment.)
}
\label{fig:poisson_higher_rank}
\end{figure}

\subsubsection*{Gamma and Poisson measurements}

Let us now consider the case where the entries of $\bY_{ij} > 0$ of the data matrix $\bY$ are independently sampled from a Gamma or Poisson distribution with mean $\bX_{ij} > 0$.  We again consider estimators of the form \eqref{eq:hatXw1_eps}. In this context, we compare the following spectral shrinkage estimators,
set for $\varepsilon = 10^{-6}$, as:
\begin{description}
\item[$\bullet$] PCA shrinkage
\begin{flalign*}
  \hat{\bX}^r = \sum_{k = 1}^r \max\left[\tilde \sigma_{k} \tilde{\bu}_{k} \tilde{\bv}_{k}^{t}, \varepsilon \right]
  \quad\text{with}\quad
  \hat{\sigma}_k = \tilde{\sigma}_k \1_{\left\{ k \in \tilde{s} \right\} }, &&
  &&
\end{flalign*}
\item[$\bullet$] GSURE/SUKLS/PURE/SUKLA driven soft-thresholding
\begin{flalign*}
\hat{\bX}_{\mathrm{soft}} = \sum_{k = 1}^{\min(m, n)} \max\left[\hat{\sigma}_k \tilde{\bu}_{k} \tilde{\bv}_{k}^{t}, \varepsilon \right]
\quad \text{with} \quad
\hat{\sigma}_k =
( \tilde \sigma_{k} - \lambda(\bY))_{+}, &&
  \end{flalign*}
  \item[$\bullet$] GSURE/SUKLS/PURE/SUKLA driven weighted estimator
\begin{flalign*}
\hat{\bX}_w^r = \sum_{k = 1}^r \max\left[\hat{\sigma}_k \tilde{\bu}_{k} \tilde{\bv}_{k}^{t}, \varepsilon \right]
\quad \text{with} \quad
\hat{\sigma}_k = w_k(\bY) \tilde{\sigma}_k \1_{\left\{ k \in \tilde{s} \right\} }, &&
\end{flalign*}
\end{description}
where $r \in [1, \min(n, m)]$,
and $\tilde s$ is the approximated active subset
as defined in Section \ref{sec:bulk_exp_fam}.
For the soft-thresholding,
the value $\lambda(\bY) > 0$ is obtained by a numerical solver in order to minimize
either the $\GSURE$ or the $\SUKLS$ criterion (in the Gamma case)
and either the $\PURE$ or the $\PUKLA$ criterion (in the Poisson case).
As shown in Section \ref{sec:optexp}, in the case of Gamma (resp.~Poisson) measurements,
the value of $w_{k}(\bY)$ for $k \in \tilde s$ which minimizes the $\GSURE$ (resp.~$\PURE$) or the $\SUKLS$ (resp.~$\PUKLA$),
cannot be obtained in closed form.
As an alternative, we adopt a greedy one-dimensional optimization strategy
starting from the matrix $\tilde \sigma_1 \tilde\bu_1 \tilde\bv_1^t$
and next updating the weights $w_{\ell}$ sequentially by starting $\ell=1$ to $\ell=\min(n, m)$, with the constraint that, for all $\ell \notin \tilde{s}$, the weight $w_{\ell}$  is set to zero.
To this end, we resort to one-dimensional optimization techniques
in the interval $[0, 1]$
using Matlab's command \texttt{fminbnd}.
This strategy is used for $\GSURE$, $\SUKLS$, $\PURE$ and $\PUKLA$
by evaluating them as described in Section \ref{sec:algo}.
As in the Gaussian setting, we compare this spectral estimators
with their oracle counterparts given by
\begin{align*}
&{\bX}_{\mathrm{soft}} = \sum_{k = 1}^{\min(m, n)} \max\left[\hat{\sigma}_k \tilde{\bu}_{k} \tilde{\bv}_{k}^{t}, \varepsilon \right]
\quad \text{with} \quad
\hat{\sigma}_k =
( \tilde \sigma_{k} - \lambda^{\mathrm{oracle}}(\bY))_{+}, \quad \text{and}\\
&{\bX}_w^r = \sum_{k = 1}^r \max\left[\hat{\sigma}_k \tilde{\bu}_{k} \tilde{\bv}_{k}^{t}, \varepsilon \right]
\quad \text{with} \quad
\hat{\sigma}_k = w_k^{\mathrm{oracle}}(\bY) \tilde{\sigma}_k \1_{\left\{ k \in \tilde{s} \right\} }.
\end{align*}
where ${w}_{k}^{\mathrm{oracle}}( \bX )$ and $\lambda^{\mathrm{oracle}}(\bY)$ minimizes
one of the objective $\SE_\eta$, $\KLS$, $\SE$ or $\KLA$ (non-expected risks)
over the set of matrices sharing with $\bY$ the same $r$ first left and right singular vectors, and soft-thresholding approximations respectively.
Note again that ${\bX}^r_w$ and ${\bX}_{\mathrm{soft}}$ are ideal approximations of $\bX$ that cannot be used in practice but serve as benchmarks to evaluate the performances of the data-driven estimators $\hat{\bX}^r_{w}$ and $\hat{\bX}_{\mathrm{soft}}$.

The results for the Gamma noise are reported
on Figure \ref{fig:gamma_higher_rank}.
As in the Gaussian setting, it can be observed that $\hat{\bX}^r_{w}$ and
${\bX}^r_w$ achieve comparable performances,
as well as $\hat{\bX}_{\mathrm{soft}}$ and ${\bX}_{\mathrm{soft}}$
showing that the $\GSURE$ (resp.~$\SUKLS$)
accurately estimates the $\MSE_\eta$ (resp.~$\KLS$).
Visual inspection of the restored matrices tends to show that
the estimators driven by $\MSE_\eta$ or $\GSURE$ produce
less relevant results compared to $\KLS$ or $\SUKLS$,
as confirmed by the curves of $\NMSE$ and $\MKLS$.
Performance in terms of $\NMSE$ also illustrates that
minimizers of $\SE_\eta$ do not coincides with those of $\SE$.
As in the Gaussian setting, $\hat{\bX}^r_{w}$ and ${\bX}^r$
outperform $\hat{\bX}_{\mathrm{soft}}$, ${\bX}_{\mathrm{soft}}$ and standard PCA $\hat{\bX}^r$ provided that $r$ is large enough.
Moreover, the performance of $\hat{\bX}^r_{w}$ obtained with $\KL$ objectives
plateaus to its optimum when the rank $r$ becomes large.
Again, this allows us to choose $r = \min(n, m)$
when we do not have {\it a priori} on the true rank $r^\star$.

The results for the Poisson noise are reported
on Figure \ref{fig:poisson_higher_rank}.
The conclusions are similar to the Gaussian and Gamma cases.
Obviously, the $\NMSE$ is smaller for approximations that
minimizes $\SE$ (or $\PURE$) than for
those minimizing $\KLA$ (or $\PUKLA$).
However, visual inspection
of the obtained matrices tends to demonstrate that minimizing such
objectives might be less relevant than minimizing
$\KL$ objectives. In this setting, the performance of
$\hat \bX_w^r$ is on a par with the one of $\hat \bX_{\mathrm{soft}}$ based on $\PUKLA$.
In fact, for other choices of matrices $\bX$,
$\hat \bX_w^r$ based on $\PUKLA$ might improve, in terms of $\MKLS$,
much more on $\hat \bX_{\mathrm{soft}}$,
and might improve not as much on $\hat \bX_w^r$  based on $\PURE$.
Nevertheless, whatever $\bX$, we observed that $\hat \bX_w^r$ driven by $\PUKLA$
always reaches at least as good performance in terms of $\MKLS$ as
the best of $\hat \bX_w^r$ driven by $\SE$ and $\hat \bX_{\mathrm{soft}}$.

Fig.~\ref{fig:gamma_higher_rank}.(m), Fig.~\ref{fig:gamma_higher_rank}.(o),
Fig.~\ref{fig:poisson_higher_rank}.(m) and Fig.~\ref{fig:poisson_higher_rank}.(o) show that
when the above estimators are used without the active set
(i.e., by choosing $\tilde s = [1, \min(n, m]$),
the performance of $\hat{\bX}^r_{w}$ actually decreases when the rank $r$ becomes too large.
As in the Gaussian setting, this can be explained by the fact that the $\GSURE$,
$\SUKLS$, $\PURE$ and $\PUKLA$ suffer from estimation variance in the case of over parametrization,
hence, they cannot be used to estimate jointly a too large number of weights.
The active set $\tilde{s}$ (in the same manner as the bulk edge)
seems to provide a relevant selection of
the weights that can be jointly and robustly estimated in a data driven way.

\subsection{Signal matrix with equal singular values and increasing rank} \label{sec:rankincreasing}

We  propose now to highlight potential limitations of our approach in the situation where the rank $r^{\ast}$ of the matrix $\bX = \sum_{k=1}^{r^{\ast}} \sigma_{k} \bu_{k} \bv_{k}^{t} $ is let growing and all positive singular values $\sigma_{k}$ of $\bX$ are equal, namely
\begin{equation}
\bY = \sum_{k=1}^{r^{\ast}}  \sigma_{k}  \bu_{k} \bv_{k}^{t} + \bW \quad \text{with} \quad  \sigma_{k}  = \gamma c_{n,m}^{1/4} \mbox{ for all } 1 \leq k \leq r^{\ast}, \label{eq:modincreasingrank}
\end{equation}
where $\bu_{k} \in \R^{n} $ and $\bv_{k} \in \R^{m}$ are vectors with unit norm that are fixed, $c_{n,m} = \frac{n}{m}$ and $\bW$ is centered random matrix whose entries are iid Gaussian variables with variance $\tau^{2} = 1/m$. We again choose to fix $n = 100$ and $m = 200$, while the true rank is $r^{\ast}$ let growing from 1 to $\min(n,m)$ in the following numerical experiments. The constant $\gamma$ is chosen to be larger than 1. Hence, eq.~\eqref{eq:modincreasingrank}  corresponds to the Gaussian spiked population model in the setting where all positive singular values are equal and larger than the threshold $c_{n,m}^{1/4}$. The choice $ \sigma_{k}  = \gamma c_{n,m}^{1/4} $ with $\gamma > 1$ is motivated by the results from Proposition \ref{prop:sv}.

For a given value of the true rank $r^{\ast}$, we performed experiments involving $M=1000$ realizations from model \eqref{eq:modincreasingrank} to compare the NMSE of the  estimators  by  oracle soft-thresholding $ \bX_{\mathrm{soft}}$,  data-driven soft-thresholding $ \hat{\bX}_{\mathrm{soft}}$, PCA full rank $\hat{\bX}^{r_{\max}}$ i.e.\ $r_{\max} = \min(n, m)$, oracle full rank approximation $\bX^{r_{\max}}_w$, and data-driven full rank estimation $\hat \bX^{r_{\max}}_{w}$ and $\hat \bX^{r_{\max}}_{*}$. All these estimators have been introduced in Section \ref{sec:rankmoretwo}.

In Figure \ref{fig:gaussian_increasing_true_rank}, we report the results of numerical experiments by displaying errors bars  of the NMSE of these estimators as functions of the true rank $r^{\ast}$. For low values of the true rank ($r^{\ast} \leq 20$), the data-driven estimators $\hat \bX^{r_{\max}}_{w}$ (our approach) and $\hat \bX^{r_{\max}}_{*}$ (shrinkage rule from  \cite{GavishDonoho})  achieve the best performances that are similar in term of median value of the NMSE. However,
our approach has some limitations with respect to the performances of the estimator from  \cite{GavishDonoho} or data-driven soft-thresholding \cite{MR3105401}  in the setting where the signal matrix has equal positive singular values and when its rank is increasing. Moreover, the error bands of the NMSE for our approach becomes significantly larger than those of the other data-driven estimators when the true rank $r^{\ast}$ increases. This illustrates that SURE minimization may lead to estimators with a high variance in the case of over parametrization, that is, when there exists a large number of significant and close singular values in the signal matrix.

\begin{figure}[!t]
\centering
\subfigure[$\gamma = 6$]{\includegraphics[height=0.27\linewidth]{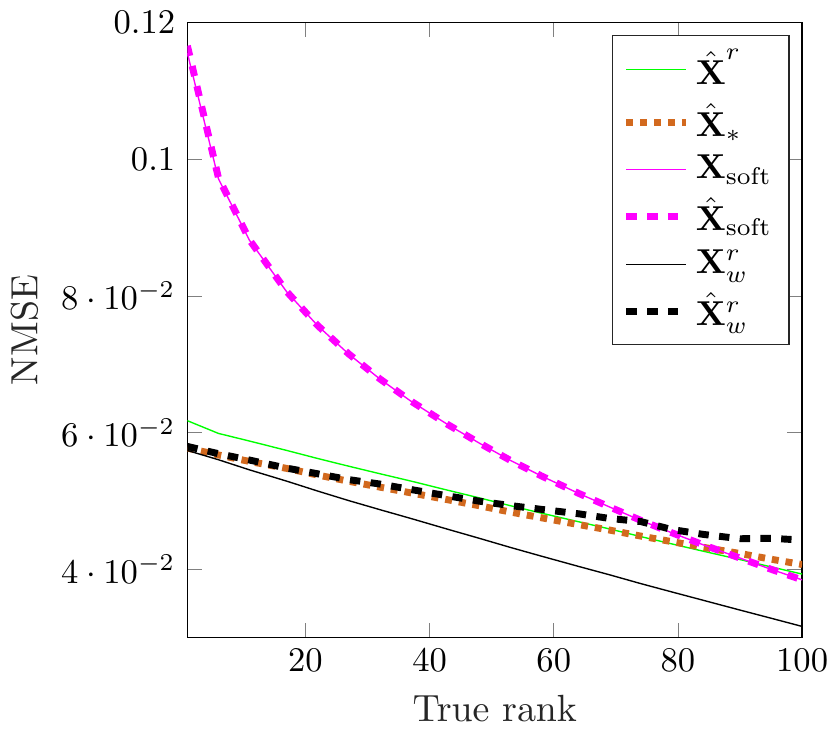}}%
\subfigure[$\gamma = 6$]{\includegraphics[height=0.27\linewidth]{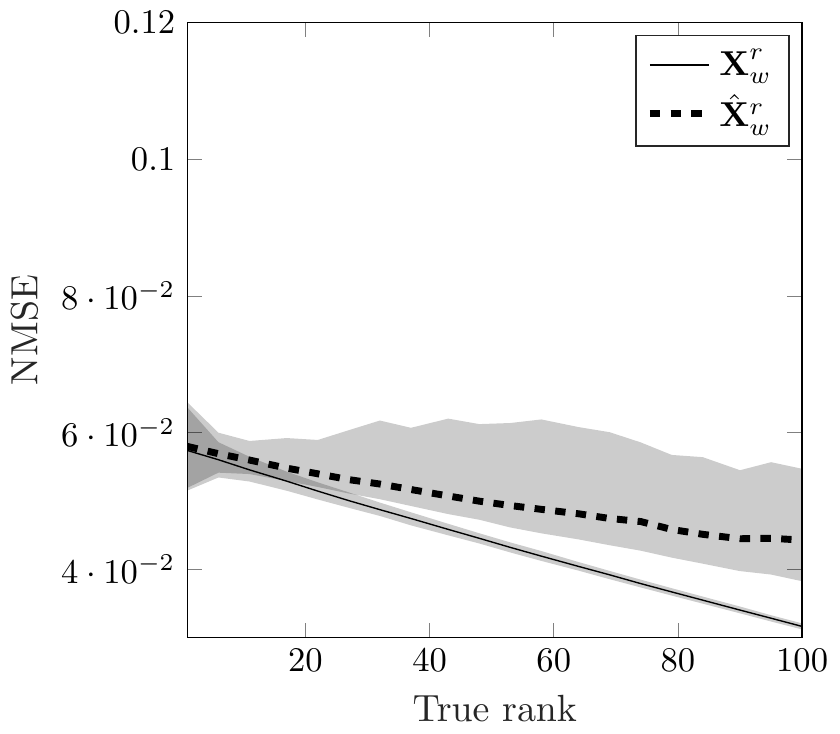}}%
\subfigure[$\gamma = 6$]{\includegraphics[height=0.27\linewidth]{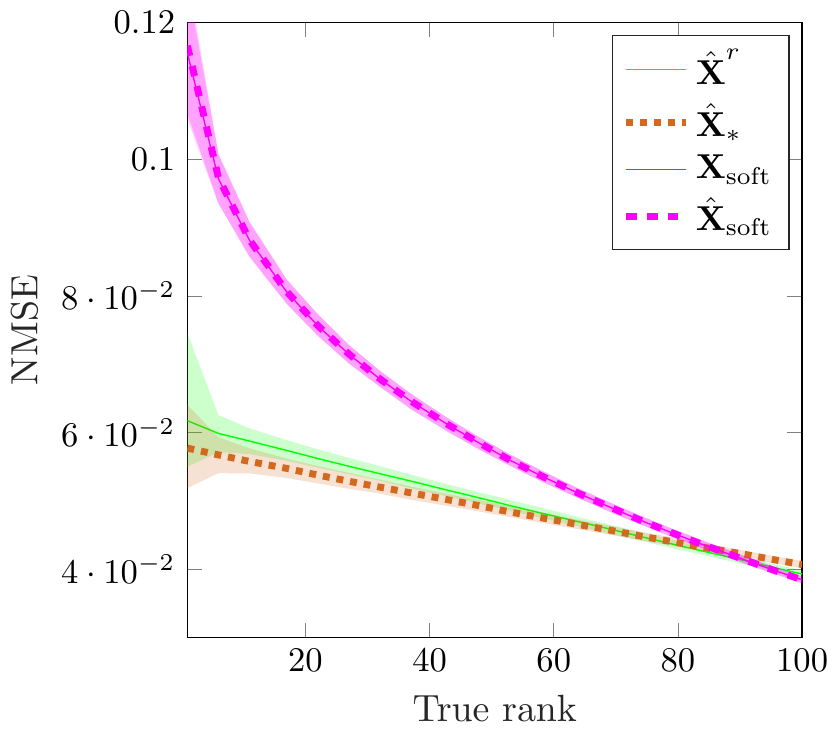}}%
\\
\subfigure[$\gamma = 4$]{\includegraphics[height=0.27\linewidth]{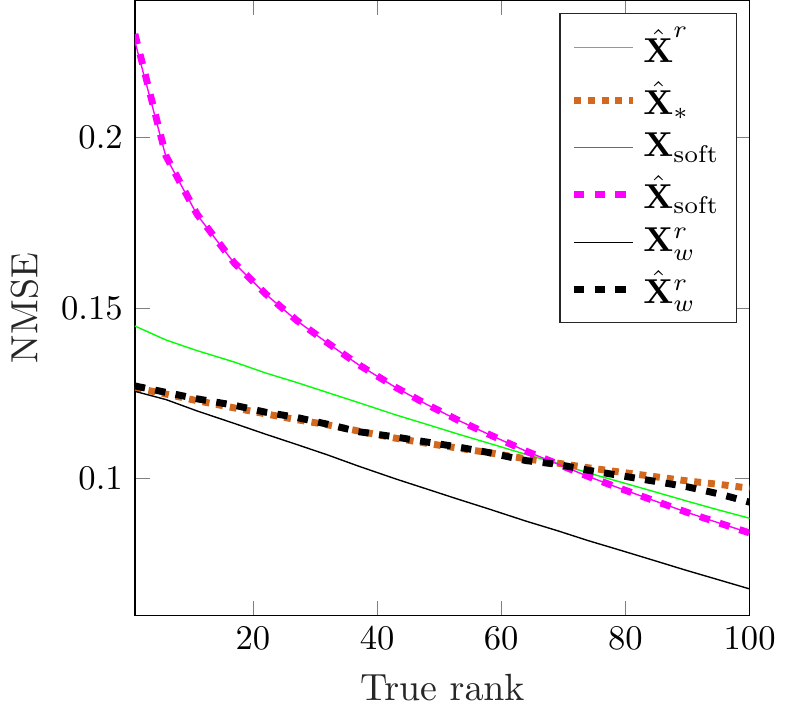}}%
\subfigure[$\gamma = 4$]{\includegraphics[height=0.27\linewidth]{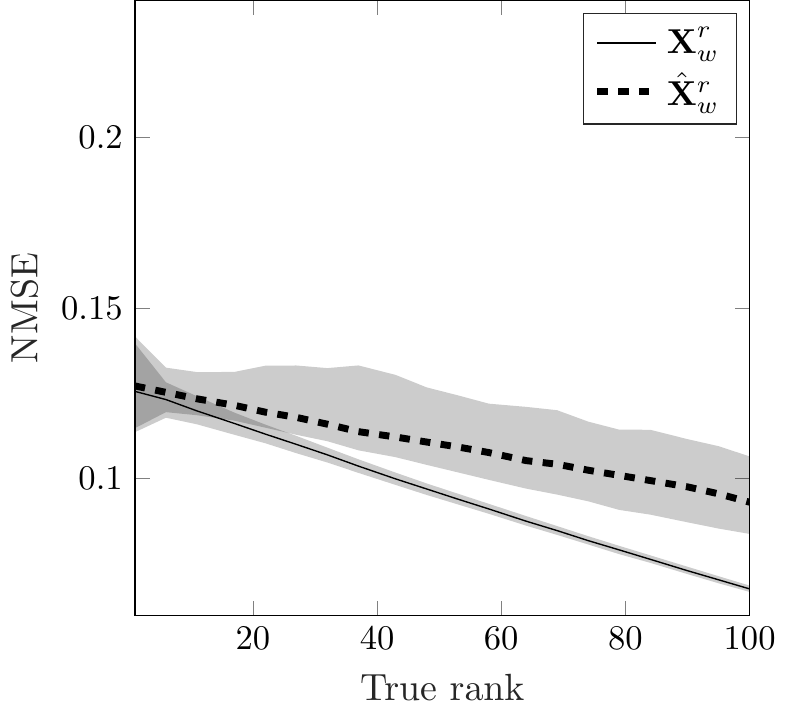}}%
\subfigure[$\gamma = 4$]{\includegraphics[height=0.27\linewidth]{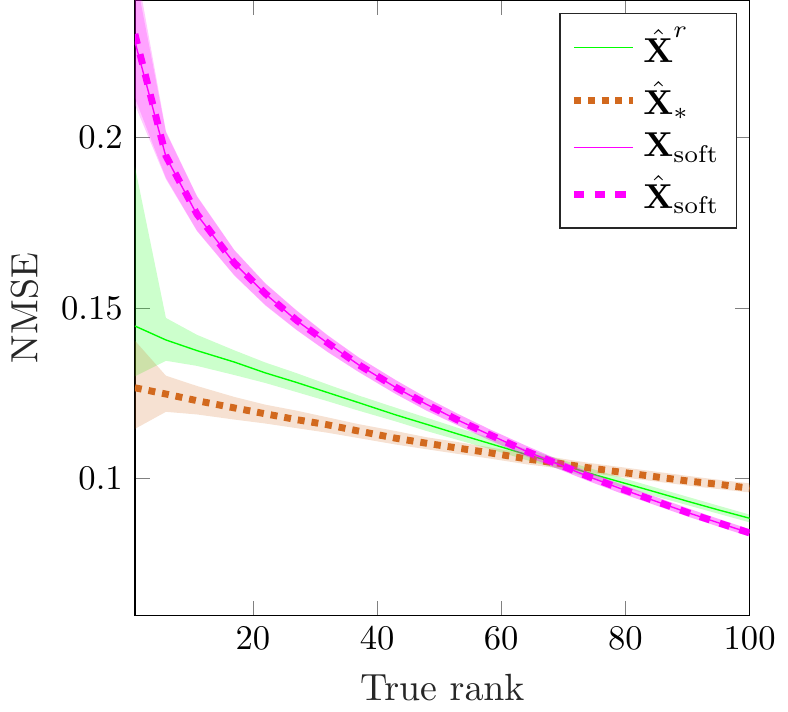}}%
\\
\subfigure[$\gamma = 2$]{\includegraphics[height=0.27\linewidth]{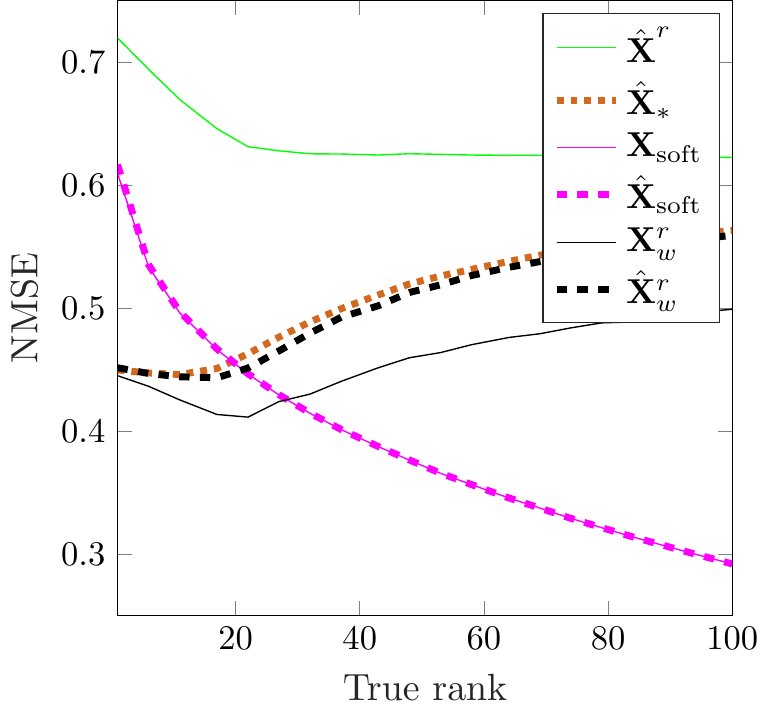}}%
\subfigure[$\gamma = 2$]{\includegraphics[height=0.27\linewidth]{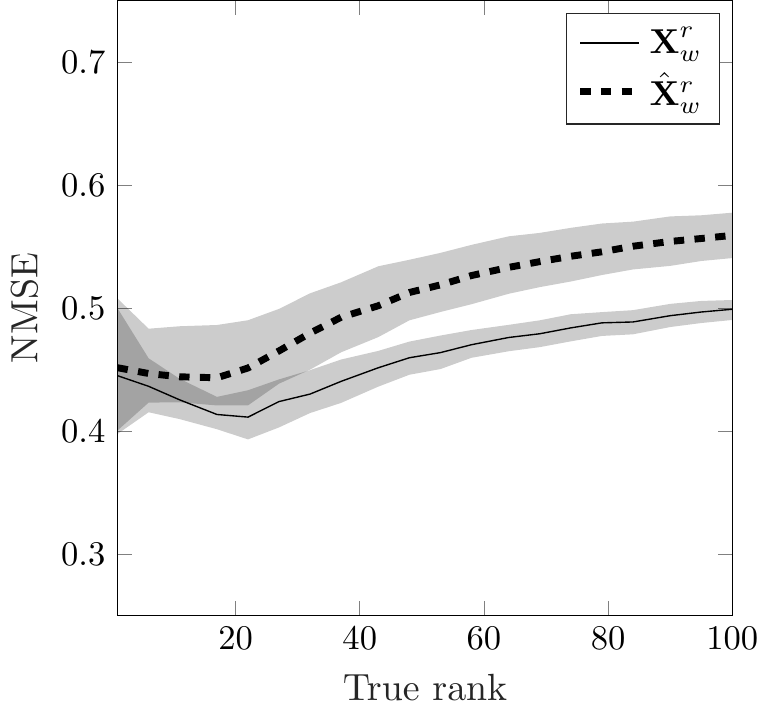}}%
\subfigure[$\gamma = 2$]{\includegraphics[height=0.27\linewidth]{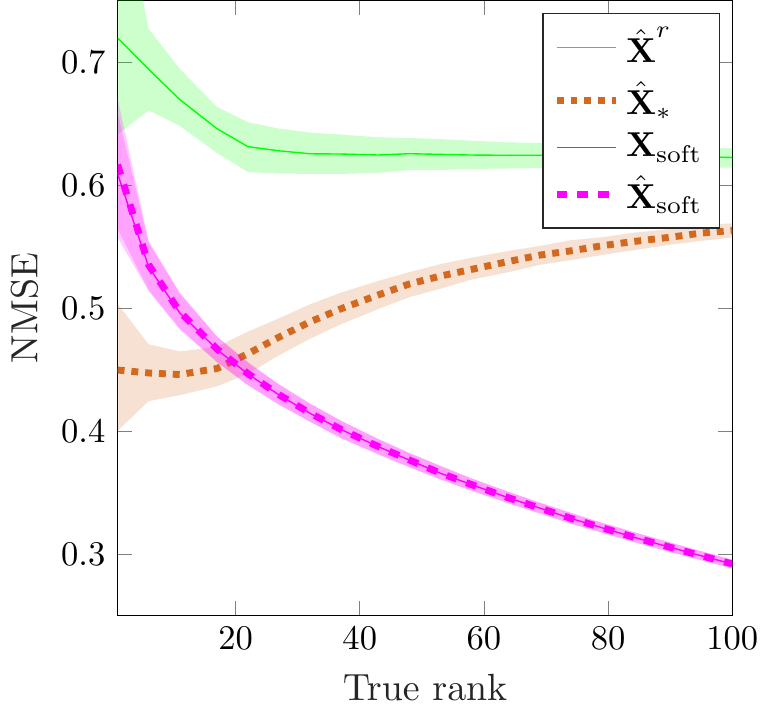}}%
\\[-0.5em]
\caption{Comparison of NMSE as a function of the true rank $r^{\ast}$ in model \eqref{eq:modincreasingrank} for different values of $\gamma$ for the estimator by  oracle soft-thresholding $ \bX_{\mathrm{soft}}$,  data-driven soft-thresholding $ \hat{\bX}_{\mathrm{soft}}$, PCA full rank $\hat{\bX}^{r_{\max}}$ i.e.\ $r_{\max} = \min(n, m)$, oracle full rank approximation $\bX^{r_{\max}}_w$, and data-driven full rank estimation $\hat \bX^{r_{\max}}_{w}$ and $\hat \bX^{r_{\max}}_{*}$. The active set set of singular values is of the form $\hat{s} =\{1,\ldots,\hat{r}\}$ where $\hat{r} = \max \{k \; ; \; \tilde{\sigma}_k > c_+^{n,m}\}$ is an estimator of the rank using knowledge of the bulk edge $c_{+} \approx c_+^{n,m}$ (a), (d), (g) Median value of the NMSE of the various estimators over $M=1000$ Gaussian noise realizations in model \eqref{eq:modincreasingrank} as a function of the true rank $r^{\ast}$. (b), (c), (e), (f), (h), (i) The grey areas represent error bands of the NMSE of data-driven and oracle estimators.
}
\label{fig:gaussian_increasing_true_rank}
\end{figure}

%

\subsection{Influence of the dimension and the signal-to-noise ratio}

In Section \ref{sec:rankincreasing}, we used simulated data consisting of a signal matrix with equal positive singular values and an increasing rank. In such a setting , it is likely that the empirical weights $w_{k}(\bY)$, used in our approach, will have a high variance due to the term  $ \sum_{\ell =1 ; \ell \neq k}^{\min(n,m)} \frac{ \tilde{\sigma}_{k}^{2}}{\tilde{\sigma}_{k}^{2} - \tilde{\sigma}_{\ell}^{2} }$ in their expression \eqref{eq:optGauss}.  However, the numerical experiments carried out in Section  \ref{sec:rankincreasing} correspond to a very specific configuration of the signal matrix (with many equal singular values and a high rank)  which is not likely to be encountered with real data. 

To conclude these numerical experiments, we finally analyze the influence of the dimension of the data and the signal-to-noise ratio on the performances of our approach and  the estimator from \cite{GavishDonoho} in a more realistic setting (with Gaussian noise). These two estimators are the ones giving the best results, and it is thus of interest to compare them with further experiments.

We use real and square signal matrices $\bX \in \R^{n \times n}$  having a relatively fast decay of their singular values, see Figure \ref{fig:mandrill} and Figure \ref{fig:cameraman}. We choose to re-size them to let $n$ varying from 20 to 250, and we define the root of the signal-to-noise ratio (RSNR) as
$$
\mathrm{RSNR} = \frac{\sqrt{\frac{1}{n^2} \sum_{i,j = 1}^{n} (\bX_{ij} - \bar{\bX})^2}}{\tau} \quad \mbox{ with }  \quad \bar{\bX} = \frac{1}{n^2} \sum_{i,j = 1}^{n} \bX_{ij}.
$$
For each value of $n$ and RSNR (ranging from 5 to 10), we performed experiments involving $M=400$ realizations from model \eqref{eq:modincreasingrank} to compare the NMSE of the  estimators  by data-driven full rank estimation $\hat \bX^{r_{\max}}_{w}$ (our approach) and $\hat \bX^{r_{\max}}_{*}$ (shrinkage rule from  \cite{GavishDonoho}) with $r_{\max} = n$. In Figure \ref{fig:mandrill} and Figure \ref{fig:cameraman}, we report the results of these numerical experiments by displaying errors bars  of the NMSE of these estimators as functions of the dimension $n$. It can be seen that our approach dominates numerically the  estimator from \cite{GavishDonoho} (for all values of $n$ and RSNR) in settings that are more likely to be encountered in practice than the simulated data used in Section \ref{sec:rankincreasing}.

\begin{figure}[!t]
\centering
\subfigure[Signal matrix $\bX$]{\includegraphics[height=0.27\linewidth]{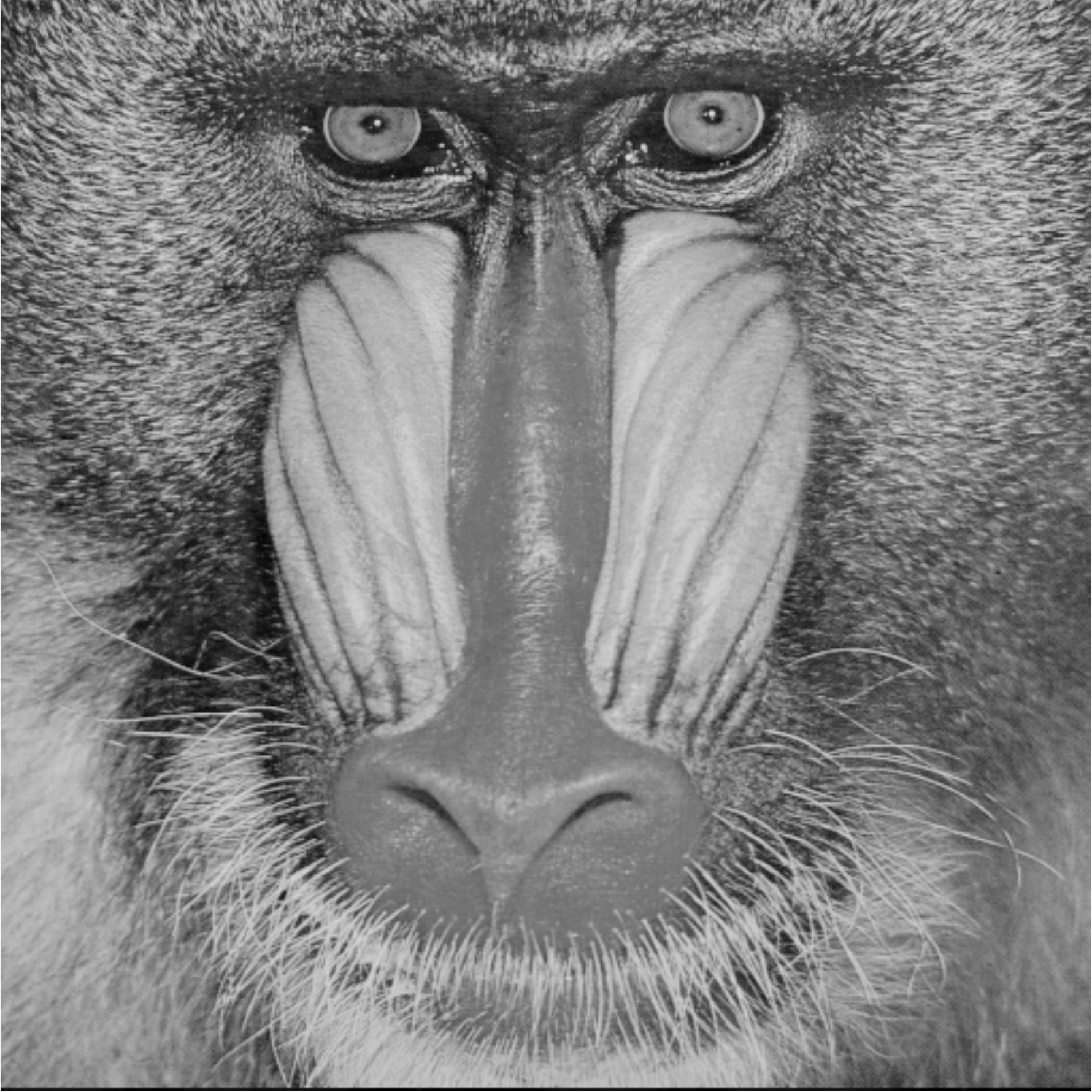}}%
\hspace{0.5cm}
\subfigure[Singular values of $\bX$]{\includegraphics[height=0.27\linewidth]{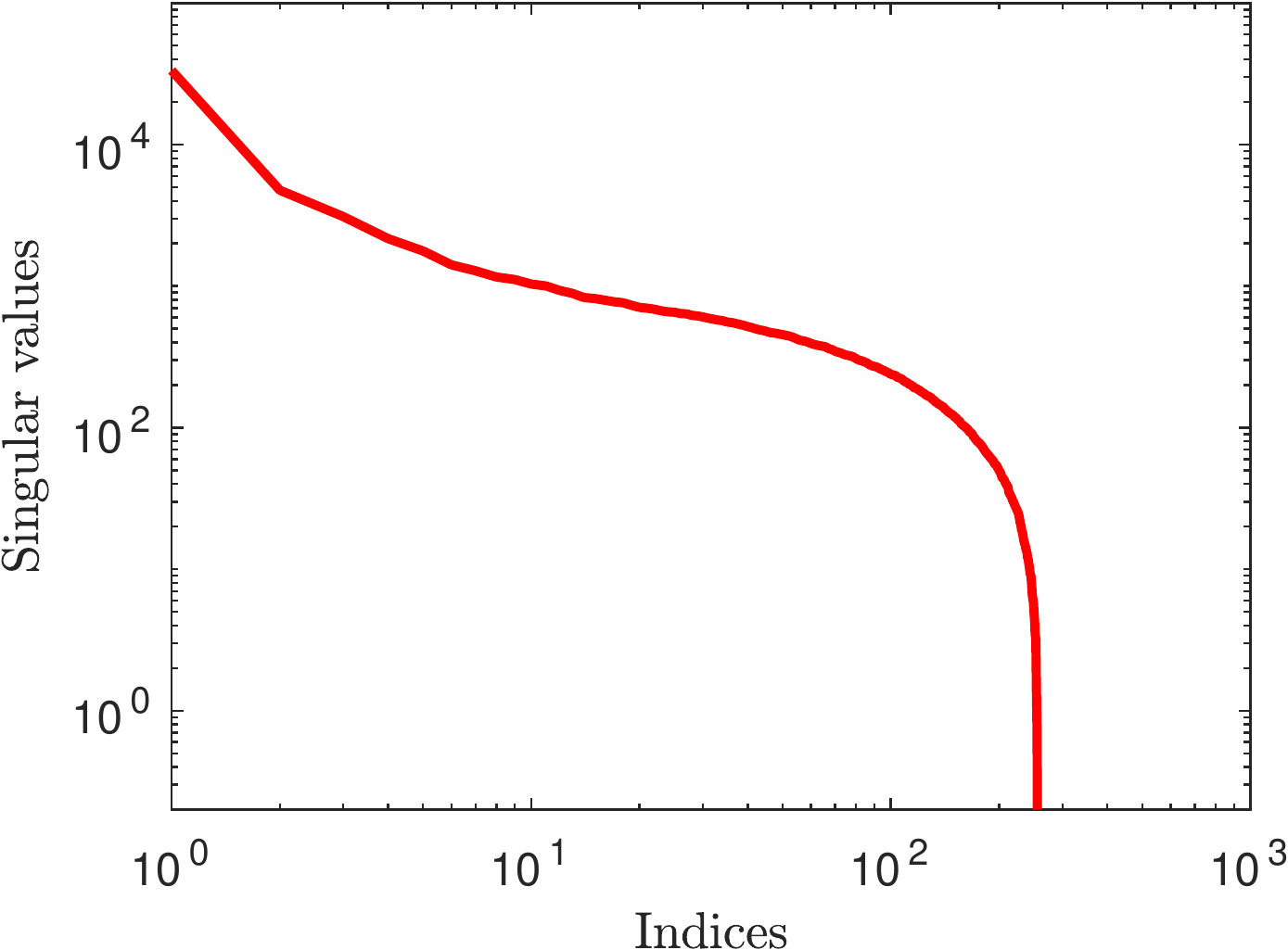}}%
\\
\subfigure[$\mathrm{RSNR} = 10$]{\includegraphics[height=0.27\linewidth]{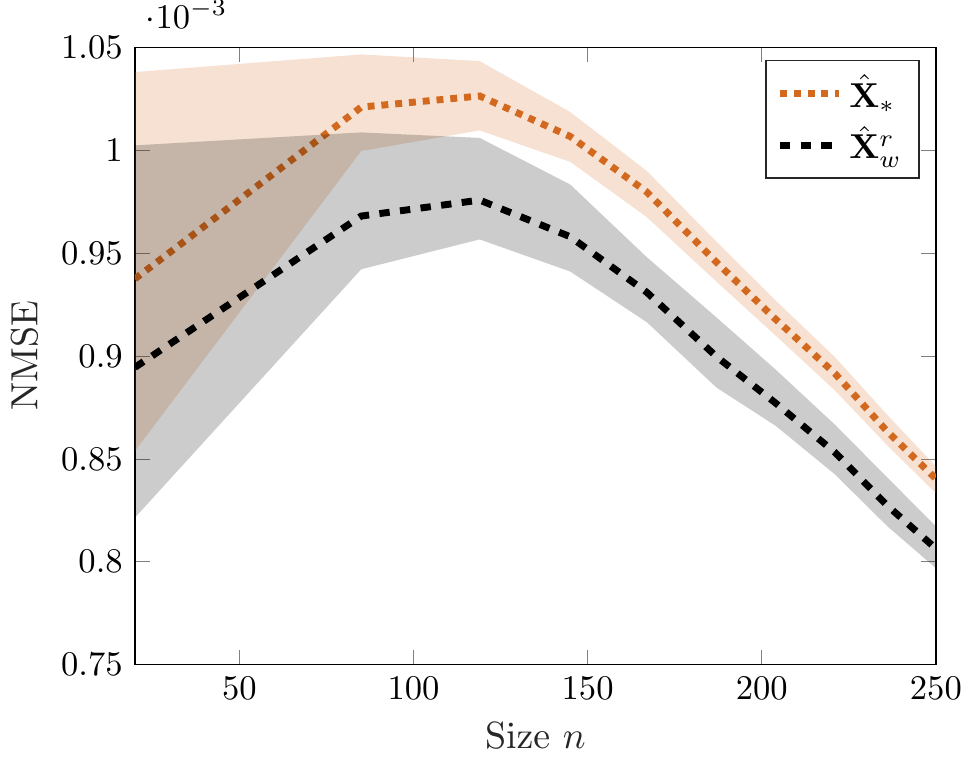}}%
\subfigure[$\mathrm{RSNR} = 7$]{\includegraphics[height=0.27\linewidth]{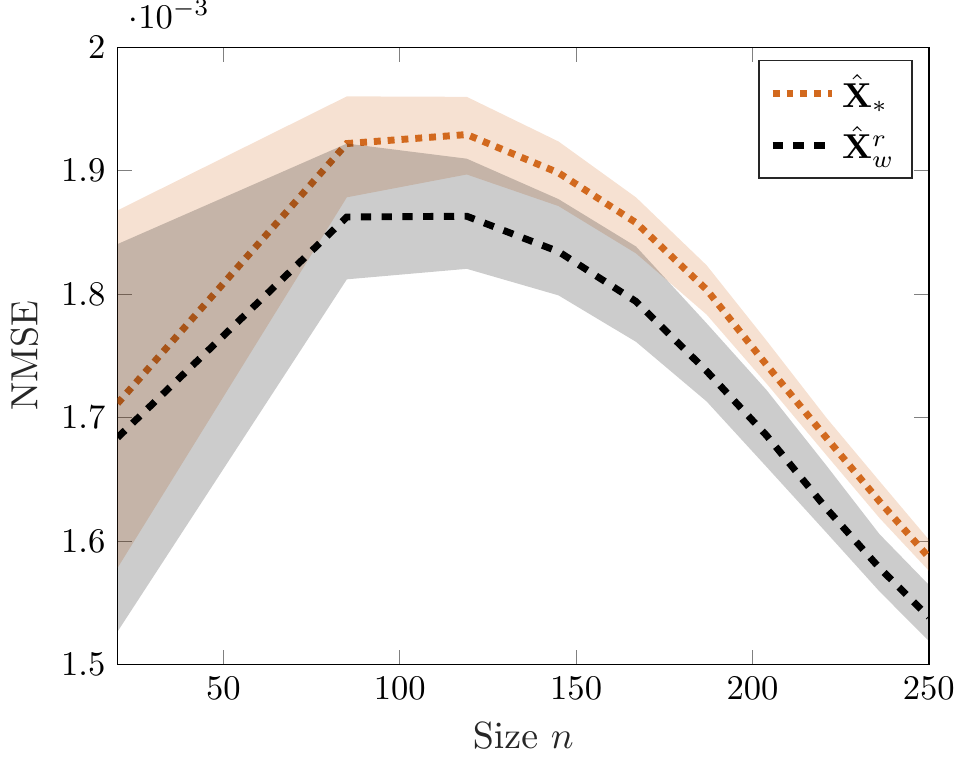}}%
\subfigure[$\mathrm{RSNR} = 5$]{\includegraphics[height=0.27\linewidth]{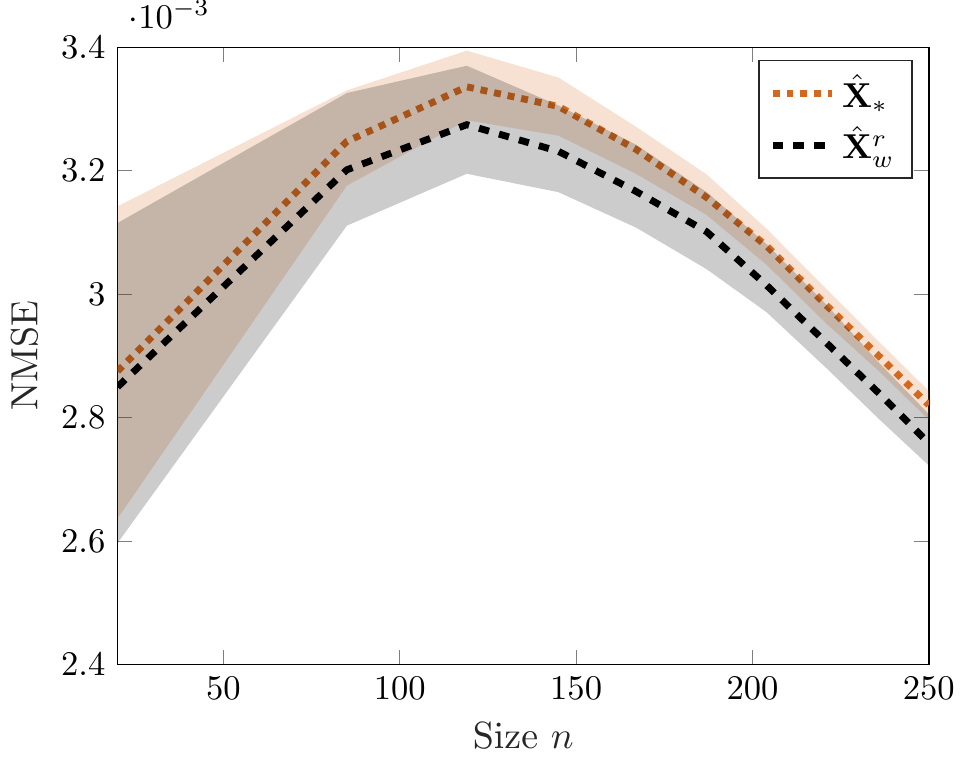}}%
\\[-0.5em]
\caption{Comparison of NMSE as a function of the dimension $n$ in model \eqref{eq:modincreasingrank} with a square matrix $\bX$ for different values of RSNR for the estimator  $\hat \bX^{r_{\max}}_{w}$ and $\hat \bX^{r_{\max}}_{*}$ with $r_{\max} = n$. The active set set of singular values is of the form $\hat{s} =\{1,\ldots,\hat{r}\}$ where $\hat{r} = \max \{k \; ; \; \tilde{\sigma}_k > c_+^{n,m}\}$ is an estimator of the rank using knowledge of the bulk edge $c_{+} \approx c_+^{n,m}$. (a) Signal matrix of size $250 \times 250$, (b) Decay of the singular values of $\bX$ in log-log scale, (c), (d), (e) Median value of the NMSE of $\hat \bX^{r_{\max}}_{w}$ and $\hat \bX^{r_{\max}}_{*}$ over $M=400$ Gaussian noise realizations in model \eqref{eq:modincreasingrank} as a function of the dimension $n$. The orange and grey areas represent error bands of the NMSE of these two estimators.}
\label{fig:mandrill}
\end{figure}

\begin{figure}[!t]
\centering
\subfigure[Signal matrix $\bX$]{\includegraphics[height=0.27\linewidth]{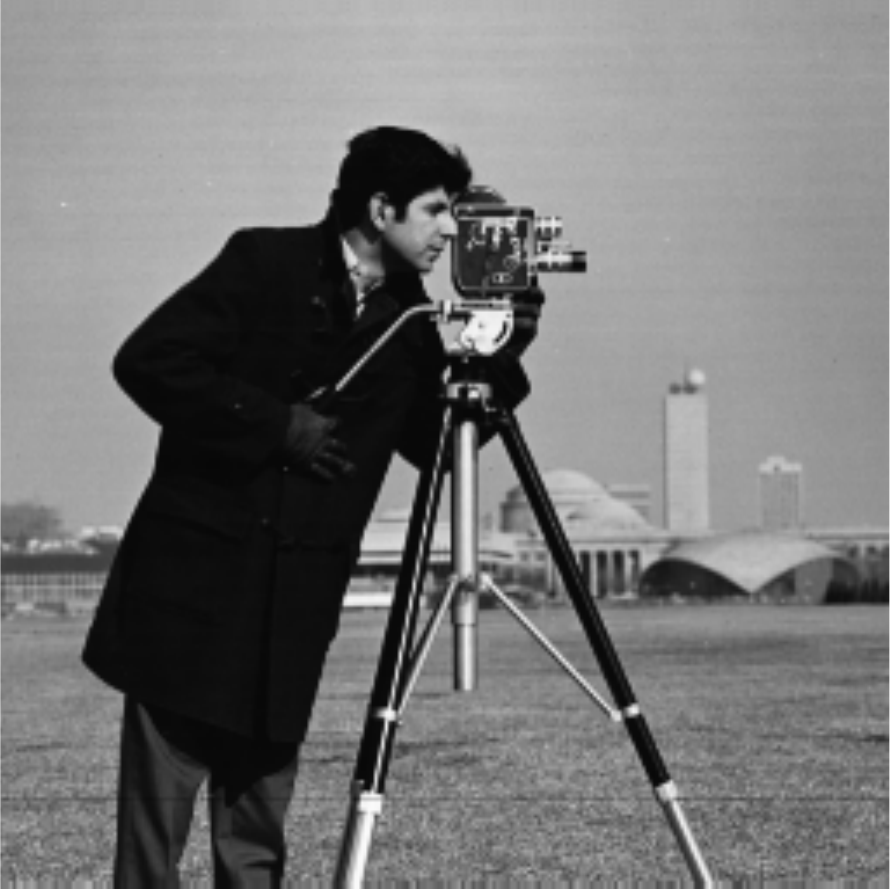}}%
\hspace{0.5cm}
\subfigure[Singular values of $\bX$]{\includegraphics[height=0.27\linewidth]{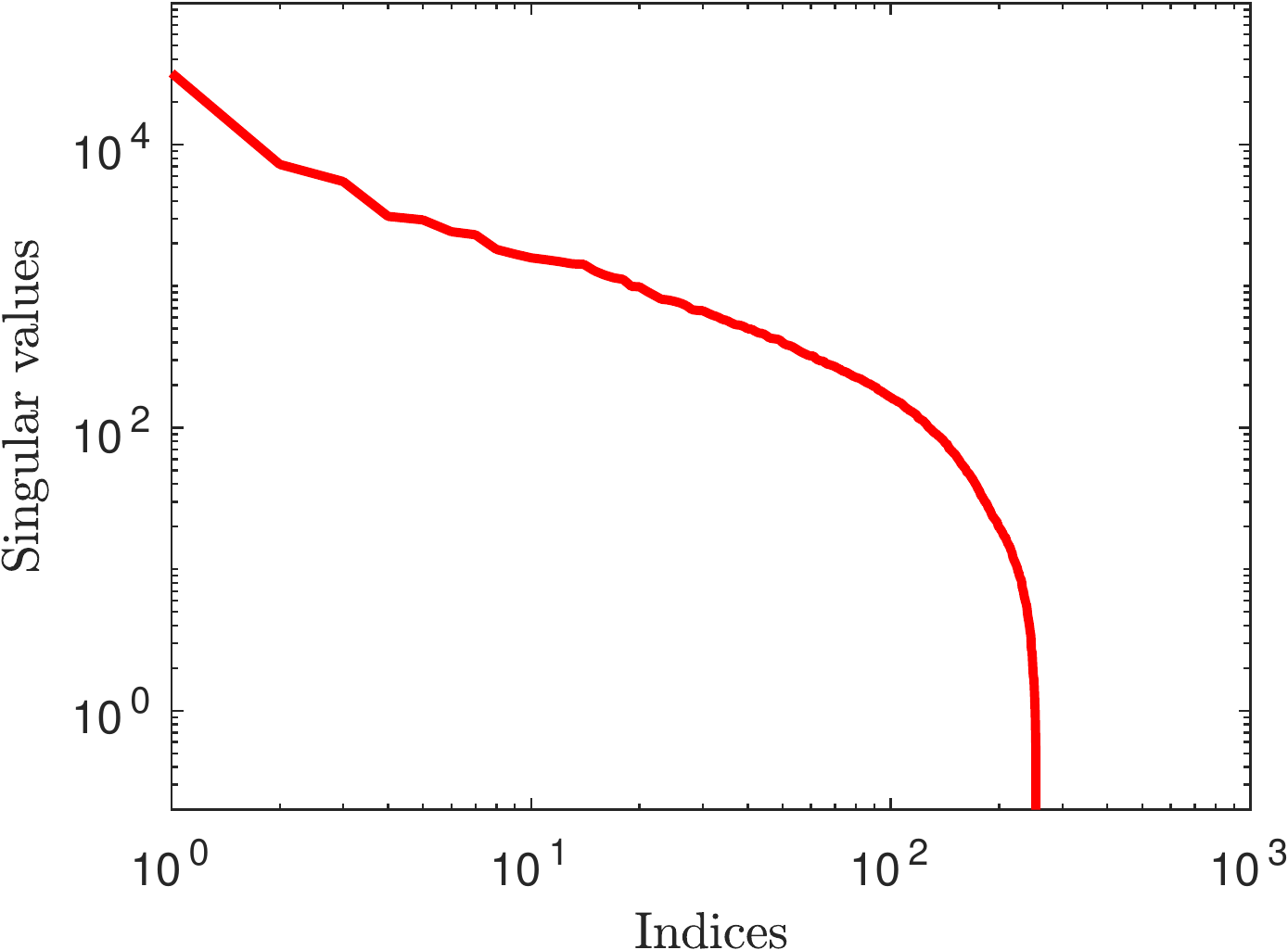}}%
\\
\subfigure[$\mathrm{RSNR} = 10$]{\includegraphics[height=0.27\linewidth]{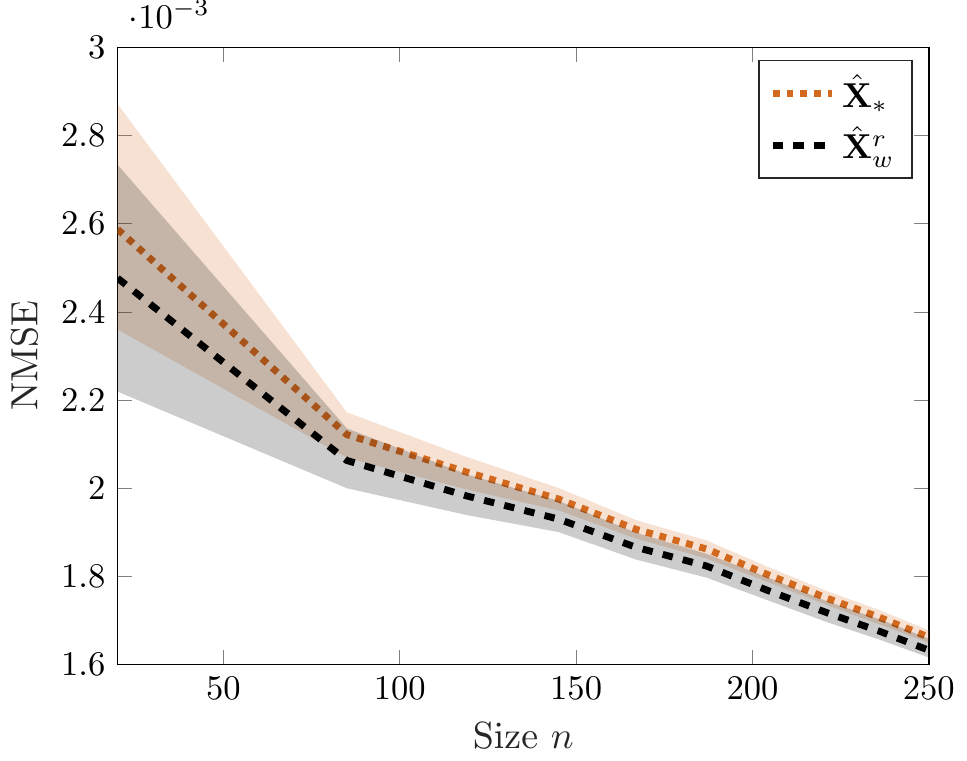}}%
\subfigure[$\mathrm{RSNR} = 7$]{\includegraphics[height=0.27\linewidth]{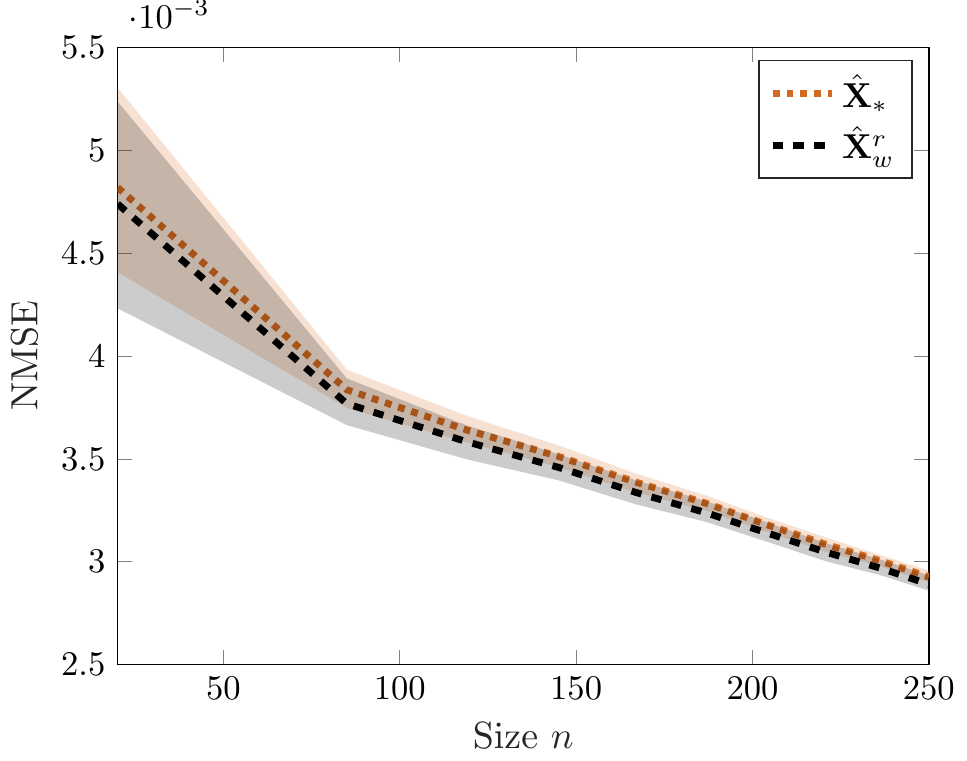}}%
\subfigure[$\mathrm{RSNR} = 5$]{\includegraphics[height=0.27\linewidth]{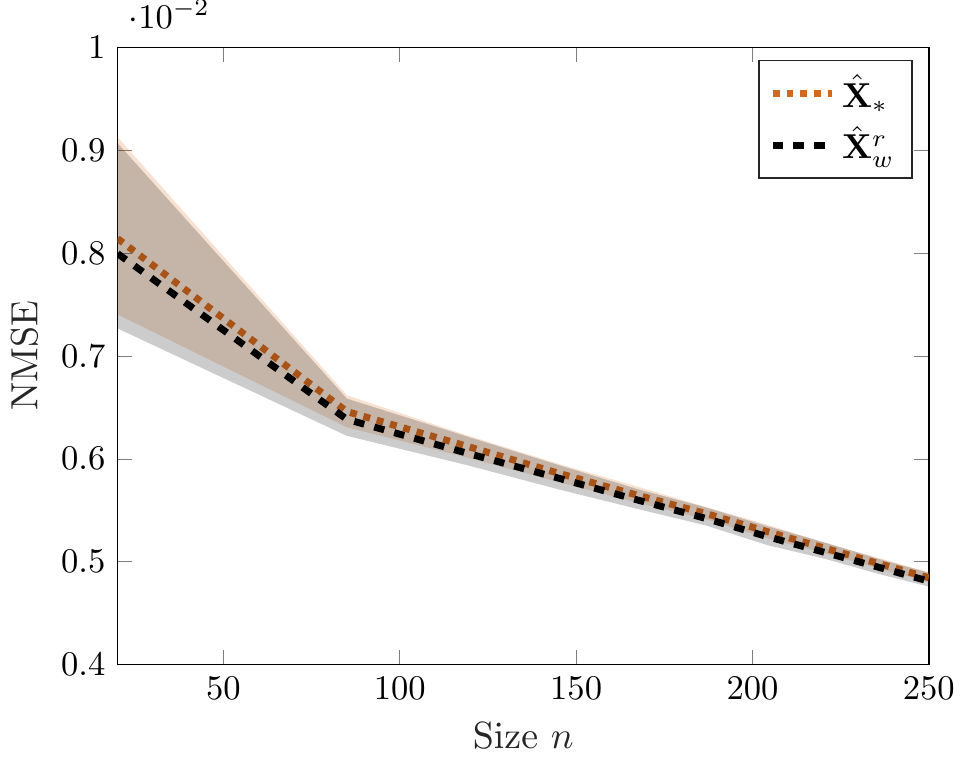}}%
\\[-0.5em]
\caption{Same as Fig.~\ref{fig:mandrill} with another signal matrix $\bX$.}
\label{fig:cameraman}
\end{figure}

\appendix
\section{Proof of the main results} \label{sec:proofs}

\subsection{Proof of Proposition \ref{prop:sti}}

Let us first introduce some notation and definitions to be used in the proof. For all $1 \leq \ell \leq n$, let $\tilde{\lambda}_{\ell}$ be the eigenvalues of ${\bY \bY^{t}}$ namely $\tilde{\lambda}_{\ell}=\tilde{\sigma}_{\ell}^{2}$. For a fixed $1 \leq k \leq r^{\ast}$ such that $\sigma_{k} > c^{1/4}$, let us introduce the complex-valued function $g_{k}$ defined by
\begin{align*}
g_{k}(z) =   \frac{1}{n}  \sum_{\ell =1 ; \ell \neq k}^{n} \frac{1}{z - \tilde{\lambda}_{\ell} } \mbox{\quad for  $z \in \mathbb C \setminus \rm{supp}(\mu_{k})$},
\end{align*}
where $\rm{supp}(\mu_{k}) = \left\{ \tilde{\lambda}_{\ell}; 1 \leq \ell \leq n, \; \ell \neq k \right\}$ is the support of the random measure
$
\mu_{k} =  \frac{1}{n}  \sum_{\ell =1 ; \ell \neq k}^{n} \delta_{\tilde{\lambda}_{\ell}}
$
on $\R_{+}$, where $\delta_{\lambda}$ denotes the Dirac measure at $\lambda$. It is clear that
\begin{align*}
g_{k}(z) = \int  \frac{1}{z - \lambda }  d \mu_{k}(\lambda).
\end{align*}
The main difficulty in the proof is to  show that, almost surely,
\begin{align*}
\lim_{n \to + \infty} g_{k}(  \tilde{\sigma}_{k}^{2} ) = \frac{1}{ \rho^2\left( \sigma_{k} \right)}  \left( 1 + \frac{1}{\sigma_{k}^{2}} \right),
\end{align*}
which is the purpose of what follows.

For a  matrix $A \in \R^{n \times m}$ (with $n \leq m$), we denote its singular values by $\sigma_{1}(A) \geq \sigma_{2}(A) \geq \ldots \geq \sigma_{n}(A) \geq 0$.  Hence, one has that $\tilde{\sigma}_{\ell} = \sigma_{\ell}(\bY)$ for all $1 \leq \ell \leq n$. Now, we recall that $\bY = \bX + \bW$ where $\bX$ is a fixed matrix of rank $r^{\ast}$ and $\bW$ is a random matrix with iid entries sampled from a Gaussian distribution with zero mean and variance $\frac{1}{m}$. The first step in the proof is to show that the random measure $\mu_{k}$ behaves asymptotically as the almost sure limit of the empirical spectral measure $\mu_{\bW \bW^{t}}$ of the Wishart matrix $\bW \bW^{t}$. By definition, the eigenvalues of $\bW \bW^{t}$ are $\lambda_{\ell}(\bW)=\sigma_{\ell}^{2}(\bW)$ for all $1 \leq \ell \leq n$ and $\mu_{\bW \bW^{t}}$ is thus defined as
\begin{align*}
\mu_{\bW \bW^{t}} =  \frac{1}{n}  \sum_{\ell =1}^{n} \delta_{\lambda_{\ell}(\bW)}.
\end{align*}
It is well know (see e.g.\ Theorem 3.6 in \cite{MR2567175}) that, once $m = m_{n} \geq n$ and $\lim_{n \to + \infty} \frac{n}{m} = c$ with $0 <c \leq 1$, then, almost surely, the empirical spectral measure $\mu_{\bW \bW^{t}}$ converges weakly to the so-called Marchenko-Pastur distribution $\mu_{MP}$ which is deterministic and has the following density
$\frac{d\mu_{MP}(\lambda)}{d\lambda} =\frac{1}{2\pi c \lambda}
\sqrt{(c^2_+ -\lambda) (\lambda-c^2_-)} ~ 1\! \!{\sf I}_{[c^2_-, c^2_+]}(\lambda)$.
We recall that such a convergence can also be characterized through the so-called Cauchy or Stieltjes transform which is defined for any probability measure $\mu$ on $\R$ as
\begin{align*}\forall z \in \C \text{\, outside the support of $\mu$}, \quad g_\mu(z)=\int \frac{1}{z - \lambda} d\mu(\lambda).\end{align*}
By eq.~(3.3.2) in \cite{MR2567175}, one obtains that, almost surely,
\begin{equation}
\lim_{n \to \infty} \int \frac{1}{z - \lambda} d \mu_{\bW \bW^{t}}(\lambda) = g_{MP}(z) \mbox{ for any }  z \in  \mathbb C \setminus \R, \label{eq:convMP}
\end{equation}
where $g_{MP}$ is the Cauchy transform of $\mu_{MP}$ and
\begin{align*}
 g_{MP}(z) =  \int \frac{1}{z - \lambda} d\mu_{MP}(\lambda) = \frac{z - (1 - c) - \sqrt{(z-(c+1))^2-4c} }{2 c z} \mbox{\quad for all $z \in \mathbb C \setminus [c_{-}^2,c_{+}^2]$}.
\end{align*}
Moreover,  by Proposition 6 in \cite{MR1971583}, the convergence \eqref{eq:convMP} is uniform over any compact subset of $\mathbb C \setminus \mathbb R$.


Then, it follows from the so-called Weyl's interlacing inequalities (see e.g.\ Theorem 3.1.2 in \cite{opac-b1086036}) that for all $1 \leq \ell \leq n$
\begin{equation}
 \sigma_{\ell + r^{\ast}}(\bW)  \leq \sigma_{\ell}(\bY) \leq  \sigma_{\ell-r^{\ast}}(\bW), \label{eq:Weyl}
\end{equation}
with the convention that  $\sigma_{k}(\bW) = - \infty$ if $k > n$ and $\sigma_{k}(\bW) = + \infty$ if $k \leq 0$.
Thanks to the results that have been recalled above on the asymptotic properties of  $\mu_{\bW \bW^{t}}$, one may use inequalities \eqref{eq:Weyl}  to prove that, almost surely, the random measure $\mu_{k}$ converges weakly to the Marchenko-Pastur distribution $\mu_{MP}$.
Under the assumptions of Proposition  \ref{prop:sti} and using Proposition \ref{prop:sv}, it can be shown that there exists $\eta_{k} > 0$
such that, almost surely and for all sufficiently large $n$
\begin{align*}
\tilde{\lambda}_{\ell} \notin K_{k} := [\rho^{2}(\sigma_{k}) - \eta_{k}, \rho^{2}(\sigma_{k}) + \eta_{k}]
\end{align*}
for any $1 \leq \ell \leq n$ with $\ell \neq k$. Now, recall that the support $\rm{supp}(\mu_{k})$ of the random measure $\mu_{k}$ is $\left\{ \tilde{\lambda}_{\ell}; 1 \leq \ell \leq n, \; \ell \neq k \right\}$, and that $\rm{supp}(\mu_{MP}) = [c_{-}^2,c_{+}^2]$. Hence, for all sufficiently large $n$, one has that
\begin{align*}
\rm{supp}(\mu_{k}) \cap K_{k} = \emptyset \quad \mbox{ and }  \quad \rm{supp}(\mu_{MP}) \cap K_{k} = \emptyset.
\end{align*}
Therefore, thanks to the weak convergence of $\mu_{k}$ to $\mu_{MP}$ and  using Ascoli's Theorem, one may prove that
\begin{equation}
\lim_{n \to \infty} \sup_{z \in K_{k}}  |g_{k}(z) -  g_{MP}(z)|  = 0  \mbox{ almost surely}. \label{eq:convuni}
\end{equation}
Thanks to our assumptions, one has that, almost surely, $\lim_{n \to + \infty} \tilde{\sigma}^{2}_{k} = \rho^{2}\left( \sigma_{k} \right)$ by Proposition \ref{prop:sv}. Hence, almost surely and for all sufficiently large $n$, one has that $\tilde{\sigma}_{k}^{2} \in K_{k}$ and so
\begin{align*}
|g_{k} (\tilde{\sigma}_{k}^{2}) - g_{MP} (\rho^{2}\left( \sigma_{k} \right))| \leq \sup_{z \in K_{k}} |g_{k} (z) - g_{MP}(z)| + |g_{MP} (\tilde{\sigma}_{k}^{2}) - g_{MP} (\rho^{2}\left( \sigma_{k} \right))|.
\end{align*}
Therefore, using the uniform convergence \eqref{eq:convuni} of $g_{k}$ to $g_{MP}$ and the continuity of $g_{MP}$ at $z = \rho^{2}\left( \sigma_{k} \right)$, one obtains that, almost surely,
\begin{align*}
\lim_{n \to + \infty} g_{k} (\tilde{\sigma}_{k}^{2}) = g_{MP} (\rho^{2}\left( \sigma_{k} \right)) = \frac{1}{ \rho^2\left( \sigma_{k} \right)} \times \frac{\rho^{2}\left( \sigma_{k} \right) - 1 + c - \sqrt{(\rho^{2}\left( \sigma_{k} \right)-(c+1))^2-4c} }{2 c}.
\end{align*}
Since  $g_{k} (\tilde{\sigma}_{k}^{2}) =  \frac{1}{n}  \sum_{\ell =1 ; \ell \neq k}^{n} \frac{1}{\tilde{\sigma}_{k}^{2} - \tilde{\sigma}_{\ell}^{2} }$, using the above equation and relation \eqref{eq:link}, it follows immediately that
$
g_{MP}(\rho^{2}\left( \sigma_{k} \right)) =  \frac{1}{ \rho^2\left( \sigma_{k} \right)}  \left( 1 + \frac{1}{\sigma_{k}^{2}} \right)
$
so that, almost surely,
\begin{align*}
\lim_{n \to + \infty} \frac{1}{n}  \sum_{\ell =1 ; \ell \neq k}^{n} \frac{\tilde{\sigma}_{k}}{\tilde{\sigma}_{k}^{2} - \tilde{\sigma}_{\ell}^{2} }
=  \lim_{n \to + \infty} \tilde{\sigma}_{k} g_{k} (\tilde{\sigma}_{k}^{2})
= \rho\left( \sigma_{k} \right) g_{MP}(\rho^{2}\left( \sigma_{k} \right))  = \frac{1}{ \rho\left( \sigma_{k} \right)}  \left( 1 + \frac{1}{\sigma_{k}^{2}} \right),
\end{align*}
which completes the proof.

\subsection{A technical result to prove SURE-like formulas} \label{app:GSURE}

We recall the key lemma needed to prove the SURE-like formulas in an exponential family in the continuous case. Similar results have already been formulated in different papers in the literature, see e.g.\  the review proposed in \cite{Deledalle}. 

\begin{lem} \label{lem:GSURE}
Let $\bY \in \R^{n \times m}$ be a random matrix whose entries $\bY_{ij}$ are independently sampled from the continuous exponential family \eqref{eq:expfamcan} in canonical form (that is the distribution of   $\bY_{ij}$ is absolutely continuous with respect to the Lebesgue measure $dy$ on $\R$). Suppose that the function $h$ is continuously differentiable on $\YY = \R$. Let $1 \leq i \leq n$ and $1 \leq j \leq m$, and denote by $F_{ij} : \R^{n \times m} \to \R$  a continuously differentiable function such that
\begin{equation}
\E \left[ \left| F_{ij}(\bY) \right| \right] < + \infty. \label{eq:condF}
\end{equation}
Then, the following relation holds
\begin{align*}
\E \left[ \btheta_{ij} F_{ij}(\bY)  \right] = - \E \left[  \frac{h'(\bY_{ij})}{h(\bY_{ij})}  F_{ij}(\bY)  + \frac{\partial F_{ij}(\bY) }{ \partial \bY_{ij}} \right].
\end{align*}
\end{lem}

\begin{proof}
Using the expression \eqref{eq:expfamcan} of the pdf of the random varibles $\bY_{ij}$, one has that
\begin{align*}
  \E \left[ \btheta_{ij}  F_{ij}(\bY) \right] = \int_{\R^{n \times m}}  F_{ij}(Y) h(y_{ij})  \btheta_{ij}  \exp \left( \btheta_{ij} y_{ij} - A(\btheta_{ij}) \right)  \;\mathrm{d}y_{ij}
  \prod_{\substack{1 \leq k \leq n\\1 \leq \ell \leq m\\(k,\ell) \neq (i,j)}}^{n} p(y_{k \ell} ; \btheta_{k \ell}) \;\mathrm{d}y_{k \ell}.
\end{align*}
where $Y = (y_{k \ell} )_{1 \leq k \leq n, 1 \leq \ell \leq m}$.
Thanks to condition \eqref{eq:condF}, it follows that
\begin{equation} \label{eq:condfinite}
 \int_{\R^{n \times m}}  F_{ij}(Y) h(y_{ij})   \exp \left( \btheta_{ij} y_{ij} - A(\btheta_{ij}) \right) \;\mathrm{d}y_{ij} \prod_{\substack{1 \leq k \leq n\\1 \leq \ell \leq m\\(k,\ell) \neq (i,j)}}^{n} p(y_{k \ell} ; \btheta_{k \ell}) \;\mathrm{d}y_{k \ell} < + \infty.
\end{equation}
Therefore, given that
$
\btheta_{ij} \exp \left( \btheta_{ij} y_{ij} - A(\btheta_{ij}) \right) = \frac{\partial \exp \left( \btheta_{ij} y_{ij} - A(\btheta_{ij}) \right)}{\partial y_{ij}},
$
an integration by part and eq.~\eqref{eq:condfinite} imply that
\begin{align*}
\E \left[ \btheta_{ij} F_{ij}(\bY) \right] = - \int_{\R^{n \times m}}  \frac{\partial  F_{ij}(Y) h(y_{ij}) }{\partial y_{ij}}    \exp \left( \btheta_{ij} y_{ij} - A(\btheta_{ij}) \right) \;\mathrm{d}y_{ij} \prod_{\substack{1 \leq k \leq n\\1 \leq \ell \leq m\\(k,\ell) \neq (i,j)}}^{n} p(y_{k \ell} ; \btheta_{k \ell}) \;\mathrm{d}y_{k \ell}.
\end{align*}
Now, since $\frac{\partial  F_{ij}(Y) h(y_{ij}) }{\partial y_{ij}} =   h'(y_{ij})  F_{ij}(Y) + \frac{\partial  F_{ij}(Y)}{\partial y_{ij}}   h(y_{ij})$, we finally obtain that
\begin{align*}
\E \left[ \btheta_{ij} F_{ij}(\bY) \right] = -\E \left[  \frac{h'(\bY_{ij})}{h(\bY_{ij})}  F_{ij}(\bY) + \frac{\partial F_{ij}(\bY) }{ \partial \bY_{ij}} \right],
\end{align*}
which completes the proof.
\end{proof}

\subsection{Proof of Proposition \ref{prop:SURE-MSE}}

We remark that
\begin{equation} \label{eq:MSE1}
\MSE(\hat{\btheta}^{f}, \btheta)  =  \sum_{i=1}^{n} \sum_{j=1}^{m}\left(   \E \left[ |\hat{\btheta}^{f}_{ij}(\bY)|^2 -2  \btheta_{ij}  \hat{\btheta}^{f}_{ij}(\bY) \right] + \btheta_{ij}^2 \right).
\end{equation}
Using Lemma \ref{lem:GSURE} with $F_{ij}(\bY) =  \hat{\btheta}^{f}_{ij}(\bY)$ and condition \eqref{eq:condGSURE}, it follows that
\begin{equation} \label{eq:MSE2}
\E \left[ \btheta_{ij}  \hat{\btheta}^{f}_{ij}(\bY) \right] = \E \left[  \frac{h'(\bY_{ij})}{h(\bY_{ij})}  \hat{\btheta}^{f}_{ij}(\bY) \right] + \E \left[ \frac{\partial \hat{\btheta}^{f}_{ij}(\bY) }{ \partial \bY_{ij}} \right].
\end{equation}
Then, by definition  \eqref{eq:expfamcan} of  the exponential family, we remark that
\begin{align*}
\E \left[ \frac{h''(\bY_{ij})}{h(\bY_{ij})} \right] =  \int_{\R}  h''(y_{ij})   \exp \left( \btheta_{ij} y_{ij} - A(\btheta_{ij}) \right) \;\mathrm{d}y_{ij}.
\end{align*}
Hence,  using  an integration by parts twice, we arrive at
\begin{equation} \label{eq:MSE3}
\E \left[ \frac{h''(\bY_{ij})}{h(\bY_{ij})} \right] =  \btheta_{ij}^2  \int_{\R}  h(y_{ij})   \exp \left( \btheta_{ij} y_{ij} - A(\btheta_{ij}) \right) \;\mathrm{d}y_{ij} = \btheta_{ij}^2.
\end{equation}
To complete the proof, it suffices to insert equalities \eqref{eq:MSE2} and  \eqref{eq:MSE3} into \eqref{eq:MSE1}.

\subsection{Proof of Proposition \ref{prop:SURE-MKLS}}

Thanks to eq.~\eqref{eq:MKLA}, one has that
\begin{equation} \label{eq:KL1}
\MKLS(\hat{\btheta}^{f}, \btheta) = \sum_{i=1}^{n} \sum_{j=1}^{m}    \E \left[   \hat{\btheta}^{f}_{ij}(\bY)  A'( \hat{\btheta}^{f}_{ij}(\bY) )   - \btheta_{ij} A'( \hat{\btheta}^{f}_{ij}(\bY))   -  A( \hat{\btheta}^{f}_{ij}(\bY)) \right] +  A(\btheta_{ij}).
\end{equation}
Using Lemma \ref{lem:GSURE} with $F_{ij}(\bY) = A'( \hat{\btheta}^{f}_{ij}(\bY))$ and condition \eqref{eq:condSUKLS}, it follows that
\begin{equation} \label{eq:KL2}
 \E \left[  \btheta_{ij} A'( \hat{\btheta}^{f}_{ij}(\bY))  \right] = -  \E \left[  \frac{h'(\bY_{ij})}{h(\bY_{ij})}  A'(  \hat{\btheta}^{f}_{ij}(\bY) )  \right] -  \E \left[ \frac{\partial \hat{\btheta}^{f}_{ij}(\bY) }{ \partial \bY_{ij}} A''( \hat{\btheta}^{f}_{ij}(\bY) )    \right].
\end{equation}
Thus, inserting equality \eqref{eq:KL2} into \eqref{eq:KL1} implies that
\begin{gather*}
\SUKLS(\hat{\btheta}^{f})  = \sum_{i=1}^{n} \sum_{j=1}^{m}  \left( \left( \hat{\btheta}^{f}_{ij}(\bY) +  \frac{h'(\bY_{ij})}{h(\bY_{ij})} \right) A'(  \hat{\btheta}^{f}_{ij}(\bY) ) - A(\hat{\btheta}^{f}_{ij}(\bY) )   \right) + \sum_{i=1}^{n} \sum_{j=1}^{m} A''( \hat{\btheta}^{f}_{ij}(\bY)) \frac{\partial \hat{\btheta}^{f}_{ij}(\bY) }{ \partial \bY_{ij}}
\end{gather*}
is an unbiased estimator of $\MKLS(\hat{\btheta}^{f}, \btheta) - \sum_{i=1}^{n} \sum_{j=1}^{m} A(\btheta_{ij})$. Now recall that  $ f_{ij}(\bY)=  \eta^{-1} \left(  \hat{\btheta}^{f}_{ij}( \bY )  \right)$ and that $A'(  \hat{\btheta}^{f}_{ij}(\bY) ) = \eta^{-1} \left(  \hat{\btheta}^{f}_{ij}( \bY )  \right)$ by Assumption \ref{hyp:link}. Therefore, $ \frac{\partial f_{ij}(\bY)  }{ \partial \bY_{ij}} = A''( \hat{\btheta}^{f}_{ij}(\bY)) \frac{\partial \hat{\btheta}^{f}_{ij}(\bY) }{ \partial \bY_{ij}}$, and thus
\begin{align*}
\SUKLS(\hat{\btheta}^{f})  = \sum_{i=1}^{n} \sum_{j=1}^{m}  \left( \left( \hat{\btheta}^{f}_{ij}(\bY) +  \frac{h'(\bY_{ij})}{h(\bY_{ij})} \right) A'(  \hat{\btheta}^{f}_{ij}(\bY) ) - A(\hat{\btheta}^{f}_{ij}(\bY) )   \right) + \sum_{i=1}^{n} \sum_{j=1}^{m} \frac{\partial f_{ij}(\bY)  }{ \partial \bY_{ij}},
\end{align*}
which completes the proof.



\subsection{Proof of Proposition \ref{prop:SURE-MKLA}}

Thanks to the expression \eqref{eq:MKLAPoisson} of the MKLA risk for data sampled from a Poisson distribution, it follows that
\begin{align*}
\MKLA(\hat{\btheta}^{f}, \btheta) +  \sum_{i=1}^{n} \sum_{j=1}^{m}   \bX_{ij} -  \bX_{ij} \log \left( \bX_{ij}  \right) =  \sum_{i=1}^{n} \sum_{j=1}^{m}  \E \left[ \hat{\bX}^{f}_{ij}   -  \bX_{ij} \log \left(  \hat{\bX}^{f}_{ij}  \right)  \right]
\end{align*}
In the  case of  Poisson data, one has that  $\exp \left( \btheta_{ij} \right) = \bX_{ij}$ and  $\frac{h(\bY_{ij} - 1)}{h(\bY_{ij})}  = \bY_{ij}$. Therefore, by applying Hudson's Lemma \ref{lem:hudson} with $F_{ij}(\bY) = \log\left( \hat{\bX}^{f}_{ij}\right)$, it follows that
\begin{align*}
\E \left[  \sum_{i=1}^{n} \sum_{j=1}^{m}  \bX_{ij} \log\left( \hat{\bX}^{f}_{ij}\right)\right] = \E \left[   \sum_{i=1}^{n} \sum_{j=1}^{m}  \bY_{ij} \log \left( f_{ij}(\bY - \be_{i} \be_{j}^{t})  \right) \right],
\end{align*}
which completes the proof.

\section{Implementation details} \label{sec:algo}

We discuss below an algorithmic approach to find data-driven spectral estimators.

First, we discuss on how to compute data-driven spectral estimators
from the expression of risk estimators.
For  $\SUKLS$ in continuous exponential families, and for $\SURE$ in the Gaussian case only,   eq.~\eqref{eq:SURE} and \eqref{eq:SUKLS}
provide respectively a closed-form solution that can be evaluated in linear time $O(nm)$.
On the contrary, the computations of $\GSURE$ (beyond the Gaussian case), $\PURE$
and $\PUKLA$, given respectively in eq.~\eqref{eq:GSUREexp}, \eqref{eq:PURE} and \eqref{eq:PUKLA}, cannot be
evaluated in reasonable time. They rely respectively on the
computation of the divergence $\div \hat{\btheta}^{f}( \bY )$,
$\sum \sum  \bY_{ij} f_{ij}(\bY - \be_{i} \be_{j}^{t})$ and
$\sum \sum \bY_{ij} \log\left( f_{ij}(\bY - \be_{i} \be_{j}^{t}) \right)$.
Without further assumptions, such quantities requires $O(n^2m^2)$ operations in general.
A standard approach for the computation of the divergence,
suggested in \cite{girard1989fast,ramani2008montecarlosure},
is to unbiasedly estimate it with Monte-Carlo simulations
by sampling the following relation
\begin{align*}
  \div \hat{\btheta}^{f}( \bY )
  &=
  \mathbb{E}_{\bdelta}\left[
  \text{tr} \left(
    \bdelta^t \frac{\partial \hat{\btheta}^{f}( \bY )}{\partial \bY} \bdelta
    \right)
    \right]
\end{align*}
at random directions $\bdelta \in \mathbb{R}^{n \times m}$ satisfying
$\mathbb{E}[\bdelta] = 0$, $\mathbb{E}[\bdelta_i \bdelta_i] = 1$
and $\mathbb{E}[\bdelta_i \bdelta_j] = 0$.
Following \cite{Deledalle}, a similar first order approximation can be
used for the other two quantities as
\begin{align*}
  \sum \sum \bY_{ij} f_{ij}(\bY - \be_{i} \be_{j}^{t})
  &\approx
  \sum \sum \bY_{ij} \left[
  f_{ij}(\bY) -
  \bdelta_{i,j} \left(\frac{\partial f( \bY )}{\partial \bY} \bdelta\right)_{i,j}
  \right], \quad \text{and}\\
  \sum \sum \bY_{ij} \log \left( f_{ij}(\bY - \be_{i} \be_{j}^{t})\right)
  &\approx
  \sum \sum \bY_{ij} \log \left[
  f_{ij}(\bY) -
  \bdelta_{i,j} \left(\frac{\partial f( \bY )}{\partial \bY} \bdelta\right)_{i,j}
  \right]
\end{align*}
where the entries of $\bdelta$ should be chosen Bernoulli distributed
with parameter $p=0.5$.
The advantage of these three approximations is
that they can be computed in linear time $O(n m)$ by making use
of the results of \cite{lewis-twice-spectral,Sun02,edelman-handout,MR3105401,deledalle2012risk}
that provide an expression for the directional derivative
given by
\begin{align}
  \frac{\partial f( \bY )}{\partial \bY} \bdelta
  =
  \tilde \bU (\bold D
  + \boldsymbol S
  + \boldsymbol A)
  \tilde \bV^t
\end{align}
where
$\tilde \bU$ and $\tilde \bV$ are the matrices whose
columns are $\tilde \bu_k$ and $\tilde \bv_k$, and
$\bold D$, $\boldsymbol S$ and $\boldsymbol A$ are
$n \times m$ matrices defined,
for all $1 \leq i \leq n$ and $1 \leq j \leq m$, as
\begin{align*}
  \bold D_{i,j} &=
  \bar \bdelta_{i,j}
  \times
  \left\{\begin{array}{ll}
  f'_i(\tilde \sigma_i) & \text{if} \quad i=j\\
  0
  & \text{otherwise},
  \end{array}\right.\\
  \boldsymbol S_{i,j} &=
  \frac{\bar \bdelta_{i,j} + \bar \bdelta_{j,i}}{2}
  \times
  \left\{\begin{array}{ll}
  0 & \text{if} \quad i=j\\
  \frac{f_i(\tilde \sigma_i)-f_j(\tilde \sigma_j)}{\tilde \sigma_i-\tilde \sigma_j}
  & \text{otherwise},
  \end{array}\right. \\
  \boldsymbol A_{i,j} &=
  \frac{\bar \bdelta_{i,j} - \bar \bdelta_{j,i}}{2}
  \times
  \left\{\begin{array}{ll}
    0 & \text{if} \quad i=j\\
    \frac{f_i(\tilde \sigma_i)+f_j(\tilde \sigma_j)}{\tilde \sigma_i+\tilde \sigma_j}
    & \text{otherwise},
  \end{array}\right.
\end{align*}
where $\tilde \sigma_k$ and $f_k(\tilde \sigma_k)$ are extended to $0$
for $k > \min(n,m)$ and
$\bar \bdelta = \tilde \bU^t \bdelta \tilde \bV \in \mathbb{R}^{n \times m}$.

\bibliographystyle{alpha}
\bibliography{GSURE_SVD_Thresh}

\end{document}